\documentclass{amsart}
\newtheorem{theorem}{Theorem}[section]
\newtheorem{lemma}[theorem]{Lemma}
\newtheorem{corollary}[theorem]{Corollary}
\newtheorem{proposition}[theorem]{Proposition}
\theoremstyle{definition}
\newtheorem{definition}[theorem]{Definition}
\newtheorem{example}[theorem]{Example}

\theoremstyle{remark}

\numberwithin{equation}{section}

\begin{document}

\title{Right coideal subalgebras in  $U_q^+({\mathfrak so}_{2n+1})$}
\author{V.K. Kharchenko}
\address{Universidad Nacioanal Aut\'onoma de M\'exico, Facultad de Estudios Superiores 
Cuautitl\'an, Primero de Mayo s/n, Campo 1, CIT, Cuautitlan Izcalli, Edstado de M\'exico, 54768, MEXICO}
\email{vlad@servidor.unam.mx}
\thanks{The author was supported by PAPIIT IN 108306-3, UNAM, and PACIVE CONS-304, FES-C UNAM, M\'exico.}



\begin{abstract}
We give a complete classification of right coideal subalgebras that contain 
all group-like elements for  the quantum group $U_q^+(\frak{so}_{2n+1}),$ 
provided that $q$ is not a root of 1. If $q$ has 
a finite multiplicative order $t>4,$ this classification remains valid for homogeneous right coideal 
subalgebras of the small Lusztig quantum group 
$u_q^+(\frak{ so}_{2n+1}).$ As a consequence, we determine that the total number of
right coideal subalgebras that contain the coradical equals $(2n)!!,$ the order of the
Weyl group defined by the root system of type $B_n.$
\end{abstract}
\maketitle
\markboth{V.K. Kharchenko}{Right coideal subalgebras}

 \section{Introduction}

{\bf Journal of the European Mathematical Society, in press}

In the present  paper, we continue the classification of right coideal subalgebras in 
quantized enveloping algebras started in \cite{KL}. We
offer a complete classification of right coideal subalgebras that contain all group-like elements for the 
multiparameter version of the quantum group $U_q^+(\frak{so}_{2n+1}),$ provided that the main 
parameter $q$ is not a root of 1. If $q$ has 
a finite multiplicative order $t>4,$ this classification remains valid for homogeneous right coideal 
subalgebras of the multiparameter version of the small Lusztig quantum group 
$u_q^+(\frak{ so}_{2n+1}).$ The main result of the paper establishes a bijection between all
sequences $(\theta_1, \theta_2, \ldots ,\theta_n)$ such that $0\leq \theta_k\leq 2n-2k+1,$ $1\leq k\leq n$
and the set of all (homogeneous if $q^t=1,$ $t>4$) right coideal subalgebras of $U_q^+(\frak{ so}_{2n+1}),$
$q^t\neq 1$ (respectively of $u_q^+(\frak{ so}_{2n+1})$) that contain the coradical. (Recall that in a pointed Hopf algebra, the group-like 
elements span the coradical.) In particular  there are $(2n)!!$ different
right coideal subalgebras that contain the coradical.  Interestingly this number coincides with 
the order of the Weyl group for the 
root system of type $B_n.$  In \cite {KL} we proved that the number of different right coideal 
subalgebras  that contain the coradical in $U_q^+(\frak{ sl}_{n+1})$ equals $(n+1)!,$ the order
of the Weyl group for the root system of type $A_n.$ 
Recently B. Pogorelsky \cite{Pog} proved that the 
quantum Borel algebra $U_q^+(\frak g)$ for the simple Lie algebra of type $G_2$ has 
$12$ different  right coideal subalgebras over the coradical. This number also coincides with the
order of the Weyl group of type $G_2.$ Although there is no theoretical explanation why the Weyl
group appears in these results, we state the following general hypothesis.

\smallskip
\noindent
{\bf Conjecture.} {\it Let $\frak g$ be a simple Lie algebra defined by a finite root system $R.$
The number of different  right coideal subalgebras that contain the coradical
in a quantum Borel algebra $U_q^+(\frak g)$ equals the order of the Weyl group defined by 
the root system $R,$ provided that $q$ is not a root of 1.}

In Section 2, following  \cite{KL}, we introduce main concepts and formulate some known results 
that are useful for classification. 
In the third section we prove auxiliary relations 
in a multiparameter version of $U_q^+(\frak{ so}_{2n+1}).$
In the fourth section we note that the Weyl basis 
$$
\{u[k,m]\stackrel{df}{=}[\ldots [x_k,x_{k+1}],\ldots , x_m] \, | \, 1\leq k\leq m\leq 2n-k, \ \ 
x_{n+r}\stackrel{df}{=}x_{n-r+1}\}
$$
of the Borel subalgebra 
$\frak{ so}_{2n+1}^+$ with skew brackets, $[u,v]=uv-\chi^u (g_v)vu,$
in place of the Lie operation is a set of
PBW-generators for $U_q^+(\frak{ so}_{2n+1})$ and for $u_q^+(\frak{ so}_{2n+1}).$
By means of the shuffle representation, in Theorem \ref{cos} we prove the explicit formula for the coproduct
of these PBW-generators, which is the key result for the following further considerations:
$$
\Delta (u[k,m])=u[k,m]\otimes 1+g_{km}\otimes u[k,m]
+\sum _{i=k}^{m-1}\tau _i(1-q^{-2})g_{ki}u[i+1,m]\otimes u[k,i],
$$
where $\tau _i=1$ with only one exception being $\tau _n=q,$
while $g_{ki}$ are suitable group-like elements.
Interestingly, this coproduct formula 
differs from that in $U_{q^2}^+({\mathfrak sl}_{2n+1})$
by just one term (see formula (3.3) in \cite{KA}).

In Section 5 we show that each homogeneous right coideal subalgebra
in  $U_q^+(\frak{ so}_{2n+1})$ or in $u_q^+(\frak{ so}_{2n+1})$ 
has PBW-generators of a special form, $\Phi^{S}(k,m),$ 
where  $S$ is a set of integer numbers from the interval $[1,2n].$
 The polynomial $\Phi^{S}(k,m)$ is defined by induction on the number $r$ of elements
in the set $S\cap [k,m-1]=\{ s_1,s_2,\ldots ,s_r\} ,$ $k\leq s_1<s_2<\ldots <s_r<m$
as follows:
$$
\Phi^{S}(k,m)=u[k,m]-(1-q^{-2})\sum_{i=1}^r \alpha _{km}^{s_i} \, 
\Phi^{S}(1+s_i,m)u[k,s_i],
$$
where $\alpha _{km}^{s}$  are scalars, $\alpha _{km}^{s}=\tau _{s}p(u(1+s,m),u(k,s))^{-1}.$
This implies that the set of all (homogeneous) 
right coideal subalgebras that contain the coradical is finite (Corollary \ref{vd}).

In Sections 6 and 7 we single out special sets $S$, called $(k,m)$-regular sets. 
In Proposition \ref{dec} we establish a kind of duality for elements 
$\Phi^{S}(k,m)$ with regular $S$, which provides
a powerful tool for investigation of PBW-generators for the right coideal subalgebras.

In Section 8 we define a root sequence $r({\bf U})=(\theta_1, \theta_2, \ldots ,\theta_n)$
 in the following way. 
The number $\theta _i$ is 
the maximal $m$ such that for some $S$ the value of $\Phi^{S}(i,m)$ 
belongs to {\bf U}, while the degree 
$x_i+x_{i+1}+\ldots +x_m $ of 
$\Phi^{S}(i,m)$ is not a sum of other nonzero degrees of  elements from {\bf U}.
In Theorem \ref{teor} we show that the root sequence uniquely defines the 
right coideal subalgebra {\bf U} that contains the coradical. 

In Section 9 we consider some important examples. Among them is the 
right coideal subalgebra generated by $\Phi^{S}(k,m)$ with regular $S$.
We also analyze in detail the simplest (but not  trivial, \cite{BDR}) case $n=2.$

In Section 10 we associate a right coideal subalgebra ${\bf U}_{\theta }$ to 
each sequence of integer  numbers $\theta =(\theta _1, \theta _2, \ldots , \theta _n),$
$0 \leq \theta _i\leq 2n-2i+1,$ so that $r({\bf U}_{\theta })=\theta .$  First,
by downward induction on $k$ we define sets 
$$
R_k\subseteq [k,2n-k ],\ \ T_k\subseteq [k,2n-k+1], \ 1\leq k\leq 2n
$$
 as follows. 
For $k>n$ we put $R_k=T_k=\emptyset .$
Suppose that 
$R_i,$ $T_i,$ $k<i\leq 2n$ are already defined. Denote by {\bf P} the following binary predicate
on the set of all ordered pairs $i\leq j$:
$$
{\bf P}(i,j)\rightleftharpoons j\in T_{i}\vee 2n-i+1\in T_{2n-j+1}.
$$
If $\theta _k=0,$ then we set 
$R_k=T_k=\emptyset .$ If $\theta _k\neq 0,$ then by definition $R_k$ contains 
$\tilde{\theta }_k=k+\theta _k-1$ and all $m$ satisfying the following three properties
$$
\begin{matrix}\smallskip
a)\ k\leq m<\tilde{\theta }_k; \hfill \cr \smallskip
b)\ \neg \hbox{\bf P}\, (m+1, \tilde{\theta }_k); \hfill \cr
c) \  \forall r (k\leq r<m)\ \  \hbox{\bf P}\, (r+1, m)\Longleftrightarrow 
\hbox{\bf P}\, (r+1, \tilde{\theta }_k). \hfill
\end{matrix}
$$
Further, we define an auxiliary set 
$$
T_k^{\prime }=R_k\cup \bigcup_{s\in R_k} \{ a\, |\, s<a\leq 2n-k, \ {\bf P}(s+1,a)\} ,
$$
and  put
$$
T_k=\left\{ \begin{matrix} T_{k}^{\prime }, \hfill 
& \hbox{ if } (2n-R_k)\cap T_k^{\prime }=\emptyset ;\hfill \cr 
T_{k}^{\prime }\cup \{ 2n-k+1\}, \hfill & \hbox{ otherwise.}\hfill 
\end{matrix}
\right. 
$$
Next, the subalgebra {\bf U}$_{\theta }$ by definition is generated over {\bf k}$[G]$ by values in
 $U_q({\mathfrak so}_{2n+1})$
or in $u_q({\mathfrak so}_{2n+1})$ of the polynomials
$\Phi ^{T_k}(k,m),$ $1\leq k\leq m$ with $m\in R_k.$

Theorem \ref{teor} and Theorem \ref{su4} show that
all right coideal subalgebras over the coradical
have the form ${\bf U}_{\theta }.$

In Section 11 we restate the main result in a slightly more general form. 
We consider homogeneous right coideal subalgebras {\bf U} 
such that the intersection $\Omega ={\bf U}\cap G$
with the group $G$  of all group-like elements is a subgroup. In this case 
${\bf U}={\bf k }[\Omega ]{\bf U}_{\theta }^{1},$ where $ {\bf U}_{\theta }^{1}$
is a subalgebra generated by $g^{-1}_aa$ when $a=\Phi ^{T_k}(k,m)$ runs 
through the above described generators of {\bf U}$_{\theta }.$

The present paper extends \cite{KL} by similar methods and in a parallel way.
However technically it is much more complicated. The proof of the explicit formula 
for comultiplication (Theorem \ref{cos}) essentially depends on the shuffle representation 
given in Proposition \ref{shu}, while the same formula for case $A_n$ 
is proved by a simple induction \cite{KA}. The elements $\Phi^{S}(k,m)$
that naturally appear as PBW-generators for right coideal subalgebras 
do not satisfy all necessary properties for further development. 
Therefore in Section 7 we have to introduce and investigate the elements $\Phi^{S}(k,m)$ 
with so called $(k,m)$-{\it regular} sets $S$. In Proposition \ref{dec} we establish a powerful 
duality for such elements. Interestingly, as a consequence of the classification, 
we prove  that every right coideal subalgebra over the coradical is generated as an algebra by
elements $\Phi^{S}(k,m)$  with $(k,m)$-{\it regular} sets $S$ (Corollary \ref{rug}). 
The construction of {\bf U}$_{\theta }$ is more complicated and it has an important  new element, 
a binary predicate defined on the ordered pairs of indices. 
In \cite{KL} we relatively easy find a differential subspace generated by  $\Psi^{S}(k,m)$
by virtue of the fact that this element is linear in each variable that it depends on.
However the elements $\Phi^{S}(k,m)$ that appeared in the present work are not linear in some variables.
In this connection we fail to find their partial derivatives in an appropriate form. Instead, in Theorem \ref{iex2},
using the root technique developed in Section 8, we find algebra generators of the right coideal 
subalgebra generated  by $\Phi^{S}(k,m)$ with $(k,m)$-regular set $S$.

\section{Preliminaries}

\noindent
{\bf PBW-generators}.
Let $A$ be an algebra over a field {\bf k} and $B$ its subalgebra
with a fixed basis $\{ g_j\, |\,  j\in J\} .$ A linearly ordered subset $V\subseteq A$ is said to be
a {\it set of PBW-generators of $A$ over $B$} if there exists 
a function $h:V\rightarrow {\bf Z}^+\cup {\infty },$
called the {\it height function}, such that the set of all products
\begin{equation}
g_jv_1^{n_1}v_2^{n_2}\cdots v_k^{n_k}, 
\label{pbge}
\end{equation}
where $j\in J,\ \ v_1<v_2<\ldots <v_k\in V,\ \ n_i<h(v_i), 1\leq i\leq k$
is a basis of $A.$ The value $h(v)$ is referred to the {\it height} of $v$ in $V.$

\smallskip
\noindent
{\bf Skew brackets}.
Recall that a Hopf algebra $H$ is referred to as a {\it character} Hopf 
algebra if the group $G$ of all group like elements is commutative
and $H$ is generated over {\bf k}$[G]$ by skew primitive 
semi-invariants $a_i,\ i\in I:$
\begin{equation}
\Delta (a_i)=a_i\otimes 1+g_i\otimes a_i,\ \ \ \
g^{-1}a_ig=\chi ^i(g)a_i, \ \ g, g_i\in G,
\label{AIc}
\end{equation}
where $\chi ^i,\, i\in I$ are characters of the group $G.$
By means of the Dedekind Lemma it is easy to see that 
every character Hopf algebra is graded by the monoid $G^*$ of characters
generated by $\chi ^i:$
\begin{equation}
H=\sum _{\chi \in G^*}\oplus H^{\chi }, \ \ H^{\chi }=\{ a\in H\ |\ g^{-1}ag=\chi (g)a, \  g\in G\}.
\label{grad}
\end{equation}
Let us associate a ``quantum" variable $x_i$ to $a_i.$
For each word $u$
in $X=\{ x_i\, |\, i\in I\}$ we denote by $g_u$ or gr$(u)$ an element of $G$
that appears from $u$ by replacing each $x_i$ with $g_i.$
In the same way we denote by $\chi ^u$ a character that appears from $u$
by replacing of each $x_i$ with $\chi ^i.$
We define a bilinear skew commutator on homogeneous linear combinations of words
in $a_i$ or in $x_i,$ $i\in I$ by the formula
\begin{equation}
[u,v]=uv-\chi ^u(g_v) vu,
\label{sqo}
\end{equation}
where sometimes for short we use the notation  $\chi ^u(g_v)=p_{uv}=p(u,v).$
Of course $p(u,v)$ is a bimultiplicative map: 
\begin{equation}
p(u, vt)=p(u,v)p(u,t), \ \ p(ut,v)=p(u,v)p(t,v).
l* {sqot}
\end{equation}
The brackets satisfy the following Jacobi identity:
\begin{equation}
[[u, v],w]=[u,[v,w]]+p_{wv}^{-1}[[u,w],v]+(p_{vw}-p_{wv}^{-1})[u,w]\cdot v.
\label{jak1}
\end{equation}
or, equivalently, in the other less symmetric form
\begin{equation}
[[u, v],w]=[u,[v,w]]+p_{vw}[u,w]\cdot v-p_{uv}v\cdot [u,w].
\label{jak2}
\end{equation}
Jacobi identity (\ref{jak1}) implies the following conditional  identity
\begin{equation}
[[u, v],w]=[u,[v,w]],\hbox{ provided that } [u,w]=0.
\label{jak3}
\end{equation}
By the evident induction on length  this conditional identity admits  the following  generalization,
see \cite[Lemma 2.2]{KL}.
\begin{lemma}
If $y_1,$ $y_2,$ $\ldots ,$ $y_m$ are homogeneous linear combinations of words such that
$[y_i,y_j]=0,$ $1\leq i<j-1<m,$ then  the bracketed polynomial $[y_1y_2\ldots y_m]$ is independent 
of the precise alignment of brackets:
\begin{equation}
[y_1y_2\ldots y_m]=[[y_1y_2\ldots y_s],[y_{s+1}y_{s+2}\ldots y_m]], \ 1\leq s<m.
\label{ind}
\end{equation}
\label{indle}
\end{lemma}

The brackets are related with the product by the following ad-identities
\begin{equation}
[u\cdot v,w]=p_{vw}[u,w]\cdot v+u\cdot [v,w], 
\label{br1f}
\end{equation}
\begin{equation}
[u,v\cdot w]=[u,v]\cdot w+p_{uv}v\cdot [u,w].
\label{br1}
\end{equation}
In particular, if $[u,w]=0,$ we have
\begin{equation}
[u\cdot v,w]=u\cdot [v,w].
\label{br2}
\end{equation}

The antisymmetry identity transforms into the following two equalities
\begin{equation}
[u,v]=-p_{uv}[v,u]+(1-p_{uv}p_{vu})u\cdot v
\label{cha}
\end{equation}
\begin{equation}
[u,v]=-p_{vu}^{-1}[v,u]+(p_{vu}^{-1}-p_{uv})v\cdot u.
\label{cha1}
\end{equation}
In particular, if $p_{uv}p_{vu}=1,$ the ``color" antisymmetry,  $[u,v]=-p_{uv}[v,u],$ holds.

The group  $G$ acts on the free algebra ${\bf k}\langle X\rangle $
by $ g^{-1}ug=\chi ^u(g)u,$ where $u$ is an arbitrary monomial 
in $X.$
The skew group algebra $G\langle X\rangle $
has the natural Hopf algebra structure 
$$
\Delta (x_i)=x_i\otimes 1+g_i\otimes x_i, 
\ \ \ i\in I, \ \ \Delta (g)=g\otimes g, \ g\in G.
$$
We fix a Hopf algebra homomorphism
\begin{equation}
\xi :G\langle X\rangle \rightarrow H, \ \ \xi (x_i)=a_i, \ \ \xi (g)=g, \ \ i\in I, \ \ g\in G.
\label{gom}
\end{equation}

\smallskip
\noindent 
{\bf PBW-basis of a character Hopf algebra}.
A {\it constitution} of a word $u$ in $G \cup X$
is a family of non-negative integers $\{ m_x, x\in X\} $
such that $u$ has $m_x$ occurrences of $x.$
Certainly almost all $m_x$ in the constitution are zero.
We fix an arbitrary complete order, $<,$ on the set $X.$
Normally if $X=\{ x_1,\ldots ,x_n\} ,$ we set $x_1>x_2>\ldots >x_n.$

Let $\Gamma ^+$ be the free additive (commutative) monoid generated by $X.$
The monoid $\Gamma ^+$ \label{Gamma} is a completely ordered monoid with respect to 
the following order:
\begin{equation}
m_1x_{i_1}+m_2x_{i_2}+\ldots +m_kx_{i_k}>
m_1^{\prime }x_{i_1}+m_2^{\prime }x_{i_2}+\ldots +m_k^{\prime }x_{i_k}
\label{ord}
\end{equation}
if the first from the left nonzero number in
$(m_1-m_1^{\prime}, m_2-m_2^{\prime}, \ldots , m_k-m_k^{\prime})$
is positive, where $x_{i_1}>x_{i_2}>\ldots >x_{i_k}$ in $X.$
We associate a formal degree $D(u)=\sum _{x\in X}m_xx\in \Gamma ^+$
to a word $u$ in $G\cup X,$ where $\{ m_x\, | \, x\in X\}$ is the constitution of $u.$
Respectively, if $f=\sum \alpha _i u_i\in G\langle X\rangle ,$ $0\neq \alpha _i\in {\bf k},$
then 
\begin{equation}
D(f)={\rm max}_i\{ D(u_i)\} .
\label{degr}
\end{equation}

On the set of all words in $X$ we fix the lexicographical order
with the priority from the left to the right,
where a proper beginning of a word is considered to 
be greater than the word itself.

A non-empty word $u$
is called a {\it standard} word (or {\it Lyndon} word, or 
{\it Lyndon-Shirshov} word) if $vw>wv$
for each decomposition $u=vw$ with non-empty $v,w.$
A {\it nonassociative} word is a word where brackets 
$[, ]$ somehow arranged to show how multiplication applies.
If $[u]$ denotes a nonassociative word, then by $u$ we denote 
an associative word obtained from $[u]$ by removing the brackets.
The set of {\it standard nonassociative} words is the biggest set $SL$
that contains all variables $x_i$
and satisfies the following properties.

1)\ If $[u]=[[v][w]]\in SL,$
then $[v],[w]\in SL,$
and $v>w$
are standard.

2)\  If $[u]=[\, [[v_1][v_2]]\, [w]\, ]\in SL,$ then $v_2\leq w.$

\noindent
Every standard word has
only one alignment of brackets such that the appeared
nonassociative word is standard (Shirshov theorem \cite{pSh1}).
In order to find this alignment one may use the following inductive
procedure: 

\smallskip
\noindent
{\bf Algorithm}. \label{algo} The factors $v, w$
of the nonassociative decomposition $[u]=[[v][w]]$
are the standard words such that $u=vw$
and  $v$ has the minimal length (\cite{pSh2}, see also \cite{Lot}).
\begin{definition}  \rm A {\it super-letter}
is a polynomial that equals a nonassociative standard word
where the brackets mean (\ref{sqo}).
A {\it super-word} is a word in super-letters. 
\label{sup1}
\end{definition}

By Shirshov's theorem every standard word $u$
defines only one super-letter, in what follows we shall denote it by $[u].$
The order on the super-letters is defined in the natural way:
$[u]>[v]\iff u>v.$

In what follows we fix a notation $H$ for a  character Hopf algebra homogeneous in each $a_i,$
see (\ref{AIc}) and (\ref{gom}).
\begin{definition} \rm
A super-letter $[u]$
is called {\it hard in }$H$
provided that its value in $H$
is not a linear combination
of values of super-words of the same degree (\ref{degr})
in smaller than $[u]$ super-letters. 
\label{tv1}
\end{definition}
\begin{definition} \rm
We say that a {\it height} of a hard  super-letter $[u]$ in $H$
equals $h=h([u])$ if  $h$
is the smallest number such that: first, $p_{uu}$
is a primitive $t$-th root of 1 and either $h=t$
or $h=tl^r,$ where $l=$char({\bf k}); and  then the value 
of $[u]^h$ in $H$
is a linear combination of super-words of the same degree (\ref{degr})
in  less than $[u]$ super-letters.
If there exists no such number,  then the height equals infinity.
\label{h1}
\end{definition}
\begin{theorem} $(${\rm \cite[{Theorem 2}]{Kh3}}$).$
The values of all hard  super-letters in $H$ with the above defined height function
form a set of PBW-generators for $H$ over {\bf k}$[G].$
\label{BW}
\end{theorem}
 
\noindent 
{\bf PBW-basis of a homogeneous right coideal subalgebra}. 
The set $T$ of PBW-generators for a homogeneous right coideal subalgebra {\bf U},
${\bf k}[G]\subseteq {\bf U}\subseteq H,$ 
can be obtained from the PBW-basis given in Theorem \ref{BW}
in the following way; see  \cite[Theorem 1.1]{KhT}.

Suppose that for a hard super-letter $[u]$ there exists  a homogeneous element $c\in ${\bf U}
with the leading term $[u]^s$ in the PBW-decomposition given in Theorem \ref{BW}:    
\begin{equation}
c=[u]^s+\sum_i \alpha _iW_i\in \hbox{\bf U},
\label{vad1}
\end{equation}
where $W_i$ are the basis super-words starting with less than $[u]$ super-letters.
 We fix one of the elements with the minimal $s,$ and denote it by $c_u.$ 
Thus, for every hard super-letter $[u]$ in $H$ we have at most one element $c_u.$
We define the height function by means of the following lemma.

\begin{lemma}$\!\!\! ($\cite[{\rm Lemma 4.3}]{KhT}$).$ 
In the representation $(\ref{vad1})$ of the chosen element $c_u$
either $s=1,$ or $p(u,u)$ is a primitive $t$-th root of $1$ and $s=t$ or 
$($in the case of positive characteristic$)$ $s=t({\rm char}\, {\bf k})^r.$
\label{nco1}
\end{lemma}
If the height of $[u]$ in $H$ is infinite, then the height of $c_u$ in {\bf U}
is defined to be infinite as well. If the height of $[u]$ in $H$ equals $t,$ then, due to 
the above lemma, $s=1$ (in the PBW-decomposition (\ref{vad1}) the exponent 
$s$ must be less than
the height of $[u]$). In this case the height of $c_u$ in {\bf U} is supposed to be $t$ as well.
If the characteristic $l$ is positive, and the height of $[u]$ in $H$ equals
$tl^r,$ then we define the height of $c_u$ in {\bf U} to be equal to $tl^r/s.$

\begin{proposition} {\rm (\cite[Proposition 4.4]{KhT}).}
The set of all chosen $c_u$ with the above defined height function
forms a set of PBW-generators for {\bf U} over {\bf k}$[G].$  
\label{pro}
\end{proposition}

We are reminded that the PBW-basis is not uniquely defined in the above process. 
Nevertheless the set of leading terms of the PBW-generators indeed is uniquely defined.

\begin{definition} \rm  The degree $sD(c_u)\in \Gamma ^+ $ of a PBW-generator $c_u$
is said to be an {\bf U}-{\it root}. An {\bf U}-root $\gamma \in \Gamma ^+ $
is called a {\it simple} {\bf U}-{\it root} if it is not a sum of two or more other {\bf U}-roots.
\label{root}
\end{definition}
The set of {\bf U}-roots, and the set of 
simple  {\bf U}-roots are invariants for any right coideal subalgebra {\bf U}.

\smallskip
\noindent
{\bf Shuffle representation}. If the kernel of $\xi $ defined in (\ref{gom}) is contained in 
the ideal $G\langle X\rangle ^{(2)}$ generated by $x_ix_j,$ $i,j\in I,$
then there exists a Hopf algebra projection $\pi : H\rightarrow {\bf k}[G],$
$a_i\rightarrow 0, $ $g_i\rightarrow g_i .$ Hence by the Radford theorem
\cite{Rad}  we have a decomposition in a biproduct,
$H=A\# {\bf k}[G],$ where $A$ is a subalgebra generated by $a_i,$ $i\in I,$
 see \cite[\S 1.5, \S 1.7]{AS}.
\begin{definition} \rm 
In what follows we denote by  ${\bf \Lambda }$ the biggest Hopf ideal in $G\langle X\rangle ^{(2)},$
where as above $G\langle X\rangle ^{(2)}$ is the ideal of $G\langle X\rangle $
generated by $x_ix_j,$ $i,j\in I.$ The ideal ${\bf \Lambda }$  is homogeneous in each $x_i\in X$, see \cite[Lemma 2.2]{KA}.
\label{lam}
\end{definition}

If Ker$\, \xi ={\bf \Lambda }$ or, equivalently, 
if $A$ is a quantum symmetric algebra (a Nichols algebra \cite[\S 1.3, Section 2]{AS}),
then $A$ has a shuffle representation as follows.

The algebra $A$ has a structure of a {\it braided Hopf algebra}, \cite{Tak1},
with a braiding $\tau (u\otimes v)=p(v,u)^{-1}v\otimes u.$
The braided coproduct  $\Delta ^b$ on $A$ 
is connected with the coproduct on $H$ in the following way
\begin{equation}
\Delta ^b(u)=\sum _{(u)}u^{(1)}\hbox{gr}(u^{(2)})^{-1}\underline{\otimes} u^{(2)}.
\label{copro}
\end{equation}
The tensor space $T(V),$ $V=\sum x_i{\bf k}$ also has a structure of a braided Hopf algebra.
This is the {\it quantum shuffle algebra} $Sh_{\tau }(V)$ with the coproduct 
\begin{equation}
\Delta ^b(u)=\sum _{i=0}^m(z_1\ldots z_i)\underline{\otimes} (z_{i+1}\ldots z_m),
\label{bcopro}
\end{equation}
where $z_i\in X,$ and $u=(z_1z_2\ldots z_{m-1}z_m)$ is the tensor
$z_1\otimes z_2\otimes \ldots \otimes z_{m-1}\otimes z_m$ 
considered as an element of $Sh_{\tau }(V).$ The shuffle product satisfies
\begin{equation}
(w)(x_i)=\sum _{uv=w}p(x_i,v)^{-1}(ux_iv), \ \ (x_i)(w)=\sum _{uv=w}p(u,x_i)^{-1}(ux_iv).
\label{spro}
\end{equation}
The map $a_i\rightarrow (x_i)$ defines an embedding of the braided Hopf algebra
$A$ into the braided Hopf algebra $Sh_{\tau }(V).$ This embedding is 
extremely useful for calculation of the coproduct due to formulae (\ref{copro}), (\ref{bcopro}).

\smallskip
\noindent 
{\bf Differential calculus}. 
The free algebra ${\bf k}\langle X\rangle $ has a coordinate differential calculus
\begin{equation}
\partial _j( x_i)=\delta _i^j,\ \ \partial _i(uv)=\partial _i(u)\cdot v+\chi ^{u}(g_i)u\cdot \partial _i(v).
\label{defdif}
\end{equation}
The partial derivatives connect the calculus with the coproduct on ${\bf k}\langle X\rangle$ via
\begin{equation}
\Delta (u)\equiv u\otimes 1+\sum _ig_i\partial _i(u)\otimes x_i\ \ \ 
(\hbox{mod }G\langle X\rangle \otimes \hbox{\bf k}\langle X\rangle ^{(2)}),
\label{calc}
\end{equation}
where ${\bf k}\langle X\rangle ^{(2)}$ is the  ideal
generated by $x_ix_j,$ $1\leq i,j\leq n.$
\begin{lemma} Let $u\in {\bf k}\langle X\rangle $ be an element homogeneous in each $x_i$.
If  $p_{uu}$ is a $t$-th primitive root of 1, then
\begin{equation}
\partial_i(u^t)=p(u,x_i)^{t-1}\underbrace{[u,[u,\ldots [u}_{t-1},\partial _i(u)]\ldots ]].
\label{ddif1}
\end{equation}
\label{dert}
\end{lemma}
\begin{proof} First of all we note that the sequence
$p_{uu},p_{uu}^2,\ldots ,p_{uu}^{t-1}$ contains all $t$-th roots 
of 1 except the 1 itself. All members in this sequence are different.
 Hence we may write a polynomial equality
\begin{equation}
(1-x^t)=(1-x)\prod _{s=1}^{t-1}(1-p_{uu}^{s}x).
\label{mno}
\end{equation}

Let us calculate the right hand side of (\ref{ddif1}). Denote by $L_u,$
$R_u$ the operators of left and right multiplication by $u$ respectively. 
The right hand side of (\ref{ddif1}) has the following operator representation
$$
p(u,x_i)^{t-1}(\partial _i(u)\cdot \prod_{s=1}^{t-1} (L_u-Qp_{uu}^{s-1}R_u)),
$$
where $Q=p(u,\partial _i(u))=p_{uu}p(u,x_i)^{-1}.$ Consider a polynomial
$$
f(\lambda )= \prod_{s=1}^{t-1}(1-Qp_{uu}^{s-1}\lambda )
\stackrel{df}{=}\sum _{k=0}^{t-1}\alpha _k\lambda ^k.
$$
Since the operators $R_u$ and $L_u$ commute, we may develop the multiplication in the operator product considering $R_u$ and $L_u$
as formal commutative variables:
$$
\prod _{s=1}^{t-1}(L_u-Qp_{uu}^{s-1}R_u)=L_u^{t-1}f(\frac{R_u}{L_u})
=\sum _{k=0}^{t-1}\alpha _kL_u^{t-1-k}R_u^k.
$$
Thus the right hand side of (\ref{ddif1}) equals
$$
p(u,x_i)^{t-1}\sum _{k=0}^{t-1}\alpha _ku^{t-1-k}\partial _i(u)u^k.
$$
Further, since $Q=p_{uu}p(u, x_i)^{-1},$ the polynomial $f$ has a representation
$$
f(\lambda )=\prod _{s=1}^{t-1}(1-p_{uu}^{s}\xi ),
$$
where $\xi =\lambda p(u, x_i)^{-1}.$ Taking into account (\ref{mno}), we get
$$
f(\lambda )=\frac{1-\xi ^t}{1-\xi }=
\frac{1-\lambda ^tp(u,x_i)^{-t}}{1-\lambda p(u,x_i)^{-1}}
$$
$$
=1+\lambda p(u,x_i)^{-1}+\lambda ^2 p(u,x_i)^{-2}+\cdots +
\lambda^{t-1} p(u,x_i)^{1-t};
$$
that is, $\alpha _k=p(u,x_i)^{-k},$ while the right hand side of (\ref{ddif1}) takes the form
\begin{equation}
\sum _{k=0}^{t-1}p(u,x_i)^{t-1-k}u^{t-1-k}\partial _i(u)u^k.
\label{uff}
\end{equation}
At the same time Leibniz formula (\ref{defdif})
 shows that $\partial _i(u^t)$ also equals (\ref{uff}).
\end{proof}

\smallskip
\noindent 
{\bf MS-criterion}. The quantum symmetric algebra has a lot of nice characterizations.
One of them says that this is the {\it optimal algebra}  for the calculus defined by (\ref{defdif}).
In other words the above defined algebra $A$
is a quantum symmetric algebra (or, equivalently, Ker$\, \xi ={\bf \Lambda }$)
if and only if all constants in $A$ are scalars.

For braidings of Cartan type this characterization was proved 
by A. Milinski and H.-J. Schneider in \cite{MS}, and then generalized for arbitrary 
(even not necessarily  invertible)  braidings by the author in \cite[Theorem 4.11]{Kh1}.
Moreover, if $X$ is finite, then ${\bf \Lambda }$ 
(as well as any differential ideal in ${\bf k}\langle X\rangle $)
is generated as a left ideal by constants from ${\bf k}\langle X\rangle ^{(2)},$
see \cite[Corrolary 7.8]{Kh1}.
Thus, we may formulate the following criterion that is useful for checking relations.
\begin{lemma} {\rm (Milinski---Schneider criterion).} Suppose that
Ker$\, \xi ={\bf \Lambda }.$
If a polynomial $f\in {\bf k}\langle X\rangle $
 is a constant in $A$ $($that is, $\partial _i(f)\in {\bf \Lambda },$ $i\in I),$
then there exists $\alpha \in \, ${\bf k} such that $f-\alpha =0$ in $A.$ 
\label{MS}
\end{lemma}
Of course one can easily prove this criterion
by means of (\ref{copro}), (\ref{bcopro}) and (\ref{calc})
using the above shuffle representation since 
(\ref{bcopro}) implies that all constants in the shuffle coalgebra are scalars.

\smallskip
\noindent 
{\bf Quantum Borel algebra}. Let $C=||a_{ij}||$ be a symmetrizable by 
$D={\rm diag }(d_1, \ldots , d_n)$ generalized Cartan matrix, $d_ia_{ij}=d_ja_{ji}.$
Denote by $\mathfrak g$  a Kac-Moody algebra defined by $C,$ see \cite{Kac}.
Suppose that  parameters $p_{ij}$ are related by
\begin{equation}
p_{ii}=q^{d_i}, \ \ p_{ij}p_{ji}=q^{d_ia_{ij}},\ \ \ 1\leq i,j\leq n. 
\label{KM1}
\end{equation}
Denote by $g_j$ a linear transformation $g_j:x_i\rightarrow p_{ij}x_i$ of the linear space 
spanned by a set of variables $X=\{ x_1, x_2, \ldots , x_n\} .$
By $\chi ^i$ we denote a character $\chi^i :g_j\rightarrow p_{ij}$ of the group $G$ generated by 
$g_i,$ $1\leq i\leq n.$ We consider each $x_i$ as a ``quantum variable" with parameters $g_i,$ $\chi ^i.$
 As above we denote by $G\langle X\rangle $ the skew group algebra 
with commutation rules $x_ig_j=p_{ij}g_jx_i,$ $1\leq i, j\leq n.$ 
This algebra has a structure of a character Hopf algebra 
\begin{equation}
\Delta (x_i)=x_i\otimes 1+g_i\otimes x_i,\ \ \ \ \Delta (g_i)=g_i\otimes g_i.
\label{AIr}
\end{equation}
In this case the multiparameter quantization $U^+_q ({\mathfrak g})$ of the  Borel subalgebra ${\mathfrak g}^+$
is a homomorphic image of $G\langle X\rangle $ defined by Serre relations 
with the skew brackets in place of the Lie operation:
\begin{equation}
[\ldots [[x_i,\underbrace{x_j],x_j], \ldots ,x_j]}_{1-a_{ji} \hbox{ times}}=0, \ \ 1\leq i\neq j\leq n.
\label{KM2}
\end{equation}
By \cite[Theorem 6.1]{Khar} the left hand sides of these relations are skew-primitive elements 
in $G\langle X\rangle .$ Therefore the ideal generated by  these elements is a Hopf ideal,
while $U^+_q ({\mathfrak g})$  has a natural structure of a character Hopf algebra.

\begin{lemma} {\rm (\cite[Corollary 3.2]{KL})}.
If $q$ is not a root of 1 and $C$ is of finite type, then every subalgebra  $U$ of $U_q^+(\frak{g})$ 
containing $G$ is homogeneous with  respect to each of the variables $x_i$.
\label{odn1}
\end{lemma}
\begin{definition}
If the multiplicative order $t$ of $q$ is finite, 
then we define $u^+_q ({\mathfrak g})$ as $G\langle X\rangle/{\bf \Lambda },$
where ${\bf \Lambda }$ is the biggest Hopf ideal in $G\langle X\rangle ^{(2)},$
see Definition \ref{lam}. 
\label{small}
\end{definition}
Since a skew-primitive element generate a Hopf ideal,
 ${\bf \Lambda }$ contains all skew-primitive elements of $G\langle X\rangle ^{(2)}.$
Hence relations (\ref{KM2}) are still valid in $u^+_q ({\mathfrak g}).$

\section{Relations in quantum Borel algebra $U_q^+({\mathfrak so}_{2n+1}).$}
In what follows we fix a parameter $q$ such that $q^4\not=1,$ $q^3\not= 1.$
If $C$ is a Cartan matrix of type $B_n,$ relations (\ref{KM1}) take up the form
\begin{equation}
p_{nn}=q,\,  p_{ii}=q^2, \ \ p_{i\, i+1}p_{i+1\, i}=q^{-2}, \ 1\leq i<n; 
\label{b1rel}
\end{equation}
\begin{equation}
p_{ij}p_{ji}=1,\  j>i+1. 
\label{b1rell}
\end{equation}
Starting with parameters $p_{ij}$ satisfying these relations, we define 
the group $G$ and the character Hopf algebra  $G\langle X\rangle $
as in the above subsection. In this case
the quantum Borel algebra $U^+_q ({\frak so}_{2n+1})$
is a homomorphic image of $G\langle X\rangle $ subject to the following relations 
\begin{equation}
[x_i,[x_i,x_{i+1}]]=0, \ 1\leq i<n; \ \ [x_i,x_j]=0, \ \ j>i+1;
\label{relb}
\end{equation}
\begin{equation}
[[x_i,x_{i+1}],x_{i+1}]=[[[x_{n-1},x_n],x_n],x_n]=0, \ 1\leq i<n-1.
\label{relbl}
\end{equation}
Here we slightly modify  Serre relations (\ref{KM2}) so that the left hand side of each 
relation is a super-letter. It is possible to do due to the following
general relation in ${\bf k}\langle X\rangle ,$ see \cite[Corollary 4.10]{Kh4}:
\begin{equation}
[\ldots [[x_i,
\underbrace{x_j],x_j],\ldots x_j]}_n=\alpha \underbrace{[x_j,[x_j,\ldots [x_j}_n,x_i]\ldots ]], 
\ \ 0\neq \alpha \in {\bf k},
\label{rsc}
\end{equation}
provided that $p_{ij}p_{ji}=p_{jj}^{1-n}.$

\begin{definition} \rm
The elements $u,v$ are said to be 
{\it separated} if there exists an index $j,$ $1\leq j\leq n,$
such that either $u\in {\bf k}\langle x_i\ |\ i<j\rangle ,$
$v\in {\bf k}\langle x_i\ |\ i>j\rangle $ or vise versa: 
$u\in {\bf k}\langle x_i\ |\ i>j\rangle ,$
$v\in {\bf k}\langle x_i\ |\ i<j\rangle .$
\label{sep}
\end{definition}
\begin{lemma}
In the algebra $U^+_q ({\frak so}_{2n+1})$ every two separated
homogeneous in each $x_i\in X$ elements $u,v$  $($skew$)$commute, $[u,v]=[v,u]=0.$
\label{sepp}
\end{lemma}
\begin{proof} Relations (\ref{b1rell}) and conditional antisymmetry (\ref{cha})
show that $[x_i,x_j]=[x_j,x_i]=0,$ provided that $|i-j|>1.$ Now relations 
(\ref{br1f}), (\ref{br1}) allow one to perform an evident induction.
\end{proof}

Certainly the subalgebra of $U_q^+({\mathfrak so}_{2n+1})$ 
generated over {\bf k}$[g_1,\ldots , g_{n-1}]$ by $x_i,$ $1\leq i<n$
is the Hopf algebra $U_{q^2}^+({\mathfrak sl}_{n})$ defined by the Cartan matrix 
of type $A_{n-1}.$ 
Let us replace  just one parameter $p_{nn}\leftarrow q^2.$ Then the quantum Borel 
algebra $U_{q^2}^+({\mathfrak sl}_{n+1})$ is a homomorphic image  of $G^{\prime }\langle X\rangle$
subject to relations
\begin{equation}
[[x_i,x_{i+1}],x_{i+1}]=[x_i,[x_i,x_{i+1}]]=[x_i,x_j]=0, \ \ j>i+1.
\label{rela}
\end{equation}
Here $G^{\prime }$ is the group generated by transformations $g_1,\ldots ,g_{n-1},g_n^{\prime },$
where $g_n^{\prime }(x_i)=g_n(x_i)$ with only one exception being $g_n^{\prime }(x_n)=q^{2}x_n.$

\begin{lemma}
A linear in $x_n$ relation $f=0,$ $f\in {\bf k}\langle X\rangle $ is valid in 
$U_q^+({\mathfrak so}_{2n+1})$ if and only if it is valid
in the above algebra $U_{q^2}^+({\mathfrak sl}_{n+1}).$
\label{xn}
\end{lemma}
\begin{proof}
The element $f$ as an element of free algebra belongs to the 
ideal generated by the defining relations that are independent of 
$x_n$ or linear in $x_n.$ All that relations are the same for $U_q^+({\mathfrak so}_{2n+1})$
and for $U_{q^2}^+({\mathfrak sl}_{n+1}).$
\end{proof}
\begin{lemma}
If $u$ is a standard word, then either $u=x_kx_{k+1}\ldots x_m,$ $k\leq m\leq n,$
or $[u]=0$ in $U_{q^2}^+({\mathfrak sl}_{n+1}).$ Here $[u]$ is a nonassociative word
with the standard alignment of brackets, see Algorithm on page  $\pageref{algo}.$ 
\label{nul}
\end{lemma}
\begin{proof}
See the third statement of \cite[Theorem $A_n$]{Kh4}.
\end{proof}
As a corollary of the above two lemmas we can prove some relations in 
$U_q^+({\mathfrak so}_{2n+1}):$
\begin{equation}
[[x_{k+1}x_kx_{k-1}], x_k]=0, \ \  [[x_{k-1}x_kx_{k+1}],x_k]=0, \ \ k<n.
\label{zan}
\end{equation}
Indeed,  $x_{k-1}x_kx_{k+1}x_k$ is a standard word and the standard alignment of brackets
is precisely 
$[[x_{k-1},[x_k,x_{k+1}]],x_k].$
 Hence (\ref{jak3}) with 
Lemma \ref{xn} and Lemma \ref{nul}
imply the latter relation. 

The former relation reduces to the latter one
by means of the replacement $x_i\leftarrow x_{n-i+1},$ $1\leq i\leq n,$ $k\leftarrow n-k+1.$
It remains to note that the defining relations (\ref{rela})
are invariant under this replacement (see (\ref{rsc})), and again use
Lemma \ref{xn} and Lemma \ref{nul}.
\begin{definition} \rm 
In what follows we denote by $x_i,$ $n<i\leq 2n$ the generator $x_{2n-i+1}.$
Respectively, $u(k,m),$ $1\leq k\leq m\leq 2n$ is the word
$x_kx_{k+1}\cdots x_{m-1}x_m,$ while $u(m,k)$ is the word
$x_mx_{m-1}\cdots x_{k+1}x_k.$  If $1\leq i\leq 2n,$ then by $\psi (i)$
we denote the number $2n-i+1,$ so that $x_i=x_{\psi (i)}.$ 
We shall frequently use the following properties of $\psi :$
if $i<j,$ then $\psi (i)>\psi (j);$ $\psi (\psi (i))=i;$ $\psi (i+1)=\psi (i)-1.$
\label{fis}
\end{definition}
\begin{definition} \rm 
If $k\leq i<m\leq 2n,$ then we denote
\begin{equation}
\sigma _k^m\stackrel{df}{=}p(u(k,m),u(k,m)),
\label{mu11}
\end{equation}
\begin{equation}
\mu _k^{m,i}\stackrel{df}{=}p(u(k,i),u(i+1,m))\cdot p(u(i+1,m),u(k,i)).
\label{mu1}
\end{equation}
\label{slo}
\end{definition}
Of course, one can find $\mu$'s and $\sigma $'s by means of (\ref{b1rel}), (\ref{b1rell}). 
It turns out that these coefficients depend only on $q.$
More precisely, 
\begin{equation}
\sigma_k^m =\left\{ \begin{matrix}
q,\hfill &\hbox{if } m=n, \hbox{ or } k=n+1;\hfill \cr
q^4,\hfill &\hbox{if }  m=\psi (k);\hfill \cr
q^{2},\hfill &\hbox{otherwise}.\hfill 
                             \end{matrix}
                    \right.
\label{mu21}
\end{equation}
Indeed, the bimultiplicativity of $p(-,-)$ implies that $\sigma _k^m=\prod _{k\leq s,t\leq m}p_{st}$
is the product of all coefficients of the $(m-k+1)\times (m-k+1)$-matrix $||p_{st}||.$ By (\ref{b1rel})
all coefficients on the main diagonal equal $q^2$ with only two possible exceptions being $p_{nn}=q,$ $p_{n+1\, n+1}=q.$
In particular, if $m<n$ or $k>n+1,$ then for non diagonal coefficients we have 
$p_{st}p_{ts}=1$ unless $|s-t|=1,$ while  $p_{s\, s+1}p_{s+1\, s}=q^{-2}.$ 
Hence $\sigma _k^m=q^{2(m-k+1)}\cdot q^{-2(k-m)}=q^2.$
If $m=n$ or $k=n+1,$ then by the same reason we have $\sigma _k^m=q^{2(m-k)+1}\cdot q^{-2(k-m)}=q.$
In the remaining case,  $k\leq n<m,$ we split the matrix in four submatrices as follows
\begin{equation}
\sigma _k^m=\sigma _k^n\cdot \sigma _{n+1}^m\cdot 
\prod_{k\leq s\leq n, \, n+1\leq t\leq m}p_{st} \cdot 
\prod_{ n+1\leq s\leq m, \,  k\leq t\leq n }p_{st}.
\label{mu22}
\end{equation}
According to Definition \ref{fis} we have $p_{st}=p_{\psi (s)\, t}=p_{s\, \psi (t)}=p_{\psi (s)\, \psi (t)}.$
Therefore the third and fourth factors in (\ref{mu22}) respectively equal
$$
\prod_{k\leq s\leq n, \, \psi (m)\leq t\leq n}p_{st}; \ \ 
\prod_{ \psi (m)\leq s\leq n, \,  k\leq t\leq n }p_{st}.
$$
In particular, if $\psi (m)=k,$ then all four factors in (\ref{mu22}) coincide with $\sigma _k^n=q,$
hence $\sigma _k^m=q^4.$ If $\psi (m)\neq k,$ say $\psi (m)>k,$ then we split the rectangular 
$A=[k,n]\times [\psi (m),n]$ in a union of the square  $B=[\psi (m),n]\times [\psi (m),n]$ and the rectangular 
$C=[k,\psi (m)-1]\times [\psi (m),n].$ Similarly the rectangular 
$A^*=[\psi (m),n]\times [k,n]$ is a union of the same square  and the rectangular
$C^*=[\psi (m),n]\times [k,\psi (m)-1].$ Certainly, if $(s,t)\in C,$ then $t-s>1$ unless $t=\psi (m)-1,$
$s=\psi (m).$ Hence relations (\ref{b1rell}) imply 
$$\prod _{(s,t)\in C}p_{st}p_{ts}=p_{\psi (m)-1\, \psi (m)}p_{\psi (m)\ \psi (m)-1}=q^{-2}.$$
At the same time $\prod_{(s,t)\in B} p_{st}=\sigma _{\psi (m)}^n=q.$ Finally, (\ref{mu22}) takes the form
$$
\sigma _k^m=q\cdot q\cdot (\prod_{(s,t)\in B} p_{st})^2\cdot \prod _{(s,t)\in C}p_{st}p_{ts}=q^2,
$$
which proves (\ref{mu21}).

To find $\mu $'s we consider decomposition (\ref{mu22}) with $n\leftarrow i.$ 
Since $p(-,-)$ is a bimultiplicative map,
the product of the last two factors is precisely  $\mu _k^{m,i}.$ In particular we have 
\begin{equation}
\mu_k^{m,i}=\sigma _k^m(\sigma _k^i \sigma _{i+1}^m)^{-1}. 
\label{mu23}
\end{equation}
This formula with (\ref{mu21}) allows one easily to find the $\mu $'s.
More precisely, if $m<\psi (k),$  then
\begin{equation}
\mu_k^{m,i} =\left\{ \begin{matrix}
q^{-4},\hfill &\hbox{if } m>n, \, i=\psi (m)-1;\hfill \cr
1,\hfill &\hbox{if }  i=n;\hfill \cr
q^{-2},\hfill &\hbox{otherwise}.\hfill 
                             \end{matrix}
                    \right.
\label{mu2}
\end{equation}
If $m=\psi (k);$ that is $x_m=x_k,$ then
\begin{equation}
\mu_k^{m,i} =\left\{ \begin{matrix}
q^2,\hfill &\hbox{if }  i=n;\hfill \cr
1,\hfill &\hbox{otherwise}.\hfill 
                             \end{matrix}
                    \right.
\label{mu3}
\end{equation}
If $m>\psi (k),$ then the $\mu $'s satisfy 
$\mu_k^{m,i}=\mu_{\psi (m)}^{\psi (k),\, \psi (i)-1},$ hence one may use (\ref{mu2}):
\begin{equation}
\mu_k^{m,i} =\left\{ \begin{matrix}
q^{-4},\hfill &\hbox{if } k\leq n,\, i=\psi (k);\hfill \cr
1,\hfill &\hbox{if }  i=n;\hfill \cr
q^{-2},\hfill &\hbox{otherwise}.\hfill 
                             \end{matrix}
                    \right.
\label{mu4}
\end{equation}

\

We define the bracketing of $u(k,m),$ $k\leq m$ as follows.
\begin{equation}
u[k,m]=\left\{ \begin{matrix} [[[\ldots [x_k,x_{k+1}], \ldots ],x_{m-1}], x_m],\hfill 
&\hbox{if } m<\psi (k);\hfill \cr
 [x_k,[x_{k+1},[\ldots ,[x_{m-1},x_m]\ldots ]]],\hfill &\hbox{if } m>\psi (k);\hfill \cr 
\beta [u[n+1,m],u[k,n]],\hfill &\hbox{if } m=\psi (k),\hfill 
\end{matrix}\right.
\label{ww}
\end{equation}
where $\beta =-p(u(n+1,m),u(k,n))^{-1}$ normalizes the coefficient at $u(k,m).$
Conditional identity (\ref{ind}) shows that the value of  $u[k,m]$ in $U_q^+({\mathfrak so}_{2n+1})$
is independent of the precise alignment of brackets provided that $m\leq n$ or $k>n.$

In what follows we denote by $\sim $ the projective equality:
$a\sim b$ if and only if $a=\alpha b,$ where $0\neq \alpha \in {\bf k}.$ 

\begin{lemma} If $t\notin \{ k-1, k\} ,$ $t<n,$ then $[u[k,n],x_t]=[x_t,u[k,n]]=0.$ 
\label{nule}
\end{lemma}
\begin{proof} If $t\leq k-2,$ then the equality follows from the second group of defining relations (\ref{relb}). 
Let $k<t<n.$ By (\ref{jak3}) we may write
$$
[u[k,n],x_t]=\hbox{\Large[}[u[k,t-2],u[t-1,n]],x_t\hbox{\Large]}
=\hbox{\Large[}u[k,t-2], [u[t-1,n],x_t]\hbox{\Large]}.
$$
By Lemma \ref{nul} the element $[u[t-1,n],x_t]$ equals zero in $U_{q^2}^+({\mathfrak sl}_{n+1})$
since the word $u(t-1,n)x_t$ is standard and the standard bracketing is precisely $[u[t-1,n],x_t].$ 
This element is linear in $x_n.$ Hence $[u[k,n],x_t]=0$ in $U_q^+({\mathfrak so}_{2n+1})$
as well due to Lemma \ref{xn}. Since $p(u(k,n),x_t)p(x_t,u(k,n))=$
$p_{t\, t+1}p_{tt}p_{t\, t-1}\cdot p_{t+1\, t}p_{tt}p_{t-1\, t}=1,$
the antisymmetry identity (\ref{cha}) applies.
\end{proof}
\begin{lemma} If $t\notin \{ \psi (m)-1, \psi (m) \} ,$ $t<n<m,$ then 
$$[x_t,u[n+1,m]]=[u[n+1,m],x_t]=0.$$ 
\label{nulm}
\end{lemma}
\begin{proof}
If $t\leq \psi (m)-2,$ then the required relation follows from the second group of relations
(\ref{relb}). Let $\psi (m)<t<n.$ By Lemma \ref{indle} the value of $u[n+1,m]$
in $U_q^+({\mathfrak so}_{2n+1})$ is independent of the alignment of brackets.
In particular $u[n+1,m]=[[w,[x_{t+1}x_tx_{t-1}]],v],$ where $w=u[n+1,\psi (t)-2],$
$v=u[\psi (t)+2,m].$ Since $p_{t\, t+1}p_{tt}p_{t\, t-1}\cdot p_{t+1\, t}p_{tt}p_{t-1\, t}=1,$
the antisymmetry identity (\ref{cha}) and the first of (\ref{zan}) imply
$[x_t, [x_{t+1}x_tx_{t-1}]]\sim [[x_{t+1}x_tx_{t-1}],x_t]=0.$
It remains to note that $[x_t, w]=[w,x_t]=0,$ $[x_t,v]=[v,x_t]=0$ according to the second group of 
defining relations (\ref{relb}).
\end{proof}
\begin{lemma} 
If  $k\leq n<m<\psi (k),$ then the value in $U_q^+({\mathfrak so}_{2n+1})$ 
of the bracketed word $[y_kx_{n+1}x_{n+2}\cdots x_m],$ where $y_k=u[k,n],$
is independent of the precise alignment of brackets. 
\label{ins}
\end{lemma}
\begin{proof} In order to apply (\ref{ind}) it suffices to check 
$[u[k,n],x_t]=0,$  $n+1<t\leq m.$ Since the application of $\psi $ change the order, we have
$k<\psi (m)\leq \psi (t)<n.$ Hence taking into account $x_t=x_{\psi (t)},$
one may use Lemma \ref{nule}.
\end{proof}
\begin{lemma} 
If $k\leq n<\psi (k)<m,$ then the value in $U_q^+({\mathfrak so}_{2n+1})$ 
of the bracketed word $[x_kx_{k+1}\cdots x_ny_m],$ where $y_m=u[n+1,m],$
is independent of the precise alignment of brackets. 
\label{ins1}
\end{lemma}
\begin{proof}
To apply (\ref{ind}) we need $[x_t,u[n+1,m]]=0,$
$k\leq t<n.$ To get these equalities one may use Lemma \ref{nulm}.
\end{proof}
\begin{lemma} If $m\neq \psi (k),$ $k\leq i<n<m,$ then 
$$[u[k,i],u[n+1,m]]=[u[n+1,m],u[k,i]]=0$$ 
unless $i=\psi (m)-1.$
\label{nulx}
\end{lemma}
\begin{proof}
Denote $u=u[k,i],$ $w=u[n+1,m].$ Relations (\ref{b1rel}), (\ref{b1rell}) imply
$p_{uw}p_{wu}=1.$ Hence by (\ref{cha}) we have $[u,w]=-p_{uw}[w,v].$ 

If $\psi (m)<k,$ then by Lemma \ref{nulm} we have
$[x_t,u[n+1,m]]=0,$ $k\leq t\leq i.$ Hence $[u[k,i],u[n+1,m]]=0.$

Suppose that  $\psi (m)>k.$ If $i<\psi (m)-1,$ then due to the second group of defining relations
(\ref{relb}) we have $[x_t,u[n+1,m]]=0,$ $k\leq t\leq i.$ Hence $[u[k,i],u[n+1,m]]=0.$

Let $\psi (m)\leq i<n.$ If we denote $u_1=u[k,\psi (m)-2],$ $u_2=u[\psi (m)-1,i],$ then certainly
$u=[u_1,u_2]$ unless $k=\psi (m)-1,$ $u=u_2.$ Since $[u_1,w]=0,$
conditional Jacobi identity (\ref{jak3}) implies that in both cases  we need just to check $[u_2,w]=0.$

Let us put $u_3=[x_{\psi (m)-1},x_{\psi (m)}],$  $u_4=u[\psi (m)+1,i].$
Then $u_2=[u_3,u_4]$ unless $i=\psi (m),$ $u_2=u_3.$
By Lemma \ref{nulm} we have $[x_t,u[n+1,m]]=0$ for all $t,$ $\psi (m)<t<n.$
Hence $[u_4,w]=0.$ Now Jacobi identity (\ref{jak1}) with $u\leftarrow u_3,$
$v\leftarrow u_4$ shows that it suffices to prove the equality $[u_3,w]=0.$

Let us put $w_1=u[n+1,m-2],$ $w_2=[x_{m-1},x_m].$ Then $w=[w_1,w_2]$
unless $m-2=n,$ $w=w_2$ (recall that we are considering the case
$\psi (m)\leq i<n,$ in particular $\psi (m)\leq n-1,$ and hence $m\geq \psi (n-1)=n+2$).
We have $[u_3,w_1]=0.$ Therefore Jacobi identity (\ref{jak1}) with 
$u\leftarrow u_3,$ $v\leftarrow w_1$ $w\leftarrow w_2$ shows that it is sufficient 
to get the equality $[u_3,w_2]=0;$ that is,
$[[x_{t-1},x_t],[x_{t+1},x_t]]=0$ with $t=\psi (m)<n.$
Since $[[x_{t-1},x_t],x_t]=0$ is one of the defining relations,
the conditional identity (\ref{jak3}) implies  
$[[x_{t-1},x_t],[x_{t+1},x_t]]$ $=[[x_{t-1}x_tx_{t+1}],x_t].$
It remains to apply the second of (\ref{zan}).
\end{proof}
\begin{lemma} If $m\neq \psi (k),$ $k\leq n<i<m,$ then 
$$[u[k,n],u[i+1,m]]=[u[i+1,m],u[k,n]]=0$$ unless $i=\psi (k).$
\label{nuly}
\end{lemma}
\begin{proof} The proof is quite similar to the proof of the above lemma.
It is based on Lemma \ref{nule} and the second of (\ref{zan})
in the same way as the proof of the above lemma is based 
on Lemma \ref{nulm} and the first of (\ref{zan}).
\end{proof}
\begin{corollary} If $m\neq \psi (k),$ $k\leq n<m,$ then in $U_q^+({\mathfrak so}_{2n+1})$ we have
\begin{equation}
u[k,m]=[u[k,n],u[n+1,m]]=\beta[u[n+1,m],u[k,n]],
\label{ww1}
\end{equation}
where $\beta =-p(u(n+1,m),u(k,n))^{-1}.$
\label{rww}
\end{corollary}
\begin{proof}
Let us denote $u=u[k,n],$ $v=u[n+1,m].$ Equalities (\ref{mu2}), (\ref{mu4}) with $i=n$
show that $p_{uv}p_{vu}=\mu _k^{m, n}=1$ provided that $m\neq \psi (k).$ Hence 
$[u,v]=uv-p_{uv}vu=-p_{uv}[v,u].$ This proves the second equality.
To prove the first one we apply Lemma \ref{ins} if
$m<\psi (k),$  and Lemma \ref{ins1} otherwise.
\end{proof}
\begin{proposition} 
If $m\neq \psi (k),$ then in $U_q^+({\mathfrak so}_{2n+1})$ for each $i,$ $k\leq i<m$  we have
$$
[u[k,i],u[i+1,m]]=u[k,m]
$$
with only two possible exceptions being  $i=\psi (m)-1,$  and $i=\psi (k).$
\label{ins2}
\end{proposition}
\begin{proof}
If $m\leq n$ or $k\geq n+1,$ then the statement follows from (\ref{ind}). 
Thus we may suppose that $m>n.$

If $i=n,$ Corollary \ref{rww} implies the required formula. 

If $i>n,$ then Corollary \ref{rww} yields $u[k,i]=[u[k,n],u[n+1,i]],$
while by Lemma \ref{nuly} we have $[u[k,n],u[i+1,m]]=0.$
Hence (\ref{jak3}) implies
$$
[[u[k,n],u[n+1,i]], u[i+1,m]]=[u[k,n],[u[n+1,i], u[i+1,m]]].
$$
Now (\ref{ind}) shows that $[u[n+1,i], u[i+1,m]]=[u[n,m]],$ and again 
Corollary \ref{rww} implies the required formula.

If $i<n,$ then Corollary \ref{rww} yields $u[i+1,m]=[u[i+1,n],u[n+1,m]],$
while by Lemma \ref{nulx} we have $[u[k,i],u[n+1,m]]=0.$
Hence (\ref{jak3}) implies
$$[[u[k,i],[u[i+1,n], u[n+1,m]]]=[[u[k,i],[u[i+1,n]], u[n+1,m]].$$
Now (\ref{ind}) shows that $[u[k,i], u[i+1,n]=u[k,n],$ and again 
 Corollary \ref{rww} implies the required formula.
\end{proof}
\begin{proposition} If $m\neq \psi (k),$ $k\leq i<j<m,$ $m\neq \psi (i)-1,$  $j\neq \psi (k),$ then 
$[u[k,i],u[j+1,m]]=0.$ If additionally $i\neq \psi (j)-1,$ then $[u[j+1,m],u[k,i]]=0.$
\label{NU}
\end{proposition}
\begin{proof}
If $m\leq n$ or $k>n,$ then $u[k,i]$ and $u[j+1,m]$ are separated by $x_j,$ hence  
the statement follows from Lemma \ref{sepp}.

If $k\leq n<i,$ then by Corollary \ref{rww} we have $u[k,i]=[a,b]$ with $a=u[k,n],$ $b=u[n+1,i].$
The second group of relations (\ref{relb}) implies $[b, u[j+1,m]]=0,$ while Lemma \ref{nuly}
implies $[a, u[j+1,m]]=0.$ Hence by (\ref{jak1}) we get the required relation.

If $j< n\leq m,$ then again by Corollary \ref{rww} we have $u[j+1,m]=[a,b]$
with $a=u[j+1,n],$ $b=u[n+1,m].$ The second group of relations (\ref{relb}) implies $[u[k,i], a]=0,$
while Lemma \ref{nulx} implies $[u[k,i], b]=0.$ Hence by (\ref{jak1}) we get the required relation.

Assume  $i\leq n\leq j.$  If $i>\psi (j)-1,$ then,  taking into account Lemma  \ref{xn}, one may apply 
 Lemma  \ref{nuly} with $n\leftarrow i,$ $i\leftarrow j.$ Similarly, if  $i<\psi (j)-1,$ one may apply
Lemma \ref{nulx} with $n\leftarrow \psi (j)-1.$ Let $i=\psi (j)-1.$ We may apply already considered case 
``$i>\psi (j)-1$" to the sequence $k\leq i<j^{\prime }<m$ with $j^{\prime }=j+1$ unless $j^{\prime }=m,$
or $j^{\prime }=\psi (k).$ Thus $[u[k,i],u[j+2,m]]=0,$ provided that $j+1\neq m,$ $j+1\neq \psi (k).$
Lemma \ref{indle} implies 
\begin{equation}
[u[k,i],x_i]=[u[k,i-2],[[x_{i-1},x_i],x_i]]=0,
\label{esh}
\end{equation}
for inequality $i<j-1$ and equality $i=\psi (j)-1$ imply $i<n.$ Now if $j+1\neq m,$ $j+1\neq \psi (k),$
then using Lemma \ref{indle}, we have 
$$
[u[k,i],u[j+1,m]]=[[u[k,i],[x_i,u[j+2,m]]]\stackrel{(\ref{jak3})}{=}[[u[k,i],x_i],u[j+2,m]]\stackrel{(\ref{esh})}{=}0,
$$
for $x_{j+1}=x_i.$
The exceptional equality $j+1=\psi (k)$ implies $k=\psi (j)-1=i.$ In this case, taking 
into account Lemma \ref{indle}, we have 
$$[x_i,u[j+1,m]]=[[x_i,[x_i,x_{i-1}]],u[j+3,m]]=0.$$
The exceptional equality $j+1=m$ implies $u[1+j,m]=x_m=x_i,$ for $\psi (j+1)=i.$
Hence relation (\ref{esh}) applies. The equality $[u[k,i],u[j+1,m]]=0$ is proved.

Assume $i\neq \psi (j)-1.$ Definition (\ref{mu1}) shows that 
$$p(u(k,i),u(j+1,m))\cdot p(u(j+1,m),u(k,i))=\mu _k^{m,i}(\mu _k^{j,i})^{-1}.$$
Using (\ref{mu2}) and (\ref{mu4}) we shall prove that $\mu _k^{m,i}=\mu _k^{j,i}.$
If $i=n,$ then $\mu _k^{m,i}=\mu _k^{j,i}=1.$ Let $i\neq n.$ If $m<\psi (k),$
then $\mu _k^{m,i}=q^{-2},$ for $i=\psi (m)-1$ is equivalent to $m=\psi (i)-1.$
Similarly $\mu _k^{j,i}=q^{-2},$ for $j\neq \psi (i)-1,$ and $j\leq m<\psi (k).$

If $m>\psi (k),$ and $i\neq \psi (k),$ then by (\ref{mu4}) we have $\mu _k^{m,i}=q^{-2},$
while $\mu _k^{j,i}=q^{-2}$ in both cases: if $j<\psi (k)$ by (\ref{mu2}), and if $j>\psi (k)$ by (\ref{mu4}). 
Finally, if $i=\psi (k),$ then $j>i=\psi (k),$ hence (\ref{mu4}) implies $\mu _k^{m,i}=\mu _k^{j,i}=q^{-4}.$

In order to  get $[u[j+1,m],u[k,i]]=0$ it remains to apply (\ref{cha}).
\end{proof}

\section{PBW-generators of the quantum Borel algebra}
\smallskip
\begin{proposition} If $q^3\neq 1,$ $q^4\neq 1,$ then
Values of the elements  $u[k,m],$ $k\leq m<\psi (k)$ form a set of PBW-generators
for the algebra $U_q^+({\mathfrak so}_{2n+1})$ over {\bf k}$[G].$ All heights are infinite.
\label{strB}
\end{proposition}
\begin{proof} By \cite[Theorem $B_n,$ p. 211]{Kh4} the set of PBW-generators 
(the values of hard super-letters, see Theorem \ref{BW}) consists of
$[u_{km}],$  $k\leq m\leq n,$ and $[w_{ks}],$ $1\leq k<s\leq n,$ where $[u_{km}],$
$[w_{ks}]$ are precisely the words $u(k,m),$ $u(k,\psi (s))$ with the standard alignment of brackets
(see Algorithm p. \pageref{algo}).
By conditional identity (\ref{ind}) we have $[u_{km}]=u[k,m]$ in $U_q^+({\mathfrak so}_{2n+1}).$
According to \cite[Lemma 7.8]{Kh4} the brackets in $[w_{ks}]$ 
are set by the following recurrence formulae:
\begin{equation}
\begin{matrix}
[w_{ks}]=[x_k[w_{k+1\, s}]], \hfill & \hbox{if  }1\leq k<s-1; \hfill \cr
[w_{k\, k+1}]=[[w_{k\, k+2}]x_{k+1}], \hfill & \hbox{if  }1\leq k<n,\hfill 
\end{matrix}
\label{wsk}
\end{equation}
where by definition $w_{k \, n+1}=u(k,n).$ We shall check the equality
$[w_{ks}]=u[k,\psi (s)]$ in $U_q^+({\mathfrak so}_{2n+1}).$

If $k=n-1,$ $s=n,$ then $w_{ks}=[[x_{n-1},x_n],x_n]=u[n-1,n+2].$

If  $k<s-1,$ then by (\ref{jak3}) we have 
$$
[x_k, \hbox{\Large [} u[k+1,n],u[n+1,\psi (s)]\hbox{\Large ]}]
=\hbox{\Large [}u[k,n], u[n+1,\psi (s)]\hbox{\Large ]},
$$ 
for $[x_k,x_t]=0, \ n+1\leq t\leq \psi (s).$ Thus, evident induction  applies due to  (\ref{ww1}).

If $s=k+1<n,$ then the second option of  (\ref{wsk}) is fulfilled. This allows us to apply already proved equality for $[w_{k\, k+2}].$ 
\end{proof}

If $q$ is not a root of 1, then the fourth statement of \cite[Theorem $B_n,$ p. 211]{Kh4}
shows that each skew-primitive element in $U_q^+({\mathfrak so}_{2n+1})$ is proportional to
either $x_i,$ $1\leq i\leq n,$ or $1-g,$ $g\in G.$ In particular $\xi (G\langle X\rangle ^{(2)})$ has no 
nonzero skew-primitive elements. At the same time due to the Heyneman-Radford theorem
\cite{HR}, \cite[Corollary 5.3]{Kh2}
every bi-ideal of a character Hopf algebra has a nonzero skew-primitive element. Therefore
Ker$\, \xi ={\bf \Lambda },$ while the subalgebra $A$ generated by values of 
$x_i,$ $1\leq i\leq n$ in $U_q^+({\mathfrak so}_{2n+1})$ has the shuffle representation 
given in Section 2.

If the multiplicative order of $q$ is finite, then by definition of 
$H=u_q^+({\mathfrak so}_{2n+1})$ we have Ker$\, \xi ={\bf \Lambda }.$
Hence the subalgebra $A$ generated by values of 
$x_i,$ $1\leq i\leq n$ in $u_q^+({\mathfrak so}_{2n+1})$ has the shuffle representation also.

Recall that by $(u(m,k))$ we denote the tensor $x_m\otimes x_{m-1}\otimes \cdots \otimes x_k$
considered as an element of $Sh_{\tau }(V).$
\begin{proposition}
Let $k\leq m\leq 2n.$ In the shuffle representation we have
\begin{equation}
u[k,m]=\alpha _k^m\cdot (u(m,k)), \ \
\alpha _k^m\stackrel{df}{=}\varepsilon_k^m (q^2-1)^{m-k}\cdot \prod _{k\leq i<j\leq m}p_{ij},
\label{shur}
\end{equation}
where 
\begin{equation}
\varepsilon _k^m =\left\{ \begin{matrix}
1,\hfill &\hbox{if } m\leq n, \hbox{or } k>n;\hfill \cr
q^{-1},\hfill &\hbox{if }  k\leq n<m,\ m\neq \psi (k);\hfill \cr
q^{-3},\hfill &\hbox{if } m=\psi (k).\hfill 
                             \end{matrix}
                    \right.
\label{eps}
\end{equation}
\label{shu}
\end{proposition}
\begin{proof} 
We use induction on $m-k.$ If $m=k,$ the equality reduces to $x_k=(x_k).$

a). Consider firstly the case $m<\psi (k).$ By the inductive supposition we have
$u[k,m-1]=\alpha _k^{m-1}\cdot (w),$ $w=u(m-1,k).$ Using (\ref{spro}) we may write
$$
u[k,m]=\alpha _k^{m-1}\{ (w)(x_m)-p(w,x_m)\cdot (x_m)(w)\}
$$
\begin{equation}
=\alpha _k^{m-1}\sum _{uv=w}\{ p(x_m,v)^{-1}-p(w,x_m)p(u,x_m)^{-1}\} (ux_mv).
\label{sum1}
\end{equation}
Since $w=uv,$ we have  $p(w,x_m)p(u,x_m)^{-1}=p(v,x_m).$ 

If $m\leq n,$  then relations (\ref{b1rell}) imply $p(v,x_m)p(x_m,v)=1$ with only one exception 
being $v=w.$
Hence sum (\ref{sum1}) has just one term. The coefficient at $(x_mw)=(u(m,k))$ equals 
$$\alpha _k^{m-1}p(w,x_m)(p(w,x_m)^{-1}p(x_m,w)^{-1}-1)=\alpha _k^{m-1}p(w,x_m)(q^2-1),$$
 which is required.

If $m=n+1,$ then still 
$p(v,x_m)p(x_m,v)=1$ with two exceptions being $v=w,$ and $v=u(n-1,k).$
In both cases $(ux_mv)$ equals $(u(m,k)).$ Hence the coefficient 
at $(u(m,k))$ in sum (\ref{sum1}) equals 
$$
p(x_n,u(k,n-1))^{-1}-p(u(k,n-1),x_n)+p(x_n,u(k,n))^{-1}-p(u(k,n),x_n)
$$
$$
=p(w,x_{n+1})\{ p_{n\,n-1}^{-1}p_{n-1\, n}^{-1}p_{nn}^{-1}-p_{nn}^{-1}
+p_{n\, n-1}^{-1}p_{nn}^{-1}p_{nn}^{-1}p_{n-1\, n}^{-1}-1\}.
$$
Due to (\ref{b1rel}), (\ref{b1rell}) we get  $\alpha _k^m=\alpha _k^{m-1}p(w,x_{n+1})(q^2-1)q^{-1},$
which is required.

Suppose that $m>n+1.$ In this case by definition $x_m=x_t,$ where $t=\psi (m)<\psi (n+1)=n.$
Let $v=u(s,k).$ If $s<t-1,$ then $v$ depends only on $x_i,$ $i<t-1,$ 
and relations (\ref{b1rel}), (\ref{b1rell}) imply $p(v,x_m)p(x_m,v)=1.$ If $s>t,$ $s\neq m-1,$
then $p(v,x_m)p(x_m,v)$ $=p_{t-1\, t}p_{tt}p_{t+1\, t}\cdot p_{t\, t-1}p_{tt}p_{t+1\, t}$ $=1.$ Hence in (\ref{sum1}) remains three terms with
$s=t-1,$ $s=t,$ and $s=m-1.$ If $v=u(t-1,k)$ or $v=u(t,k),$ then $(ux_mv)$
equals $(u(k,t)x_t^2u(t+1,m-1)),$ while the coefficient at this tensor
in sum (\ref{sum1}) is
$$
p(x_t, u(k,t-1))^{-1}-p(u(k,t-1),x_t)+p(x_t, u(k,t))^{-1}-p(u(k, t),x_t)
$$
$$
=p(u(k, t),x_t)\{ p_{t\, t-1}^{-1}p_{t-1\, t}^{-1}p_{tt}^{-1}-p_{tt}^{-1}
+p_{tt}^{-1}p_{t\, t-1}^{-1}p_{t-1\, t}^{-1}p_{tt}^{-1}-1\} =0.
$$
Thus in (\ref{sum1}) remains just one term with $v=u(m-1,k)$. It has the required coefficient:
$$
\alpha _k^m=\alpha _k^{m-1}(p(x_m,w)^{-1}-p(w,x_m))=\alpha _k^{m-1}p(w,x_m)(q^2-1).
$$

b). In perfect analogy we consider the case $m>\psi (k).$
By the inductive supposition we have
$u[k+1,m]=\alpha _{k+1}^{m}\cdot (w),$ $w=u(m,k+1).$ Using (\ref{spro}) we may write
$$
u[k,m]=\alpha _{k+1}^m\{ (x_k)(w)-p(x_k,w)\cdot (w)(x_k)\}
$$
\begin{equation}
=\alpha _{k+1}^{m}\sum _{uv=w}\{ p(u,x_k)^{-1}-p(x_k,u)\} (ux_kv).
\label{sum2}
\end{equation}

If $k>n,$  then $p(u,x_k)p(x_k,u)=1,$ unless $u=w.$
Hence (\ref{sum2}) has just one term, and the coefficient  equals 
$$\alpha _{k+1}^{m}p(x_k,w)(p(w,x_k)^{-1}p(x_k,w)^{-1}-1)=\alpha _{k+1}^{m}p(x_k,w)(q^2-1),$$
which is required.

If $k=n,$ then $p(u,x_k)p(x_k,u)=1$ with two exceptions being $u=w,$ and $u=u(m,n+2).$
In both cases $(ux_kv)$ equals $(u(m,k)),$ while the coefficient takes up the form 
$$
p(w,x_n)^{-1}-p(x_n,w)+p(u(m,n+2),x_n)^{-1}-p(x_n,u(m,n+2))
$$
$$
=p(x_n,w)\{ p_{n\,n-1}^{-1}p_{n-1\, n}^{-1}p_{nn}^{-2}-1
+p_{n\, n-1}^{-1}p_{n-1\, n}^{-1}p_{nn}^{-1}-p_{nn}^{-1}\}.
$$
Due to relations (\ref{b1rel}), (\ref{b1rell}) we get 
$\alpha _n^m=\alpha _{n+1}^{m}p(x_n,w)(q^2-1)q^{-1},$
which is required.

Suppose that $k<n.$ In this case $x_k=x_t$ with $m>t\stackrel{df}{=} \psi (k)>\psi (n)=n+1.$ Let $u=u(m,s).$ 
If $s>t,$ then $u$ depends only on $x_i,$ $i<k-1,$ 
and relations (\ref{b1rel}), (\ref{b1rell}) imply $p(x_k,u)p(u,x_k)=1.$ If $s<t-1,$ $s\neq k+1,$
then $p(x_k,u)p(u,x_k)$ $=p_{k-1\, k}p_{kk}p_{k+1\, k}\cdot p_{k\, k-1}p_{kk}p_{k+1\, k}$ $=1.$
Hence in (\ref{sum2}) remains three terms with
$s=t,$ $s=t+1,$ and $s=k+1.$ If $u=u(m, t)$ or $u=u(m, t+1),$ then $ux_kv$ 
$=u(m,t+1)x_k^2u(t-1,k),$ while the coefficient at the corresponding  tensor is
$$
p(u(m,t+1),x_k)^{-1}-p(x_k,u(m,t+1))+p(u(m,t),x_k)^{-1}-p(x_k,u(m, t))
$$
$$
 =p(x_k,u(m, t+1))\{ p_{k-1\, k}^{-1}p_{k\, k-1}^{-1}-1+
p_{kk}^{-1}p_{k-1\, k}^{-1}p_{k\, k-1}^{-1}-p_{kk}\} =0.
$$
Thus in (\ref{sum1}) remains just one term, and
$$
\alpha _k^m=\alpha _{k+1}^m(p(w,x_k)^{-1}-p(x_k,w))=\alpha _{k+1}^mp(x_k,w)(q^2-1).
$$

c). Let us consider the remaining case $m=\psi (k).$ In this case $x_m=x_k.$
If $k=n,$ $m=n+1,$ then 
$u[n,n+1]=-p_{nn}^{-1}[x_n,x_n]=(1-q^{-1})x_n^2,$ while in the shuffle representation
we have $(x_n)(x_n)=(1+q^{-1})(x_nx_n).$ Hence $u[n,n+1]=(1-q^{-2})(x_{n+1}x_n),$
which is required: $(1-q^{-2})=q^{-3}\cdot (q^{2}-1)\cdot p_{nn}.$

If $k<n,$ we put $u=u[n+1,m],$ $v=x_k,$ $w=u[k+1,n].$
By definition (\ref{ww}) we have $u[k,m]=\beta [u,[v,w]],$ where $\beta =-p(u(n+1,m),u(k,n))^{-1};$
that is, $\beta =-p_{u, vw}^{-1}.$ 
Since $u[n+1,m]=[u[n+1,m-2],[x_{k+1},x_k]],$ conditional identity 
(\ref{jak3}) implies $[u,v]=[u[n+1,m-2],[[x_{k+1},x_k],x_k]]=0.$ Thus $[[u,v],w]=0,$ 
and formula (\ref{jak2}) yields
\begin{equation}
\beta ^{-1}u[k,m]=p_{uv}x_k\cdot [u,w]-p_{vu}[u,w]\cdot x_k.
\label{kk}
\end{equation}
Formula  (\ref{ww1}) implies $\beta_1[u,w]=u[k+1,m]$ with $\beta_1=-p_{uw}^{-1}.$
Hence considered above case b) allows us to find the shuffle representation
$[u,w]=\alpha \cdot (z)$ with $z=u(m,k+1),$ and 
$\alpha =-p_{uw}\alpha _{k+1}^m.$
By (\ref{spro}) the shuffle representation of  the right hand side of (\ref{kk}) is
$$
\alpha \sum _{sy=u(m,k+1)}
\left( p_{uv}p(s,x_k)^{-1}-p_{vw} p(x_k,y)^{-1}\right) \cdot (sx_ky) 
$$ 
We have $\beta \alpha =-\beta p_{uw}\alpha _{k+1}^m$ $=p_{uv}^{-1}\alpha _{k+1}^m,$
and 
$$
p_{uv}p_{vu}=p_{k+1\,k}p_{kk}p_{k\, k+1}p_{kk}=q^2
$$
since $k<n.$ Therefore we get
\begin{equation}
u[k,m]=\alpha _{k+1}^m\sum _{sy=u(m,k+1)}
\left( p(s,x_k)^{-1}-q^{-2}p(x_k,s)\right) \cdot (sx_ky).  
\label{uje}
\end{equation}
If $s\notin \{ \emptyset, x_m, z=u(m,k+1)\},$ then
$
p(s,x_k)p(x_k,s)=p_{k+1\,k}p_{kk}p_{k\, k+1}p_{kk}=q^2,
$
that is in (\ref{uje}) remains just three terms. If $s=\emptyset $ or $s=x_m,$
then $(sx_ky)=(x_kz)$ since $x_m=x_k.$ Hence the coefficient at $(x_kz)$ in (\ref{uje})
equals $1-q^{-2}+p_{kk}^{-1}-q^{-2}p_{kk}=0.$ Thus in (\ref{uje}) remains just one term
with the coefficient
$$
\alpha _{k+1}^m(p(z,x_k)^{-1}-q^{-2}p(x_k,z))=
\alpha _{k+1}^mp(x_k,z)q^{-2}(q^{2}-1)=\alpha _k^m
$$
since $p(z,x_k)\cdot p(x_k,z)=p_{kk}p_{k+1\, k}p_{k+1\, k}\cdot  p_{kk}p_{k\, k+1}p_{k\, k+1}=1.$
\end{proof} 
\begin{theorem} In $U_q^+({\mathfrak so}_{2n+1})$
the coproduct on the elements $u[k,m],$ $k\leq m\leq 2n$ 
 has the following explicit form
\begin{equation}
\Delta (u[k,m])=u[k,m]\otimes 1+g_kg_{k+1}\cdots g_m\otimes u[k,m]
\label{co}
\end{equation}
$$
+\sum _{i=k}^{m-1}\tau _i(1-q^{-2})g_kg_{k+1}\cdots g_i\, u[i+1,m]\otimes u[k,i],
$$
where $\tau _i=1$ with only one exception being $\tau _n=q.$
\label{cos}
\end{theorem}
\begin{proof}
Formulae (\ref{shur}), (\ref{bcopro}), and (\ref{copro}) show that the coproduct has 
form (\ref{co}), where 
$\tau _i(1-q^{-2})=\alpha _k^m(\alpha _k^i\alpha _{i+1}^m)^{-1}\chi ^{u(i+1,m)}(g_kg_{k+1}\cdots g_i).$
 We have 
$$
\left( \prod _{k\leq a<b\leq i}p_{ab}\prod _{i+1\leq a<b\leq m}
p_{ab}\right) ^{-1}\prod _{k\leq a<b\leq m}p_{ab}=p(u(k,i),u(i+1),m).
$$ 
Therefore definition of $\mu _k^m$ given in  (\ref{mu1}) and definition of $\alpha _k^m$ given in (\ref{shur})
imply $\tau _i(1-q^{-2})$ $=\varepsilon _k^m(\varepsilon_k^i\varepsilon_{i+1}^m)^{-1}
(q^2-1)\mu _k^{m,i};$ that is, 
$\tau _i=\varepsilon _k^m(\varepsilon_k^i\varepsilon_{i+1}^m)^{-1}q^2\mu _k^{m,i}.$ 
By (\ref{mu23}) we have $\mu _k^{m,i}=\sigma _k^m(\sigma _k^i\sigma _{i+1}^m)^{-1}.$
Using (\ref{mu21}) and (\ref{eps}) we see that
\begin{equation}
\varepsilon _k^m\sigma _k^m=\left\{ \begin{matrix}
q^2,\hfill &\hbox{if } m<n \hbox{ or } k>n+1 ;\hfill \cr
q,\hfill &\hbox{otherwise}.\hfill 
                             \end{matrix}
                    \right.
\label{tau1}
\end{equation}
Now it is easy to check that   the
$\tau $'s have the following elegant form
\begin{equation}
\tau_i=\varepsilon _k^m\sigma _k^m(\varepsilon _k^i\sigma _k^i)^{-1}(\varepsilon _{i+1}^m\sigma _{i+1}^m)^{-1} q^2
=\left\{ \begin{matrix}
q,\hfill &\hbox{if } i=n;\hfill \cr
1,\hfill &\hbox{otherwise}.\hfill 
                             \end{matrix}
                    \right.
\label{tau}
\end{equation}
Interestingly, the coproduct formula
differs from that in $U_{q^2}^+({\mathfrak sl}_{2n+1})$ by just one term, see formula (3.3) in \cite{KA}.
\end{proof}

\smallskip
Now we are going to find PBW-generators for $u_q^+({\mathfrak so}_{2n+1}).$
To do this we need more relations in $U_q^+({\mathfrak so}_{2n+1}).$
\begin{lemma}
If $k\leq m<\psi (k),$ then in the algebra $U_q^+({\mathfrak so}_{2n+1})$ we have
\begin{equation}
\left[ u[k,m],[u[k,m],u[k+1,m]]\right] =0.
\label{reldd}
\end{equation}
\label{reld}
\end{lemma}
\begin{proof}
Suppose, first, that $m<\psi (k)-1.$ In this case both words $u(k,m)$ and $u(k+1,m)$ are standard.
The standard alignment of brackets for these words is defined by (\ref{wsk}).
However in Proposition \ref{strB} we have seen that 
$[u(k,m)]=u[k,m],$ and hence also $[u(k+1,m)]=u[k+1,m]$ in the algebra 
$U_q^+({\mathfrak so}_{2n+1}).$

The word $w=u(k,m)u(k,m)u(k+1,m)$ is standard. Algorithm given on p. \pageref{algo}
 shows that the standard alignment of brackets is precisely
$$\left[ [u(k,m)],[[u(k,m)],[u(k+1,m)]]\right] .$$
 Hence the value of the super-word $[w]$
in $U_q^+({\mathfrak so}_{2n+1})$ equals the left hand side of (\ref{reldd}). 

By Proposition \ref{strB} all hard super-letters in $U_q^+({\mathfrak so}_{2n+1})$
are $[u(k,m)],$ $k\leq m<\psi (k).$ Hence $[w]$ is not hard. The multiple use of Definition  \ref{tv1}
shows that the value of $[w]$ is a linear combination of the values of super-words in smaller than 
$[w]$ hard super-letters. Since $U_q^+({\mathfrak so}_{2n+1})$ is homogeneous,
each of the super-words in that decomposition has two hard super-letters smaller than 
$[w]$ and of degree 1 in $x_k$ 
(if a hard super-letter $[u(r,s)]$ is of degree 2 in $x_k,$ then $r<k$ and $u(r,s)>w$).
At the same time all such hard super-letters are  
$[u(k,m+1)],$ $[u(k,m+2)], \ldots ,$ $[u(k,2n-k)].$ Each of them has degree 2 in $x_{m+1}$
if $m\geq n,$ and at least 1 if $m<n.$
Hence the super-word has degree in $x_{m+1}$ at least 4 if $m\geq n,$ and
 at least 1 if $m<n.$ However $w$ is of degree
3 in $x_{m+1}$ if $m\geq n,$ and is independent of $x_{m+1}$ if $m<n.$ 
Therefore the decomposition is empty, and $[w]=0.$ 

Let, then,  $m=\psi (k)-1.$
In this case $u(k+1,m)$ is not standard and we may not apply the above arguments.
Nevertheless we shall prove in a similar way that 
$[u[k,2n-k], x_t]=0,$ $k<t\leq n.$ This would imply both $[u[k,2n-k],u[k+1,2n-k] ]=0$ and (\ref{reldd}).

If $k+1<t<n,$ then Lemma \ref{nule} and Lemma \ref{nulm}
imply 
$$[u[k,n],x_t]=[u[n+1,2n-k],x_t]=0.$$
Due to Corollary \ref{rww} we have
$[u[k,2n-k], x_t]=0.$

If $t=k+1,$  we consider the word $v=u(k,2n-k)x_{k+1}.$ This is a standard word, and the 
standard alignment of brackets is $[v]=[[u(k,2n-k)]x_{k+1}].$ Therefore the value of the 
super-letter $[v]$ equals $[u[k,2n-k], x_{k+1}].$ At the same time $[v]$ does not belong to the set of 
PBW-generators; that is, this is not hard. The multiple use of Definition  \ref{tv1}
shows that the value of $[v]$ is a linear combination of the values of super-words in smaller than 
$[v]$ hard super-letters. Each of the super-words in that decomposition 
has a hard super-letter smaller than $[v]$ and of degree 1 in $x_k.$ However there are no such 
super-letters. Thus the decomposition is empty and $[v]=0.$

Let $t=n.$ If $k=n-1,$ then $[u[k,2n-k], x_n]=[[[x_{n-1},x_n],x_n],x_n]=0$ due to (\ref{relbl}).
If $k=n-2,$ we consider the word $u=u(k,2n-k) x_n=x_{n-2}x_{n-1}x_nx_nx_{n-1}x_n.$
This is a standard word, while the super-letter $[u]$ is not hard. Again, there do not exists
a hard super-letter smaller than $[u]$ and of degree 1 in $x_{n-2}.$ Hence $[u]=0$
in $U_q^+({\mathfrak so}_{2n+1}).$ The standard alignment of brackets
is $[[x_{n-2}x_{n-1}x_nx_n][x_{n-1}x_n]].$ Hence we get
$$
[[x_{n-2},[[x_{n-1},x_n],x_n]], [x_{n-1},x_n]]=0.
$$
At the same time $[x_{n-2},x_n]=0$ and $[[[x_{n-1},x_n],x_n],x_n]=0$ imply
$$[[x_{n-2},[[x_{n-1},x_n],x_n]], x_n]=0.$$
Conditional identity (\ref{jak3}) yields
$$
[[x_{n-2},[[x_{n-1},x_n],x_n]], [x_{n-1},x_n]]=[[[x_{n-2},[[x_{n-1},x_n],x_n]], x_{n-1}],x_n],
$$
which is required, for $[u[n-2,n+2],x_n]=[[[x_{n-2},[[x_{n-1},x_n],x_n]], x_{n-1}],x_n].$

Finally, suppose that $k<n-2.$ Denote $u_1=u[k,n-3],$ $v_1=u[n+3,2n-k],$ $w_1=u[n-2,n+2].$
We have already proved that $[w_1,x_n]=0.$ The second group of relations (\ref{relb})
implies $[u_1,x_n]=0,$ $[v_1,x_n]=0.$ At the same time, due to Proposition \ref{ins2},
we have 
$u[k,2n-k]=[u[k,n+2],v_1]$ and $u[k,n+2]=[u_1,w_1];$ that is $u[k,2n-k]=[[u_1,w_1],v_1].$
This certainly implies the required relation $[u[k,2n-k],x_n]=0.$
\end{proof}
\begin{proposition}
If the multiplicative order $t$ of $q$ is finite, $t>4,$ then the values of $u[k,m],$ 
$k\leq m<\psi (k)$ form a set of PBW-generators for $u_q^+({\mathfrak so}_{2n+1})$ 
over {\bf k}$[G].$ The height $h$ of $u[k,m]$ equals $t$ if $m=n$ or $t$ is odd.
If $m\neq n$ and $t$ is even, then $h=t/2.$ In all cases $u[k,m]^h=0$
in $u_q^+({\mathfrak so}_{2n+1}).$
\label{strBu}
\end{proposition}
\begin{proof}
First we note that Definition \ref{tv1} implies that  a non-hard super-letter in $U_q^+({\mathfrak so}_{2n+1})$
 is still non-hard in $u_q^+({\mathfrak so}_{2n+1}).$ Hence all hard super-letters in 
$u_q^+({\mathfrak so}_{2n+1})$  are in the list $u[k,m],$ $k\leq m<\psi (k).$
If, next, $u[k,m]$  is not  hard in $u_q^+({\mathfrak so}_{2n+1}),$ then by the multiple use of
Definition \ref{tv1} the value of $u[k,m]$ is a linear combination of super-words in hard super-letters 
smaller than given $u[k,m].$ Since $u_q^+({\mathfrak so}_{2n+1})$ is homogeneous,
each of the super-words in that decomposition has a hard super-letter smaller than 
$u[k,m]$ and of degree 1 in $x_k.$ 
At the same time all such hard super-letters are  in the list
$[u(k,m+1)],$ $[u(k,m+2)], \ldots ,$ $[u(k,2n-k)].$ Each of them has degree 2 in $x_{m+1}$
if $m\geq n,$ and at least 1 if $m<n.$
Hence the super-word has degree at least 2 if $m\geq n,$ and
 at least 1 if $m<n.$ However $u[k,m]$ is of degree
1 in $x_{m+1}$ if $m\geq n,$ and is independent of $x_{m+1}$ if $m<n.$ 
Therefore the decomposition is empty, and $u[k,m]=0.$ This contradicts Proposition \ref{shu},
for $(u(m,k))\neq 0$ in the shuffle algebra.

Denote for short $u=u[k,m].$ Equation (\ref{mu21}) implies 
$p_{uu}=q$ if $m=n$ and $p_{uu}=q^2$ otherwise (recall that now $m<\psi (k)$). By Definition \ref{h1}
the minimal possible value for the height is precisely the $h$ given in the proposition.
It remains to show that $u^h=0$ in $u_q^+({\mathfrak so}_{2n+1}).$
By Lemma \ref{MS} it suffices to prove that $\partial _i(u^h)=0,$ $1\leq i\leq n.$
Lemma \ref{dert} yields 
$$\partial_i(u^h)=p(u,x_i)^{h-1}\underbrace{[u,[u,\ldots [u}_{h-1},\partial _i(u)]\ldots ]].$$
Coproduct formula (\ref{co}) with (\ref{calc}) implies
\begin{equation}
\partial _i(u)=\left\{ \begin{matrix}
(1-q^{-2})\tau _ku[k+1,m],\hfill & \hbox{ if } i=k<m;\hfill \cr
0, \hfill & \hbox{ if } i\neq k;\hfill \cr
1, \hfill & \hbox{ if } i=k=m.\hfill
\end{matrix} \right.
\label{pdee}
\end{equation}
At the same time Lemma \ref{reld} provides the relation
$\left[u,[u,u[k+1,m]]\right] =0$ in $U_q^+({\mathfrak so}_{2n+1}),$
and hence in $u_q^+({\mathfrak so}_{2n+1})$ as well. 
Since always $h>2,$ we get the required equalities $\partial_i(u^h)=0,$ $1\leq i\leq n.$
\end{proof}

\noindent
{\bf Remark}. To prove (\ref{co}) we have used the shuffle representation. Therefore 
if $q$ has a finite multiplicative order, then
(\ref{co}) is proved only for $u_q^+({\mathfrak so}_{2n+1}).$
However we have seen that the kernel of the natural homomorphism 
$U_q^+({\mathfrak so}_{2n+1})\rightarrow u_q^+({\mathfrak so}_{2n+1})$
is generated by the elements $u[k,m]^h,$ $k\leq m<\psi (k).$
Degree of  $u[k,m]^h$ in a given
$x_i$ is either zero or greater than 2. At the same time all tensors in (\ref{co}) have degree
less than or equal to 2 in each variable. Therefore (\ref{co}), and hence (\ref{pdee}), are valid in 
$U_q^+({\mathfrak so}_{2n+1})$ provided that $q$ has a finite multiplicative order $t>4$
as well. 

\section{PBW-generators for right coideal subalgebras}

In what follows we denote by $A_{k+1},$ $k<n$ a subalgebra of 
$U_q^+({\mathfrak so}_{2n+1})$ or $u_q^+({\mathfrak so}_{2n+1})$
generated by $x_i,$ $k<i\leq n,$ 
respectively $A$ is a subalgebra generated by all $x_i,$ $1\leq i\leq n.$
Of course ${\bf k}[g_{k+1},\ldots g_n]A_{k+1}$
may be identified with $U_q^+({\mathfrak so}_{2(n-k)+1})$ or $u_q^+({\mathfrak so}_{2(n-k)+1}).$

Suppose that a homogeneous element $f\in {\bf k}\langle X\rangle $ is linear in the maximal 
letter $x_k,$ $1\leq k\leq n$ that it depends on: deg$_k(f)=1,$ deg$_i(f)=0,$
$i<k.$ Then in the decomposition of $a= \xi (f)$ in the PBW-basis defined in Proposition \ref{strB}
or Proposition \ref{strBu}
each summand has just one PBW-generator that depends on $x_k$
since $U_q^+({\mathfrak so}_{2n+1})$ and $u_q^+({\mathfrak so}_{2n+1})$ are homogeneous in each $x_i.$
Moreover this PBW-generator, considered as a super-letter, starts by $x_k,$ 
hence it is the biggest super-letter of the summand. 
In particular this super-letter is located at the end of the basis super-word;
that is, the PBW-decomposition  takes up the form
\begin{equation}
a=\sum_{i=k}^{2n-k} F_iu[k,i], \ \ F_i\in A_{k+1}.
\label{sp1}
\end{equation}
\begin{definition} \rm
The set Sp$  (a)$ of all $i$ such tat in (\ref{sp1}) we have $F_i\neq 0$ is called the {\it spectrum} of $a.$
\label{spe}
\end{definition}

Let $S$ be a set of integer numbers from the interval $[1,2n].$
We define a polynomial $\Phi^{S}(k,m),$ $1\leq k\leq m\leq 2n$
by induction on the number $r$ of elements
in the set $S\cap [k,m-1]=\{ s_1,s_2,\ldots ,s_r\} ,$ $k\leq s_1<s_2<\ldots <s_r<m$
as follows:
\begin{equation}
\Phi^{S}(k,m)=u[k,m]-(1-q^{-2})\sum_{i=1}^r \alpha _{km}^{s_i} \, 
\Phi^{S}(1+s_i,m)u[k,s_i]
\label{dhs}
\end{equation}
where $\alpha _{km}^{s}=\tau _{s}p(u(1+s,m),u(k,s))^{-1},$ 
while the $\tau $'s was defined in (\ref{tau}).

\smallskip

We display the element  $\Phi ^{S}(k,m)$
schematically as a sequence of black and white points labeled by the numbers
$k-1,$ $k,$ $k+1, \ldots $ $m-1,$ $m,$ where the first point is always white, and
the last one is always black, while an intermediate point labeled by $i$ is black if and only if 
$i\in S:$  
\begin{equation}
 \stackrel{k-1}{\circ } \ \ \stackrel{k}{\circ } \ \ \stackrel{k+1}{\circ } 
\ \ \stackrel{k+2}{\bullet }\ \ \ \stackrel{k+3}{\circ }\ \cdots
\ \ \stackrel{m-2}{\bullet } \ \ \stackrel{m-1}{\circ }\ \ \stackrel{m}{\bullet }
\label{grb}
\end{equation}
Sometimes, if $k\leq n<m,$ it is more convenient  to display the element
$\Phi ^{S}(k,m)$ in two lines putting the points labeled by indices 
$i,\psi (i)$ that define the same variable $x_i=x_{\psi (i)}$ in one column:
\begin{equation} 
\begin{matrix}
 \ \ \ \ \ \ \ & \stackrel{m}{\bullet } \ \cdots & \bullet  & \stackrel{\psi (i)}{\circ }   \ \cdots & \stackrel{n+1}{\bullet }  \cr 
 \stackrel{k-1}\circ \ \circ  \ \cdots \ & \stackrel{\psi (m)}{\circ } \ \cdots & \bullet  & \stackrel{i}{\bullet } \ \cdots & \stackrel{n}{\circ }   
\end{matrix}
\label{grb1}
\end{equation}
Below, to illustrate the notion of a regular set, we shall need a {\it shifted representation} 
that appears from (\ref{grb1}) by shifting the upper line to the left by one step and putting
the colored point  labeled by $n,$ if any, to the vacant position (so that this point appears twice 
in the shifted scheme):
\begin{equation} 
\begin{matrix}
 \ \ \ \ \ \ \ & \stackrel{m}{\bullet } \ \cdots & \circ  
& \stackrel{n+i}{\circ }   \ \cdots & \stackrel{n+1}{\bullet } & \stackrel{n}{\circ } \Leftarrow \cr 
 \stackrel{k-1}\circ \ \circ  \ \cdots \ & \stackrel{\psi (m)-1}{\bullet } \ \cdots &
 \bullet  & \stackrel{n-i}{\bullet } \ \cdots & \stackrel{n-1}{\circ }& \stackrel{n}{\circ } \hfill   
\end{matrix}
\label{grb2}
\end{equation}

\smallskip
If $k\leq m<\psi (k),$ then definition (\ref{dhs}) shows that the spectrum of 
$\Phi^{S}(k,m)$ is contained in $S\cup \{ m\},$ while its leading term is $u[k,m].$
However if $m\geq \psi (k),$ then Eq.  (\ref{dhs}) do not provide sufficient information
even for the immediate conclusion that $\Phi^{S}(k,m)\neq 0.$ In particular 
some of the factors $\Phi^{S}(1+s_i,m)$ in (\ref{dhs}) may be zero even if
$k\leq m<\psi (k).$  Hence {\it a priori}
the spectrum of $\Phi^{S}(k,m), $ $k\leq m<\psi (k)$ 
may be a proper subset of $S\cup \{ m\}.$

\smallskip 
By ${\pi }_{kl},$ $1\leq k\leq l<\psi (k)$ we denote a natural projection of $U_q^+({\mathfrak so}_{2n+1})$
or $u_q^+({\mathfrak so}_{2n+1})$ onto {\bf k}$u[k,l]$
with respect to the PBW-basis defined in Proposition \ref{strB} or \ref{strBu} respectively.

\begin{lemma} 
If $a\in A_{k+1},$ then ${\pi }_{kl}(au[k,i])=0,$ $k\leq i<\psi (k)$ unless $a\in {\bf k},$ $i=l.$
\label{las1}
\end{lemma}
\begin{proof} The PBW-decomposition $\tilde{a}$ of $a$
in basis defined in Proposition \ref{strB} or \ref{strBu}
involves only PBW-generators that belong to $A_{k+1}.$
All of them are less than $u[k,i].$ Hence the PBW-decomposition of
$au[k,i]$ is $\tilde{a}u[k,i].$ We have ${\pi }_{kl}(\tilde{a}u[k,i])\neq 0$
only if $\tilde{a}\in {\bf k},$ $i=l.$
\end{proof}
\begin{lemma} If $a\in A_{k+1},$ $k\leq l<\psi (k),$ then
\begin{equation}
\Delta (au[k,i])\cdot ({\rm id}\otimes {\pi }_{kl})=\left\{ 
\begin{matrix}
0,\hfill & \hbox{if } i<l; \hfill \cr
ag_{kl}\otimes u[k,l],\hfill & \hbox{if } i=l; \hfill \cr
\tau_l(1-q^{-2})a\, g_{kl}u[l+1,i]\otimes u[k,l],\hfill & \hbox{if } i>l, \hfill \cr
\end{matrix}
\right.
\label{dif}
\end{equation}
where by definition $g_{kl}=g(u[k,l])=g_kg_{k+1}\cdots g_l.$
\label{las2}
\end{lemma}
\begin{proof}
By means of (\ref{co}) we have $\Delta (au[k,i])=\sum\limits_{(a),\, j}a^{(1)}\alpha_jg_{kj}u[j+1,i]
\otimes a^{(2)}u[k,j]$ for suitable $\alpha_j\in {\bf k}.$ By the above lemma we 
get ${\pi }_{kl}(a^{(2)}u[k,j])=0$ unless $a^{(2)}\in {\bf k},$ $i=l.$ 
It remains to apply explicit formula (\ref{co}).
\end{proof}
\begin{lemma} If $k\leq l<m<\psi (k),$ then
$$
\Delta (\Phi^{S}(k,m))\cdot ({\rm id}\otimes {\pi }_{kl})=\left\{ \begin{matrix}
0,\hfill & \hbox{if } l\in S;\hfill \cr
\tau_l(1-q^{-2})g_{kl}\Phi^{S}(1+l,m)\otimes u[k,l],\hfill & \hbox{if } l\notin S.\hfill 
\end{matrix}
\right.
$$
\label{las3}
\end{lemma}
\begin{proof}
Let us apply $\Delta ({\rm id}\otimes \pi _{kl})$ to (\ref{dhs}). 
Since $a_i\stackrel{df}{=}\Phi^{S}(1+s_i,m)\in A_{k+1},$
we may use Lemma \ref{las2}. 
We have $a_i\, g_{kl}=\chi^{a_i}(g_{kl})g_{kl}\, a_i,$ $\chi^{a_i}(g_{kl})=p(u(1+s_i,m),u(k,l)).$
Thus, if $s_i>l,$ then $\alpha _{km}^{s_i}\chi^{a_i}(g_{kl})=\alpha _{1+l\, m}^{s_i},$
while if $s_i=l,$ then $\alpha _{km}^{l}\chi^{a_l}(g_{kl})=\tau _l.$ Now (\ref{dif}) implies the required
relation. \end{proof}
\begin{lemma} Let $k\leq l<m<\psi (k),$ and $a\in A_{k+1}$
is a nonzero homogeneous element with $D(a)=D(u(1+l,m)).$ Denote by $\nu_a$
any homogeneous projection $\nu_a:U_q^+({\mathfrak so}_{2n+1})\rightarrow a{\bf k}.$
If $D(b)=D(u(1+i,m)),$ then we have  
$$
\Delta (bu[k,i])\cdot ({\rm id}\otimes \nu_a )=
\left\{ 
\begin{matrix}
0,\hfill & \hbox{if } l<i<m; \hfill \cr
g_au[k,l]\otimes a,\hfill & \hbox{if } i=l,\, b=a; \hfill \cr
g_ab^{\prime }u[k,i]\otimes a,\hfill & \hbox{if } i<l. \hfill \cr
\end{matrix}
\right.
$$
\label{las4}
\end{lemma}
\begin{proof}
All right hand side components of the tensors in (\ref{co}) depend on $x_k$
with the only exception for the first summand. Since $\nu_a$ kills all elements 
with positive degree in $x_k,$ we have 
\begin{equation}
\Delta (bu[k,i])\cdot ({\rm id}\otimes \nu_a)=\sum_{(b)}b^{(1)}u[k,i]\otimes \nu_a(b^{(2)}).
\label{suk}
\end{equation}

If $l<i<m,$ then $D(b^{(2)})\leq D(b)<D(a),$ hence $\nu_a(b^{(2)})=0.$

If $b=a,$ $i=s,$ then $D(b^{(2)})=D(a)$ only if $b^{(1)}=g_a,$ $b^{(2)}=a.$

If $i<l,$ then (\ref{suk}) provides the third option given in the lemma.
\end{proof}
\begin{proposition} 
If a right coideal subalgebra {\bf U}$\, \supseteq {\bf k}[G]$ of $U_q^+({\mathfrak so}_{2n+1})$
or $u_q^+({\mathfrak so}_{2n+1})$
contains a homogeneous element $c\in A$ with the leading term $u[k,m],$ $k\leq m<\psi (k),$ then 
$\Phi ^{S}(k,m)\in\, ${\bf U} for a suitable subset $S$ of the spectrum of $c.$
\label{phib}
\end{proposition}
\begin{proof}

Every summand of the decomposition of $c$ in the PBW-basis defined in Proposition \ref{strB} or \ref{strBu}
has just one PBW-generator that depends on $x_k,$
for $U_q^+({\mathfrak so}_{2n+1})$ and $u_q^+({\mathfrak so}_{2n+1})$ are homogeneous in each $x_i.$
Moreover this PBW-generator, considered as a super-letter, starts by $x_k,$ 
and hence it is the biggest super-letter of the summand. 
In particular that super-letter is located at the end of the basis super-word;
that is, the PBW-decomposition  takes up the form
\begin{equation}
c=u[k,m]+\sum_{i=k}^{m-1} F_iu[k,i], \ \ F_i\in A_{k+1}, \ k\leq i<m.
\label{sp1i}
\end{equation}
By definition $i$ belongs to the spectrum, Sp$(a),$ of $a$ if and only if $F_i\neq 0.$
We may rewrite this  representation in the following way:
\begin{equation}
\Phi^{S_{t}}(k,m)+\sum_{i\in {\rm Sp}(a), \ i<t}F_iu[k,i]\in \, \hbox{\bf U},
\label{pitb1}
\end{equation}
where $t=m,$ and by definition $S_{m}=\emptyset .$ We shall prove that relation (\ref{pitb1})
with a given $t,$ $k<t\leq m,$ $S_t\subseteq {\rm Sp}(a),$ and $t\leq \inf S_t$ 
implies a relation of the same type with $t\leftarrow l,$ $S_{l}=S_t\cup \{ l\} ,$
where $l,$ as above, is the maximal $i$ in (\ref{pitb1}) such that $F_i\neq 0.$
Since certainly $l<t,$
by downward induction this will imply (\ref{pitb1}) with $t=k,$ $S=S_k\subseteq {\rm Sp}(a):$
\begin{equation}
\Phi ^{S}(k,m)\in \, \hbox{\bf U}.
\label{pitb2}
\end{equation}

Let us apply $\Delta \cdot ({\rm id}\otimes {\pi }_{kl})$ to (\ref{pitb1}),
where ${\pi }_{kl}$ is the projection onto {\bf k}$u[k,l],$ and $l$ is the maximal $i$
in (\ref{pitb1}) with $F_i\neq 0.$ By Lemma \ref{las2} we have
$\Delta (F_iu[k,i]) \cdot ({\rm id}\otimes {\pi }_{kl})=0,$ if $i<l,$ while
$\Delta (F_lu[k,l]) \cdot ({\rm id}\otimes {\pi }_{kl})=F_lg_{kl}\otimes [k,l].$
Lemma \ref{las3} implies
$\Delta (\Phi^{S_{t}}(k,m)) \cdot ({\rm id}\otimes {\pi }_{kl})=
\tau_l(1-q^{-2})g_{kl}\Phi^{S_{t}}(1+l,m)\otimes u[k,l].$
Since $U$ is a right coideal subalgebra that contains all group-like elements,
we get
\begin{equation}
F_l+\chi^{F_l}(g_{kl})^{-1}\tau_l(1-q^{-2})\Phi ^{S_t}(1+l,m)=v\in \, \hbox{\bf U}.
\label{alg1}
\end{equation}

Further, consider any homogeneous projection $\nu_a$ with $a=F_l.$
Let us apply $\Delta \cdot ({\rm id}\otimes \nu_a )$ to (\ref{pitb1}).
Since $l<\inf S_t,$ Lemma \ref{las4} and definition (\ref{dhs})
imply $\Delta (\Phi^{S_{t}}(k,m)) \cdot ({\rm id}\otimes \nu_a )=0.$
Lemma \ref{las4} shows also that
$\Delta (F_lu[k,l]) \cdot ({\rm id}\otimes \nu_a )$ $=g_au[k,l]\otimes a,$
while 
$\Delta (F_iu[k,i]) \cdot ({\rm id}\otimes \nu_a )$ $=g_aA_i^{\prime }u[k,i]\otimes a,$
$i<l.$ Hence we arrive to the relation
\begin{equation}
u[k,l]+\sum_{i\in {\rm Sp}(a), \ i<l}F_i^{\prime }u[k,i]=w\in \, \hbox{\bf U}.
\label{alg2}
\end{equation}
Relations (\ref{alg1}), (\ref{alg2}) imply 
$$
F_lu[k,l]=vw-\sum_{i\in {\rm Sp}(a), \ i<l}vF_i^{\prime }u[k,i]
-\chi^{F_l}(g_{kl})^{-1}\tau_l(1-q^{-2})\Phi ^{S_t}(1+l,m)\cdot u[k,l].
$$
This allows one to replace $F_lu[k,l]$ in (\ref{pitb1}). Since according to definition  
(\ref{dhs}) we have $\Phi ^{S_t}(k,m)
-\chi^{F_l}(g_{kl})^{-1}\tau_l(1-q^{-2})\Phi ^{S_t}(1+l,m)\cdot u[k,l]$
$=\Phi ^{S_t\cup \{ l\} }(k,m),$ we get the required relation
$$
\Phi ^{S_l}(k,m)+\sum _{i\in {\rm Sp}(a), \ i<l}(F_i-vF_i^{\prime})u[k,i]\in {\bf U}.
$$
\end{proof}
\begin{corollary} If the main parameter $q$ is not a root of $1,$ then
every right coideal subalgebra of $U_q^+({\mathfrak so}_{2n+1})$
that contains the coradical
has a set of PBW-generators of the form $\Phi ^{S}(k,m).$
In particular there exists just a finite number of the right coideal subalgebras
in $U_q^+({\mathfrak so}_{2n+1})$ that contain the coradical.
If $q$ has a finite multiplicative order  $t>4,$ then this is the case for the homogeneous  
right coideal subalgebras in $u_q^+({\mathfrak so}_{2n+1}).$
\label{vd}
\end{corollary}
\begin{proof}
If {\bf U} is a right coideal subalgebra of $U_q^+(\frak{ so}_{2n+1})$
that contains {\bf k}$\, [G],$ then by Lemma \ref{odn1} it is homogeneous
in each $x_i.$ By Proposition \ref{strB} and Proposition \ref{pro} it has 
PBW-generators of the form  (\ref{vad1}):
\begin{equation}
c_u=u^s+\sum \alpha _iW_i\in \hbox{\bf U}, \ \ \ u=u[k,m], \ k\leq m\leq \psi (k).
\label{vad25}
\end{equation}
By (\ref{mu21})  we have
$p_{uu}=\sigma_k^m=q^2$ if $m\neq n,$ and $p_{uu}=q$ otherwise.
Thus, if $q$ is not a root of 1, Lemma \ref{nco1}
shows that in (\ref{vad25})  the exponent $s$ equals 1, 
while  all heights of the $c_u$'s in {\bf U} are infinite.

If $q$ has a finite multiplicative order $t>4,$ then $u[k,m]^h=0$ in $u_q^+(\frak{ sl}_{n+1}),$
where $h$ is the multiplicative order of $p_{uu},$ see Proposition \ref{strBu}.
By Lemma \ref{nco1} in (\ref{vad25}) we have $s\in \{ 1, h, hl^r\} .$ 
Since $u[k,m]^h=u[k,m]^{hl^r}=0,$ the exponent $s$
in (\ref{vad25})  equals 1, while the height of $c_u$ in {\bf U} equals $h.$

Since {\bf U} is homogeneous with respect to each $x_i\in X,$
in both cases the PBW-generators of {\bf U} have the following form
\begin{equation}
c_u=u[k,m]+\sum \alpha _iW_i,  \  \ k\leq m\leq \psi (k),
\label{vad22}
\end{equation}
where $W_i$ are the basis super-words starting with less than $u[k,m]$ super-letters, 
$D(W_i)=$ $D(u[k,m])$ $=x_k+x_{k+1}+\ldots +x_m.$ By Proposition \ref{phib} we have
$\Phi ^{S}(k,m)\in \,${\bf U}. The leading term of $\Phi ^{S}(k,m)$
equals $u[k,m],$ see definition (\ref{dhs}). Hence we may replace $c_u$ with 
$\Phi ^{S}(k,m)$ in the set of PBW-generators.
The number of possible elements $\Phi ^{S}(k,m)$ is finite. 
Hence the total number of possible sets of PBW-generators 
of the form $\Phi ^{S}(k,m)$ is finite as well. 
\end{proof}

\section{Elements $\Phi ^{[k,m-1]}(k,m).$}

In this section we are going to prove the following relation in $U_q^+({\mathfrak so}_{2n+1}):$
\begin{equation}
\Phi ^{[k,m-1]}(k,m)=(-1)^{m-k}\left( \prod _{m\geq i>j\geq k}p_{ij}^{-1}\right)  \cdot u[\psi (m),\psi (k)],
\label{algr1}
\end{equation}
where as above $\psi (i)=2n-i+1.$ 
The main idea of the proof is to use the Milinski-Schneider criterion (Lemma \ref{MS}). 
To do this we need to 
find the partial derivatives of the both sides. In what follows by $\partial _i,$ $1\leq i\leq 2n$
we denote the partial derivation with respect to $x_i,$ see (\ref{defdif}). In particular 
$\partial _i=\partial _{\psi (i)}.$ Coproduct formula (\ref{co}) with (\ref{calc}) implies
\begin{equation}
\partial _i(u[k,m])=\left\{ \begin{matrix}
(1-q^{-2})\tau _ku[k+1,m],\hfill & \hbox{ if } x_i=x_k,\  k<m;\hfill \cr
0, \hfill & \hbox{ if } x_i\neq x_k;\hfill \cr
1, \hfill & \hbox{ if } x_i=x_k, k=m.\hfill
\end{matrix} \right.
\label{pde}
\end{equation}
This allows us easily to find the derivatives of the right hand side. By induction on $m-k$
we shall prove a similar formula
\begin{equation}
\partial _i(\Phi ^{[k,m-1]}(k,m))=\left\{ \begin{matrix}
\beta_k^m \Phi ^{[k,m-2]}(k,m-1),\hfill & \hbox{ if } x_i=x_m,\  k<m;\hfill \cr
0, \hfill & \hbox{ if } x_i\neq x_m;\hfill \cr
1, \hfill & \hbox{ if } x_i=x_m, k=m,\hfill
\end{matrix} \right.
\label{pde1}
\end{equation}
where $\beta_k^m=-(1-q^{-2})\alpha _{km}^{m-1}$ $=-(1-q^{-2})\tau _{m-1}p(x_m, u(k,m-1))^{-1}.$
To simplify the notation we remark that
$\Phi ^{[k,m-1]}(k,m)=\Phi ^{S}(k,m)$ for arbitrary $S$ that contains the interval
$[k,m-1].$ In particular in the above formula $\Phi ^{[k,m-2]}(k,m-1)=\Phi ^{[k,m-1]}(k,m-1).$

If $x_i\neq x_m,$ $x_i\neq x_k,$ then (\ref{pde}) and the inductive supposition applied to definition
(\ref{dhs}) imply $\partial _i(\Phi ^{[k,m-1]}(k,m))=0.$

If $x_i=x_k\neq x_m,$ then $\partial _i=\partial _k.$ Taking into account definition (\ref{dhs})
we have
$$
\partial _k(\Phi ^{[k,m-1]}(k,m))=\partial _k(
u[k,m]-(1-q^{-2})\sum_{i=k}^{m-1} \alpha _{km}^{i} \, 
\Phi^{[k,m-1]}(1+i,m)u[k,i]),
$$
where $\alpha _{km}^{i}=\tau _{i}p(u(1+i,m),u(k,i))^{-1},$ 
while the $\tau $'s have been defined in (\ref{tau}). By means of the inductive supposition,
skew differential Leibniz formula (\ref{defdif}), and (\ref{pde}) we may continue
\begin{equation}
=(1-q^{-2}) \tau_k\left( u[k+1,m]-
\tau_k^{-1}\alpha _{km}^{k}p(u(1+k,m),x_k)
\Phi^{[k,m-1]}(1+k,m)\right.
\label{pde2}
\end{equation}

$$
 -(1-q^{-2})\sum_{i=k+1}^{m-1} \left.  \alpha _{km}^{i} \, p(u(1+i,m),x_k)
\Phi^{[k,m-1]}(1+i,m) u[k+1,i]
\right) .
$$
Since obviously $\alpha _{km}^{k}p(u(1+k,m),x_k)=\tau_k,$ $\alpha _{km}^{i} \, p(u(1+i,m),x_k)=\alpha _{k+1\, m}^{i},$ definition
(\ref{dhs}) shows that the above expression is zero.

If $x_i=x_m\neq x_k,$ then $\partial _i=\partial _m.$ Again by definition (\ref{dhs}),
inductive supposition, skew differential Leibniz formula (\ref{defdif}), and (\ref{pde})
we have
$$
\partial _m(\Phi ^{[k,m-1]}(k,m))=-(1-q^{-2})\sum_{i=k}^{m-2} \alpha _{km}^{i} \, 
\beta_{1+i}^{m}\Phi^{[k,m-2]}(1+i,m-1)u[k,i]
$$
\begin{equation}
-(1-q^{-2})\alpha _{km}^{m-1} u[k,m-1].                     
\label{pde3}
\end{equation}
By definition  $-(1-q^{-2})\alpha _{km}^{m-1}=\beta_k^m.$ At the same time 
$$
\alpha _{km}^{i}\beta _{1+i}^{m}=\tau _ip(u(1+i,m),u(k,i))^{-1}
\cdot \{ -(1-q^{-2})\tau _{m-1}p(x_m,u(1+i,m-1))^{-1}\} 
$$
$$
=-(1-q^{-2})\tau _{m-1}p(x_m,u(k,m-1))^{-1} \cdot \tau _ip(u(1+i,m-1),u(k,i))^{-1}=
\beta _{k}^{m}\cdot \alpha _{k\, m-1}^{i}.
$$
Thus, according to (\ref{dhs}) the right hand side of (\ref{pde3})
equals $\beta_k^m \Phi ^{[k,m-2]}(k,m-1),$ which is required.

\smallskip
Finally, if $x_i=x_m=x_k,$ $k\neq m,$ that is $m=\psi (k),$ then due to skew differential 
Leibniz formula (\ref{defdif}) the derivative  
$\partial _i(\Phi ^{[k,m-1]}(k,m))$ equals the sum of the expression (\ref{pde2})
with the right hand side of (\ref{pde3}).
It remains to note that (\ref{pde2}) is still zero,
while the right hand side of (\ref{pde3}) still equals $\beta_k^m \Phi ^{[k,m-2]}(k,m-1).$
Eq. (\ref{pde1}) is completely proved.

Now we are ready to prove (\ref{algr1}) by induction on $m-k.$ If $m=k$ both sides equal $x_k.$
If $k<m,$ then the derivatives $\partial _i$ of both sides are zero with the only exception
being $x_i=x_m=x_{\psi (m)}.$ Due to (\ref{pde}) the derivative $\partial _m$ applied to 
the right hand side of (\ref{algr1}) equals
\begin{equation}
(-1)^{m-k}\left( \prod _{m\geq i>j\geq k}p_{ij}^{-1}\right)  (1-q^{-2}) \tau _{\psi (m)}\cdot
 u[\psi (m)+1,\psi (k)].
\label{piz}
\end{equation} 
Since $\psi (m)=n$ if and only if $m-1=n,$ formula (\ref{tau}) yields $\tau _{\psi (m)}=\tau _{m-1}.$
At the same time (\ref{pde1}) and the inductive supposition imply
\begin{equation}
\partial _m(\Phi ^{[k,m-1]}(k,m))=\beta_k^m (-1)^{m-1-k}\left( \prod _{m>i>j\geq k}p_{ij}^{-1}\right)
u[\psi (m)+1,\psi (k)].
\label{piz1}
\end{equation}
By definition we have
$$
\beta_k^m=-(1-q^{-2})\tau_{m-1} p(x_m,u(k,m-1))^{-1}=
-(1-q^{-2})\tau_{m-1}\prod _{m>j\geq k}p_{mj}^{-1}.
$$
Thus (\ref{piz}) coincides with (\ref{piz1}), and due to MS-criterion (\ref{algr1}) is proved. 

\smallskip
\noindent
{\bf Remark}. To prove (\ref{algr1}) we have used the MS-criterion. Therefore 
if $q$ has a finite multiplicative order $t,$ relation
 (\ref{algr1}) is proved only for $u_q^+({\mathfrak so}_{2n+1}).$
However we have seen in Proposition \ref{strBu} that
if $t>4,$ then the kernel of the natural homomorphism 
$U_q^+({\mathfrak so}_{2n+1})\rightarrow u_q^+({\mathfrak so}_{2n+1})$
is generated by the elements $u[k,m]^h,$ $h\geq 3.$
At the same time all polynomials in (\ref{algr1}) have degree
less than or equal to 2 in each variable. Therefore (\ref{algr1}) is valid in 
$U_q^+({\mathfrak so}_{2n+1})$ provided that $t>4.$

\section{$(k,m)$-regular sets}

\begin{definition} \rm Let $1\leq k\leq n<m\leq 2n.$
A set $S$ is said to be {\it white $(k,m)$-regular} if for every
$i,$ $k-1\leq i<m,$ such that $k\leq \psi (i)\leq m+1$ either $i$ or $\psi (i)-1$ does not belong 
to $S\cup \{ k-1, m\} .$

A set $S$ is said to be {\it black $(k,m)$-regular} if for every
$i,$ $k\leq i\leq m,$ such that $k\leq \psi (i)\leq m+1$ either $i$ or $\psi (i)-1$ belongs 
to $S\setminus \{ k-1,m\} .$

If $m\leq n,$ or $k>n$
(or, equivalently, if $u[k,m]$ is of degree $\leq 1$ in $x_n$),
 then by definition each set $S$ is both white and 
black $(k,m)$-regular.

A set $S$ is said to be $(k,m)$-{\it regular} if it is either black or white $(k,m)$-regular.
\label{reg1}
\end{definition}

 If $k\leq n<m$ and $S$ 
is white $(k,m)$-regular, then $n\notin S,$  for $\psi (n)-1=n.$ 
If additionally $m<\psi (k),$ then taking $i=\psi (m)-1$ we get $\psi (i)-1=m,$
hence the definition implies $\psi (m)-1\notin S.$ We see that if $m<\psi (k),$
$k\leq n<m,$ then $S$ is white $(k,m)$-regular if and only if the shifted scheme
of $\Phi ^{S}(k,m)$ given in (\ref{grb2}) has no black columns:
\begin{equation}  
\begin{matrix}
 \ \ \ \ \ \ \ & \stackrel{m}{\bullet }&\cdots & \bullet  
& \stackrel{n+i}{\circ }  & \circ  & \cdots & \stackrel{n}{\circ }   \Leftarrow \cr 
 \stackrel{k-1}\circ \ \cdots & \stackrel{\psi (m)-1}{\circ }& \cdots & \circ  
& \stackrel{n-i}{\bullet } & \circ & \cdots & \stackrel{n}{\circ }   \hfill
\end{matrix}
\label{grab}
\end{equation}
In the same way, if $m>\psi (k),$
then for $i=\psi (k)$ we get $\psi (i)-1=k-1,$ hence $\psi (k)\notin S.$
That is, if $m>\psi (k),$
$k\leq n<m,$ then $S$ is white $(k,m)$-regular if and only if the shifted scheme
(\ref{grb2}) has no black columns and the first from the left complete column
is a white one.
\begin{equation} 
\begin{matrix}
\stackrel{m}{\bullet } \ \cdots & \stackrel{\psi (k)}{\circ } & \cdots &  \bullet & \stackrel{n+i}{\circ } & \circ & \cdots & \stackrel{n}{\circ }  \Leftarrow
\cr
 \ \ \ \ \ \hfill & \stackrel{k-1}{\circ }& \cdots & \circ &
\stackrel{n-i}{\bullet } & \circ &\cdots &  \stackrel{n}{\circ }  \hfill
\end{matrix}
\label{grab1}
\end{equation}

Similarly, if $k\leq n<m$ and $S$ 
is black $(k,m)$-regular, then $n\in S.$ 
If additionally $m<\psi (k),$ then taking $i=\psi (m)-1$ we get $\psi (i)-1=m,$
hence  $\psi (m)-1\in S.$ 
We see that if $m<\psi (k),$
$k\leq n<m,$ then $S$ is black $(k,m)$-regular if and only if the shifted scheme
(\ref{grb2}) has no white columns and the first from the left complete column 
is a black one.
\begin{equation}  
\begin{matrix}
 \ \ \ \ \ \ \ &\stackrel{m}{\bullet } & \cdots & \bullet  
& \stackrel{n+i}{\circ }  & \bullet  & \cdots & \stackrel{n}{\bullet }  \Leftarrow \cr 
 \stackrel{k-1}{\circ }\ \cdots & \stackrel{\psi (m)-1}{\bullet } & \cdots & \bullet  
& \stackrel{n-i}{\bullet } & \circ & \cdots & \stackrel{n}{\bullet }  \hfill  
\end{matrix}
\label{grab2}
\end{equation}
If $m>\psi (k),$
then for $i=\psi (k)$ we get $\psi (i)-1=k-1,$ hence $\psi (k)\in S.$
That is, if $m>\psi (k),$
$k\leq n<m,$ then $S$ is black $(k,m)$-regular if and only if the shifted scheme
(\ref{grb2}) has no white columns:
\begin{equation} 
\begin{matrix}
\stackrel{m}{\bullet } \ \cdots & \stackrel{\psi (k)}{\bullet } & \cdots &  \bullet & \stackrel{n+i}{\circ } & \bullet & \cdots & \stackrel{n}{\bullet }  \Leftarrow 
\cr
 \ \ \ \ \ \hfill & \stackrel{k-1}{\circ } & \cdots & \circ &
\stackrel{n-i}{\bullet } & \bullet  & \cdots &  \stackrel{n}{\bullet }  \hfill 
\end{matrix}
\label{grab3}
\end{equation}

At the same time we should stress that if $m=\psi (k),$
then no one set is $(k,m)$-regular. Indeed, for $i=k-1$ we have
$\psi (i)-1=m.$ Hence both of the elements $i, \psi (i)-1$ belong to $S\cup \{ k-1, m\} ,$
and therefore $S$ is not white $(k,\psi (k))$-regular. If we take $i=m,$ then 
$\psi (i)-1=k-1,$ and no one of the elements $i, \psi (i)-1$
belongs to $S\setminus \{ k-1,m\} .$ Thus $S$ is neither black $(k,\psi (k))$-regular.

\smallskip
Let  $S \cap [k,m-1]=\{ s_1,s_2,\ldots , s_r\} ,$ $s_1<s_2<\ldots <s_r.$
Denote $u_i=u[1+s_i,s_{i+1}],$ $0\leq i\leq r,$ where we put formally $s_0=k-1,$ $s_{r+1}=m,$
while $u[k,m]$ has been defined in (\ref{ww}).
\begin{lemma} 
If $S$ is white $(k,m)$-regular, then
values in $U_q^+({\mathfrak so}_{2n+1})$ of the bracketed words 
$[u_ru_{r-1}\ldots u_1u_0]$ and $[u_0u_1\ldots u_{r-1}u_r]$
are independent of the precise alignment of brackets.
\label{100}
\end{lemma}
\begin{proof}
Let $0\leq i<j-1,$ $j\leq r.$ Assume $k\leq n<m.$
The points $s_{i},$ and  $\psi (1+s_i)$ 
 form a column on the shifted  scheme (\ref{grab}) or (\ref{grab1}),
for $s_i+\psi (1+s_i)=2n.$
Hence $\psi (1+s_i)$ $=\psi (s_i)-1$ is not a black point. In particular 
$s_{j+1}\neq \psi (1+s_i),$  $s_j\neq \psi (1+s_i).$ 
Similarly the points $s_{i+1},$ and  $\psi (s_{i+1})-1$ 
 form a column on the shifted  scheme, and hence 
$s_{j+1}\neq \psi (s_{i+1})-1,$ $s_j\neq \psi (s_{i+1})-1.$

We have $1+s_i\leq s_{i+1}<s_j<s_{j+1},$ $s_{j+1}\neq \psi (1+s_i),$
$s_{j+1}\neq \psi (s_{i+1})-1,$ $s_j\neq \psi (1+s_i),$ and $s_j\neq \psi (s_{i+1})-1.$
 Therefore Proposition \ref{NU} with
$k\leftarrow 1+s_i, $ $i\leftarrow s_{i+1},$ $j\leftarrow s_j,$ $m\leftarrow s_{j+1}$
implies $[u_i,u_j]=[u_j,u_i]=0.$ If $m\leq n$ or $k>n,$ then $u_i$ and $u_j$ 
are separated, hence still $[u_i,u_j]=[u_j,u_i]=0$ due to Lemma \ref{sepp}.
It remains to apply Lemma \ref{indle}.
\end{proof}
\begin{lemma} 
If $S$ is white $(k,m)$-regular, then
$[u_0u_1\ldots u_{r-1}u_r]=u[k,m].$
\label{101}
\end{lemma}
\begin{proof}
We use induction on $r.$ If $r=0,$ the equality is clear.
In the general case the inductive supposition yields 
$[u_0u_1\ldots u_{r-1}]=u[k,s_r],$ for $S$ is white $(k,s_r)$-regular.
By Proposition \ref{ins2}
we have $[u[k,s_r],u_r]=u[k,m],$ unless $s_r=\psi (m)-1$ or $s_r=\psi (k).$
However the white $(k,m)$-regularity implies that $\psi (m)-1,$ $\psi (k)$ are not 
black points. We are done.
\end{proof}
\begin{lemma}
If $S$ is white $(k,m)$-regular, then in the above notations we have
\begin{equation}
\Phi^{S}(k,m)=(-1)^r\prod_{r\geq i>j\geq 0} p(u_i,u_j)^{-1}\cdot [u_ru_{r-1}\ldots u_1u_0].
\label{fxny}
\end{equation}
\label{xny}
\end{lemma}
\begin{proof}
To prove the equality it suffices to check the recurrence relations (\ref{dhs}) for the right hand side.
We shall use induction on $r.$ If $r=0,$ there is nothing to prove. By Lemma \ref{101} we have $u[k,m]=[u_0u_1\ldots u_{r-1}u_r].$
 The inductive supposition for the white $(k,m)$-regular set 
$S\setminus \{ s_1 \}$ takes up the form
\begin{equation} 
(-1)^{r-1}p(u_1,u_0)\prod_{r\geq i>j\geq 0} p(u_i,u_j)^{-1}\cdot [u_ru_{r-1}\ldots u_2[u_0u_1]]= 
[u_0u_1u_2\ldots u_r]
\label{kho}
\end{equation}
$$
-(1-q^{-2})\sum _{l=2}^r\alpha _{k,m}^{s_l}(-1)^{r-l}
\prod_{r\geq i>j\geq l} p(u_i,u_j)^{-1} \cdot [u_ru_{r-1}\ldots u_l]\cdot [[u_0u_1]u_2\ldots u_{l-1}].
$$
By definition $p(u_0,u_1)p(u_1,u_0)=\mu _k^{s_2,s_1},$ see Definition \ref{slo}, while by (\ref{mu2}),
(\ref{mu4}) we have $\mu _k^{s_2,s_1}=q^{-2},$ for the regularity condition implies $s_1\neq n,$
$s_1\neq \psi (s_2)-1,$ $s_1\neq \psi (k).$ Hence by (\ref{cha}) we may write
$$
p(u_1,u_0)[u_0,u_1]=-[u_1,u_0]+(1-q^{-2})u_1\cdot u_0.
$$
This implies
$$
p(u_1,u_0)[u_ru_{r-1}\ldots u_2[u_0u_1]]=
-[u_ru_{r-1}\ldots u_2u_1u_0]+(1-q^{-2})[[u_ru_{r-1}\ldots u_2],u_1\cdot u_0].
$$
Since $[u_i,u_0]=0,$ $i\geq 2,$ ad-identity (\ref{br1}) yields
$$
[[u_ru_{r-1}\ldots u_2],u_1\cdot u_0]=[u_ru_{r-1}\ldots u_2u_1]\cdot u_0.
$$
Thus the left hand side of (\ref{kho}) reduces to
$$
(-1)^r\prod_{r\geq i>j\geq 0} p(u_i,u_j)^{-1}\cdot [u_ru_{r-1}\ldots u_2u_1u_0] + {\frak A},
$$ 
where 
$$
{\frak A}=(1-q^{-2})(-1)^{r-1}\prod_{r\geq i>0} p(u_i,u_0)^{-1}\prod_{r\geq i>j\geq 1} p(u_i,u_j)^{-1}
\cdot [u_ru_{r-1}\ldots u_2u_1]\cdot u_0.
$$
At the same time ${\frak A}$ up to a sign coincides with the missing summand of the right hand side of (\ref{kho}) corresponding to $l=1,$ for 
$$
\alpha _{k,m}^{s_1}
=\tau _{s_1}p(u_ru_{r-1}\ldots u_1,u_0)^{-1}=\prod _{r\geq i>0}p(u_i,u_0)^{-1}.
$$
\end{proof}
\begin{corollary}
If $S$ is white $(k,m)$-regular, $s\in S\cup \{ n\} ,$ $k\leq s<m,$ then 
$$
\Phi^{S}(k,m)=-p_{ab}^{-1} [\Phi^{S}(1+s,m),\Phi^{S}(k,s)],
$$
where $a=u(1+s,m),$ $b=u(k,s).$
\label{xny1}
\end{corollary}
\begin{proof} Let $s=s_t,$ $1\leq t\leq r.$
Since by Lemma \ref{100} the value of the bracketed word
$[u_ru_{r-1}\ldots u_0]$ is independent of the precise alignment of brackets,
we have $[u_ru_{r-1}\ldots u_0]=[[u_ru_{r-1}\ldots u_t], [u_{t-1}\ldots u_0]].$
It remains to apply Lemma \ref{xny}.

Let $k\leq s=n<m.$ Since in a white regular set, $n$ is always white, we can find $j$
such that $s_j<n<s_{j+1}.$ Denote $u_j^{\prime }=u[1+s_j,n],$ 
$u_j^{\prime \prime }=u[n+1,s_{j+1}].$ 
The points $s_j$ and  $\psi (1+s_{j})$ form a column on shifted scheme 
(\ref{grab}) or (\ref{grab1}). Hence $\psi (1+s_{j})$ is a white point.
In particular  $s_{j+1}\neq \psi (1+s_{j}).$ Thus, by Corollary \ref{rww}  with 
$k\leftarrow 1+s_j,$ $m\leftarrow s_{j+1}$ we have
$u_j=[u_j^{\prime },u_j^{\prime \prime }]$ 
$=-p(u_j^{\prime \prime },u_j^{\prime })^{-1}[u_j^{\prime \prime },u_j^{\prime }].$ 

Note that the value of the bracketed word 
\begin{equation}
[u_ru_{r-1}\ldots u_{j+1}u_j^{\prime \prime }u_j^{\prime }u_{j-1}\ldots u_1u_0]
\label{nez}
\end{equation}
is independent of the precise alignment of brackets. 
Indeed, Lemma \ref{nuly}
with $k\leftarrow 1+s_j,$ $i\leftarrow s_{i},$ $m\leftarrow s_{i+1}$
says $[u_i,u_j^{\prime }]=0,$ $i>j,$ unless $s_{i+1}=\psi (1+s_{j})$ or 
$s_{i}=\psi (1+s_j).$ However the points $s_{j},$ and  $\psi (1+s_{j})$ 
 form a column on the shifted  scheme (\ref{grab}) or (\ref{grab1}).
Hence $\psi (1+s_{j})$ is not a black point. In particular 
$s_{i+1}\neq \psi (1+s_{j}),$ and  $s_{i}\neq \psi (1+s_j).$

At the same time if $i<j-1,$ then $u_j^{\prime }$ and $u_i,$ 
are separated by $u_{j-1}$ (Definition \ref{sep}), hence Lemma \ref{sepp}
 implies $[u_j^{\prime },u_i]=0.$

In perfect analogy we get $[u_j^{\prime \prime }, u_i]=0,$ $i<j,$ and 
$[u_i, u_j^{\prime \prime }]=0,$ $i>j+1.$ Thus Lemma \ref{indle} implies that
(\ref{nez}) is independent of the precise alignment of brackets.
In particular 
$$
[u_ru_{r-1}\ldots u_{j+1}u_j^{\prime \prime }u_j^{\prime }u_{j-1}\ldots u_1u_0]
=\left[ [u_ru_{r-1}\ldots u_{j+1}u_j^{\prime \prime }],[u_j^{\prime }u_{j-1}\ldots u_1u_0]
\right]. 
$$
It remains to apply Lemma \ref{xny}. 
\end{proof}
\begin{lemma}
If $k\leq t<m,$ $t\notin S,$ then
\begin{equation} 
\Phi^{S\cup \{ t\}}(k,m)-\Phi^{S}(k,m)=(q^{-2}-1)p_{ab}^{-1}\tau _t
\Phi^{S}(1+t,m)\Phi^{S}(k,t),
\label{kho1}
\end{equation}
where $a=u(1+t,m),$ $b=u(k,t).$
\label{pfb}
\end{lemma}
\begin{proof}
We shall use induction on $m-k.$ If $m=k,$ there is nothing to prove. 
By definition (\ref{dhs}) we have 
$$
\Phi^{S\cup \{ t\}}(k,m)-\Phi^{S}(k,m)=-(1-q^{-2})
 \hbox{\large \{ } \tau _tp_{ab}^{-1}\Phi^{S}(1+t,m)u[k,t]
$$
$$
+\sum _{s_i<t}\tau _{s_i}p^{-1}_{u_iv_i}
(\Phi^{S\cup \{ t\}}(1+s_i,m)-\Phi^{S}(1+s_i,m))u[k,s_i]\hbox{\large \} } ,
$$
where $u_i=u(1+s_i,m),$  $v_i=u(k,s_i).$ By means of the inductive supposition we may continue
$$
=(q^{-2}-1)p_{ab}^{-1}\tau _t\Phi^{S}(1+t,m)
$$
$$
\cdot \hbox{\large \{ }u[k,t]
-(1-q^{-2})
\sum _{s_i<t}\tau _{s_i}p^{-1}_{u_iv_i}p^{-1}_{ab_i}p_{ab}\Phi^{S}(1+s_i,t)u[k,s_i] 
\hbox{\large \} },
$$
where $b_i=u(1+s_i,t).$ It remains to note that
$$
p^{-1}_{u_iv_i}p^{-1}_{ab_i}p_{ab}=p(u(1+s_i,t), u(k,s_i))^{-1},
$$
and use definition (\ref{dhs}).
\end{proof}
\begin{corollary}
If $S\cup \{ t\} $ is white $(k,m)$-regular, $t\notin S,$ $k\leq t<m,$ then
\begin{equation} 
\Phi^{S}(k,m)\sim \left[ \Phi^{S}(k,t),\Phi^{S}(1+t,m)\right] .
\label{ff}
\end{equation}
\label{xnt}
\end{corollary}
\begin{proof}
Denote $A=\Phi^{S}(k,t),$ $B=\Phi^{S}(1+t,m).$ By Corollary \ref{xny1}
we have 
$\Phi^{S\cup \{ t\}}(k,m)=-p_{ab}^{-1}[B,A].$ At the same time
$t\neq n$ (for $S \cup \{ t\} $ is white $(k,m)$-regular), and hence by Lemma \ref{pfb}
we get 
$
\Phi^{S\cup \{ t\}}(k,m)-\Phi^{S}(k,m)=(q^{-2}-1)p_{ab}^{-1}BA.
$
These two equalities imply 
$$
\Phi^{S}(k,m)=-p_{ab}^{-1}[B,A]-(q^{-2}-1)p_{ab}^{-1}BA
$$
\begin{equation}
=p_{ab}^{-1}(-BA+p_{BA}AB-(q^{-2}-1)BA)=p_{ab}^{-1}p_{BA}(AB-q^{-2}p_{BA}^{-1}BA).
\label{fff}
\end{equation}
By definition (\ref{slo}) we know that $p_{AB}p_{BA}={\mu }_k^{m,t}.$ In which case 
schemes (\ref{grab}), (\ref{grab1}) related to the white regular set 
$S \cup \{ t\} $ show that $t\neq \psi (m)-1,$ $t\neq n,$ $t\neq \psi (k),$
$m\neq \psi (k),$ for $t,m$ are black points. Hence formulae (\ref{mu2}), (\ref{mu4})
imply  ${\mu }_k^{m,t}=q^{-2}.$ Thus, we get $p_{AB}p_{BA}=q^{-2};$
that is, $q^{-2}p_{BA}^{-1}=p_{AB}.$ Now (\ref{fff}) reduces to (\ref{ff}).
\end{proof}
\begin{lemma}
A set $S$ is white $(k,m)$-regular if and only if $\overline{\psi ({S})-1}$
is black regular with respect to $(\psi (m),\psi (k)).$
Here by $\psi ({S})-1$ we denote $\{ \psi (s)-1\, |\, s\in {S}\} ,$
while the complement is related to the interval  $[\psi (m),\psi (k)-1].$
\label{bwd}
\end{lemma}
\begin{proof}
Let us replace the parameter $i$ with $j=\psi (i)-1$ in the definition of regularity.
Since $\psi $ change the order, we have $k-1\leq i<m$ is equivalent to
$\psi (k)+1\geq \psi (i)>\psi (m),$ that is $\psi (k)\geq j\geq \psi (m).$
Similarly the condition $k\leq \psi (i)\leq m+1$ is equivalent to
$\psi (k)\geq i\geq \psi (m)-1.$ Since $\psi (j)=i+1,$ this is
$\psi (k)+1\geq \psi (j)\geq \psi (m).$ 

The condition $i\notin {S} \cup \{ k-1,m\} $ is equivalent to
$j\notin (\psi ({S})-1)\cup \{ \psi (m)-1, \psi (k)\} ,$ which, in turn,
is equivalent to $j\in \left( \overline{\psi ({S})-1}\right) \setminus  \{ \psi (m)-1, \psi (k)\} .$
In the same way $\psi (i)-1\notin {S}\cup \{ k-1,m\} $
is equivalent to 
$\psi (j)-1\in \left( \overline{\psi ({S})-1}\right) \setminus  \{ \psi (m)-1, \psi (k)\} .$
\end{proof}
\begin{lemma}
A set $S$ is black $(k,m)$-regular if and only if $\overline{\psi ({S})-1}$
is white $(\psi (m),\psi (k))$-regular. 
\label{bwd1}
\end{lemma}
\begin{proof}
This follows from the above lemma under the substitutions
$k\leftarrow \psi (m),$ $m\leftarrow \psi (k),$ $S\leftarrow \overline{\psi ({S})-1}.$
\end{proof}

Alternatively one may easily to check Lemma \ref{bwd} and Lemma \ref{bwd1}
by means of the scheme interpretation (\ref{grab}--\ref{grab3}). Indeed,
the shifted representation for $\Phi ^{T}(\psi (m),\psi (k)),$ $T=\overline{\psi (S)-1}$
appears from one for $\Phi ^{S}(k,m)$ by changing the color of all points
and switching the rows.
\begin{proposition}
If $S$ is black $(k,m)$-regular, then 
$$
\Phi ^{S}(k,m)=(-1)^{m-k}q^{-2r}\left( \prod _{m\geq i>j\geq k}p_{ij}^{-1}\right) \cdot
\Phi ^{T}(\psi (m),\psi (k)),
$$
where $T=\overline{\psi (S)-1}$ is a white $(\psi (m),\psi (k))$-regular set
with $r$ elements,
and as above by $\psi ({S})-1$ we denote $\{ \psi (s)-1\, |\, s\in {S}\} ,$
while  the complement is related to the interval  $[\psi (m),\psi (k)-1].$
\label{dec}
\end{proposition}
\begin{proof}
We shall use double induction on $r$ and on $m-k.$ If $m=k,$ then the equality reduces to
$x_k=x_{\psi (k)}.$ If for given $k,m$ we have $r=0,$ then $S$ contains the interval
$[k,m-1]$ and the equality reduces to (\ref{algr1}).

Suppose that $r>0.$ We fix an element $t\in T.$ By the inductive supposition on $r$ we get
\begin{equation} 
\Phi^{S\cup \{ \psi (t)-1\}}(k,m)
=(-1)^{m-k}q^{-2(r-1)}\left( \prod _{m\geq i>j\geq k}p_{ij}^{-1}\right) \cdot
\Phi ^{{T}\setminus \{ t\} } (\psi (m),\psi (k)).
\label{de1}
\end{equation}
We have $t\notin \psi ({S})-1,$ and hence $\psi (t)-1\notin S.$ In particular 
$\psi (t)-1\neq n,$ and $\tau _{\psi (t)-1}=1,$ see (\ref{tau}).
 Thus, relation (\ref{kho1}) with $t\leftarrow \psi (t)-1$ implies
\begin{equation} 
\Phi^{S}(k,m)=\Phi^{S\cup \{ \psi (t)-1\}}(k,m)+(1-q^{-2})p_{ab}^{-1}a\cdot b,
\label{de2}
\end{equation}
where $a=\Phi^{S}(\psi (t),m),$ $b=\Phi^{S}(k, \psi (t)-1).$
Inductive supposition on $m-k$ yields 
$$
a=(-1)^{m-\psi (t)}q^{-2r_1}\left( \prod _{m\geq i>j\geq \psi (t)}p_{ij}^{-1}\right) \cdot
\Phi ^{T}(\psi (m),t ),
$$
$$
b=(-1)^{\psi (t)-1-k}q^{-2r_2}\left( \prod _{\psi (t)>i>j\geq k}p_{ij}^{-1}\right) \cdot
\Phi ^{T}(1+t, \psi (k)),
$$
where $r_1$ is the number of elements in $T\cap [\psi (m), t-1],$ and $r_2$
is the number of elements in $T\cap [1+t,\psi (k)-1].$ Obviously $r_1+r_2=r-1.$
Therefore 
\begin{equation} 
p_{ab}^{-1}ab=(-1)^{m-k-1}q^{-2(r-1)}\left( \prod _{m\geq i>j\geq k}p_{ij}^{-1}\right) \cdot cd,
\label{de3}
\end{equation}
where $c=\Phi ^{T}(\psi (m),t ),$ $d=\Phi ^{T}(1+t, \psi (k)).$ Now (\ref{de2}) and (\ref{de1})
imply
\begin{equation} 
\Phi^{S}(k,m)=(-1)^{m-k}q^{-2(r-1)}\left( \prod _{m\geq i>j\geq k}p_{ij}^{-1}\right)
\cdot \left\{ \Phi ^{{T}\setminus \{ t\} } (\psi (m),\psi (k))-(1-q^{-2})cd\right\} .
\label{de4}
\end{equation}
We have $t\neq n,$ for $T$ is white regular. Hence relation (\ref{kho1}) with
$S \leftarrow T\setminus \{ t\} ,$ $t\leftarrow t,$ $k\leftarrow \psi (m),$ $m\leftarrow \psi (k)$
implies
$$
\Phi ^{{T}\setminus \{ t\} } (\psi (m),\psi (k))=\Phi ^{{T}} (\psi (m),\psi (k))+(1-q^{-2})p_{dc}^{-1}dc,
$$ 
and the expression in braces of (\ref{de4}) reduces to
\begin{equation} 
\Phi ^{T} (\psi (m),\psi (k))+(1-q^{-2})p_{dc}^{-1}[d,c].
\label{de5}
\end{equation}
At the same time Corollary \ref{xny1} with
$S \leftarrow T,$ $s\leftarrow t,$ $k\leftarrow \psi (m),$ $m\leftarrow \psi (k)$
shows that $p_{dc}^{-1}[d,c]=-\Phi ^{T} (\psi (m),\psi (k)).$ This allows us to continue (\ref{de5}):
$$
=\Phi ^{T} (\psi (m),\psi (k))-(1-q^{-2})\Phi ^{T} (\psi (m),\psi (k))=q^{-2}\Phi ^{T} (\psi (m),\psi (k)).
$$
In order to get the required relation it remains to replace the expression in
braces of (\ref{de4}) with $q^{-2}\Phi ^{T} (\psi (m),\psi (k)).$
\end{proof}
\begin{corollary} 
If $S$ is $(k,m)$-regular, then 
$\Phi ^{S}(k,m)\sim \Phi ^{T}(\psi (m),\psi (k))$ for a suitable  $(\psi (m),\psi (k))$-regular set $T$.
\label{dec0}
\end{corollary}
\begin{proof} 
If $S$ is black $(k,m)$-regular, we apply Proposition \ref{dec}. If $S$ is white $(k,m)$-regular,
we still may apply Proposition \ref{dec} with $S \leftarrow T,$ $T\leftarrow S$ due to
Lemma \ref{bwd1}. 
\end{proof}
\begin{corollary} 
Let $S$ be $(k,m)$-regular. If $m>\psi (k),$ then the leading term of 
$\Phi ^{S}(k,m)$ is proportional to $u[\psi (m),\psi (k)].$
In particular always $\Phi ^{S}(k,m)\neq 0.$
\label{dec1}
\end{corollary}
\begin{proof} 
If $m<\psi (k),$ then definition (\ref{dhs}) shows that the leading term of $\Phi ^{S}(k,m)$
in the PBW-decomposition is $u[k,m],$ hence $\Phi ^{S}(k,m)\neq 0.$

If $m>\psi (k),$ then Proposition \ref{dec} 
(with $T \leftarrow  S,$ $S \leftarrow \,T$ provided that $S$ is white regular)
shows that $\Phi ^{S}(k,m)$ is proportional to $\Phi ^{T}(\psi (m),\psi (k))\neq 0,$
for $\psi (k)<\psi (\psi (m))=m.$
\end{proof} 

\begin{corollary} 
If $S$ is black $(k,m)$-regular, $t\notin {S}\setminus \{ n\} ,$ $k\leq t<m,$ then
$$
\Phi ^{S}(k,m)\sim \left[ \Phi ^{S}(k,t),\Phi ^{S}(1+t,m)\right] 
$$
\label{xnz}
\end{corollary}
\begin{proof} 
If $t\notin {S}\setminus \{ n\} ,$ then 
$\psi (t)-1\in {T}\cup \{ n\} ,$ where $T=\overline{\psi (S)-1}.$ By Proposition \ref{dec}
we have $\Phi ^{S}(k,m)\sim \Phi ^{T}(\psi (m),\psi (k)).$ Corollary \ref{xny1} yields
$$
\Phi ^{T}(\psi (m),\psi (k))\sim \left[ \Phi ^{T}(\psi (t),\psi (k)),\Phi ^{T}(\psi (m),\psi (t)-1)\right] .
$$
Since $t$ is a white point or $t=n,$ the set $S$ is black $(k,t)$-regular and black $(1+t,m)$-regular,
see the shifted schemes (\ref{grab2}), (\ref{grab3}).
Hence Proposition \ref{dec} implies $\Phi ^{S}(k,t)\sim \Phi ^{T}(\psi (t),\psi (k)),$
 $\Phi ^{S}(1+t,m)\sim \Phi ^{T}(\psi (m),\psi (t)-1).$ We are done.
\end{proof} 
\begin{corollary}
If $S \setminus \{ s\} $ is black $(k,m)$-regular, $s\in S,$ $k\leq s<m,$ then
\begin{equation} 
\Phi^{S}(k,m)\sim \left[ \Phi^{S}(1+s,m), \Phi^{S}(k,s)\right] .
\label{ff1}
\end{equation}
\label{xnt1}
\end{corollary}
\begin{proof}
Follows from Lemma \ref{xnt} and Proposition \ref{dec} in the similar way.
\end{proof}

\section{Root sequence}
\smallskip
Our next goal is to show that the total number of right coideal subalgebras containing {\bf k}$[G]$
 is less than or equal to $(2n)!!=2^n\cdot n!.$

In what follows for short we shall denote by $[k:m],$ $k\leq m\leq 2n$ the element 
$x_k+x_{k+1}+\ldots +x_m$ considered as an element of the group $\Gamma ^+.$
Of course $[k:m]=[\psi (m):\psi (k)].$
If $k\leq m<\psi (k),$ then $[k:m]$ is an $U_q^+({\mathfrak so}_{2n+1})$-root
since $u[k,m]$ is a PBW-generator for $U_q^+({\mathfrak so}_{2n+1}).$ 
The simple $U_q^+({\mathfrak so}_{2n+1})$-roots
are precisely the generators $x_k=[k:k],$ $1\leq k\leq n.$
To put it another way, the $U_q^+({\mathfrak so}_{2n+1})$-roots  form the positive
part $R^+$ of the classical root system of type $B_n,$ provided that we formally 
replace symbols $x_i$ with $\alpha _i$ 
(the Weyl basis for $R,$ see \cite[Chapter IV, \S 6, Theorem 7]{Jac}).

We fix a notation {\bf U} for a (homogeneous if $q^t=1,$ $t>4$) right 
coideal subalgebra of $U_q^+({\mathfrak so}_{2n+1}),$ $q^t\neq 1$ 
(respectively of $u_q^+({\mathfrak so}_{2n+1}),$) that contains $G.$
The {\bf U}-roots form a subset $D({\bf U})$ of $R^+.$ In this section we will see, in particular, that $D({\bf U})$
uniquely defines {\bf U}.
\begin{definition} \rm
Let $\gamma _k$ be a simple {\bf U}-root of the form $[k:m],$ $k\leq m<\psi (k)$
with the maximal $m.$
Denote by $\theta _k$ the number $m-k+1,$ the length of $\gamma _k.$
If there are no simple {\bf U}-roots of the form $[k:m],$ $k\leq m<\psi (k)$ 
we put $\theta _k=0.$
The sequence $r({\bf U})=(\theta_1, \theta_2, \ldots ,\theta_n)$
satisfies $0\leq \theta_k\leq 2n-2k+1$ and it is uniquely defined by {\bf U}.
We shall call $r({\bf U})$ a {\it root sequence of } {\bf U}, or just an $r$-{\it sequence of} {\bf U}.
By $\tilde{\theta }_k$ we denote $k+\theta_k -1,$ the maximal value of $m$ for the simple
 {\bf U}-roots of the form $[k:m]$ with fixed $k.$
\label{tet}
\end{definition}
\begin{theorem} 
For each sequence 
$\theta=(\theta_1, \theta_2, \ldots ,\theta_n),$ such that $0\leq \theta_k\leq 2n-2k+1,$
$1\leq k\leq n$ there exists at most one $($homogeneous if $q^t=1,$ $t>4$ $)$ right coideal subalgebra
{\bf U}$\, \supseteq G$ of $U_q^+(\frak{ sl}_{n+1}),$ $q^t\neq 1$ $($respectively, of $u_q^+(\frak{ sl}_{n+1}))$
with $r({\bf U})=\theta .$
\label{teor}
\end{theorem}
The proof will result from the following lemmas.
\begin{lemma} 
If $[k:m]$ is a simple {\bf U}-root, then 
there exists only one element $a\in \, ${\bf U} of the form $a=\Phi ^{S}(k,m).$ 
\label{uni}
\end{lemma}
\begin{proof} 
Suppose that $a=\Phi ^{S}(k,m)$ and $b=\Phi ^{S^{\prime }}(k,m)$
are two different elements from {\bf U}. Then $a-b$ is not a PBW-generator for {\bf U}
since its leading term with respect to the PBW-decomposition given in Proposition \ref{strB} 
is not equal to $u[k,m].$ Hence nonzero homogeneous element $a-b$ is a polynomial
in PBW-generators of {\bf U}. Thus $[k:m],$ being the degree of $a-b,$ is a sum of 
{\bf U}-roots, a contradiction.
\end{proof}
\begin{lemma} 
Let $\Phi ^{S}(k,m)\in \, ${\bf U}, $k\leq m<\psi (k).$ Suppose that 
$\Phi ^{S^{\prime }}(k,m)\notin \, ${\bf U} 
for all subsets $S^{\prime }\subset S.$ If $j\notin S,$ $k\leq j<m,$
then $\Phi ^{S}(1+j,m)\in \, ${\bf U}.
If $j\in S,$ $k\leq j<m,$ then $\Phi ^{S^{\prime \prime }}(k,j)\in \, ${\bf U}
with $S^{\prime \prime }\subseteq S \cap [k, j],$
in particular $[k:j]$ is an {\bf U}-root.
\label{sub0}
\end{lemma}
\begin{proof}  If in (\ref{dhs}) we have  $\Phi ^{S}(1+s_i,m)=0,$ then
the spectrum {\rm Sp}$(a)$ of $a=\Phi ^{S}(k,m)$ is a proper subset of $S\cup \{ m\} .$
By Proposition \ref{phib} there exists a subset 
$S^{\prime }\subseteq {\rm Sp}(a)\subset S$ such that
$\Phi ^{S^{\prime }}(k,m)\in \, ${\bf U}. This contradiction implies that
$\Phi ^{S}(1+j,m)\neq 0$ for all $j\in S\cap [k,m-1].$

If $j\notin S,$ then Lemma \ref{las3} implies $\Phi ^{S}(1+j,m)\in \, ${\bf U}.

If $j\in S,$ then we apply $\Delta \cdot ({\rm id}\otimes \nu _a)$ with 
$a=\Phi ^{S}(1+j,m)\neq 0$
as defined in Lemma \ref{las4} to both sides of  (\ref{dhs}). Lemma \ref{las4}
shows that the value of $\Delta (\Phi ^{S}(1+i,m)u[k,i])\cdot ({\rm id}\otimes \nu _a)$ 
has the following three options. If $j<i<m,$ this is zero. If $i=j,$ this is $g_au[k,j]\otimes a.$
If $i<r,$ this is $g_ab_i^{\prime }u[k,i],$ $b_i^{\prime }\in A_{k+1}.$ Since 
$\Delta (u[k,m]) \cdot ({\rm id}\otimes \nu _a)=0$ due to  (\ref{co}), we get the following 
relation
$$
b=u[k,j]+\sum_{i<j, i\in {S}} b_i^{\prime }u[k,i]\in \hbox{\bf U}, \ \ b_i^{\prime }\in A_{k+1}.
$$
By definition this means that $[k:j]$ is an {\bf U}-root, while Proposition \ref{phib} 
implies $\Phi ^{S^{\prime \prime }}(k,j)\in \, ${\bf U}
with $S^{\prime \prime }\subseteq {\rm Sp}(b)\subseteq S \cap [k, j].$
\end{proof}
\begin{lemma} 
If $[k:m]$ is a simple {\bf U}-root, $k\leq m<\psi (k),$ then 
the minimal  $S$ such that  $\Phi ^{S}(k,m)\in \, ${\bf U} equals 
$\{ j\ | \ k\leq j<m, \ [k:j] \hbox{ is an {\bf U}-root}\} ,$ and this is a $(k,m)$-regular set, see Definition 
$\ref{reg1}.$
\label{sub1}
\end{lemma}
\begin{proof} 
Suppose that $S$ is not $(k,m)$-regular. Then we have $k\leq n<m.$

If $n$ is a white point, $n\notin S,$ then by Lemma \ref{sub0} we have
 $\Phi ^{S}(1+n,m)\in \, ${\bf U}. Hence $[n+1:m]=[\psi (m):n]$ is an {\bf U}-root
due to Corollary \ref{dec1}.
Since $S$ is not white $(k,m)$-regular, on the shifted scheme (\ref{grab1})
we can find a black column, say, $n+i\in S \cup \{ m\},$
$n-i\in S.$ By Lemma \ref{sub0} applied to  $\Phi ^{S}(n+1,m)$
we have $[n+1:n+i]$ is an {\bf U}-root, while the same lemma applied to
$\Phi ^{S}(k,m)$ shows that $[k:n-i]$ is also an {\bf U}-root.
Now we have
$$
[k:m]=[k:n]+[n+1:m]=[k:n-i]+[n+1:n+i]+[n+1:m]
$$
is a sum of {\bf U}-roots, a contradiction.

If $n$ is a black point, $n\in S$, then by Lemma \ref{sub0}  we have 
$\Phi ^{S^{\prime \prime }}(k,n)\in \,${\bf U}, and $[k:n]$
is an {\bf U}-root. Since $S$ is not black $(k,m)$-regular, 
we can find $i,$ $1\leq i\leq m-n$ 
such that $n+i\notin  S \setminus \{ m\}$, $n-i\notin S$, see (\ref{grab2}). 
We have $n-i\notin S^{\prime \prime },$ for $S^{\prime \prime }\subseteq S.$
Hence Lemma \ref{sub0} applied to $\Phi ^{S^{\prime \prime }}(k,n)$ implies that
$[1+n-i:n]=[n+1:n+i]$ is an {\bf U}-root. The same lemma applied to 
$\Phi ^{S}(k,m)$ shows that $\Phi ^{S}(1+n+i,m)\in \,${\bf U}.
Hence, due to Corollary \ref{dec1}, the element $[1+n+i:m]=[\psi (m):n-i]$ is an {\bf U}-root as well.
We get a similar contradiction:
$$
[k:m]=[k:n]+[n+1:m]=[k:n]+[n+1:n+i]+[1+n+i:m].
$$

Due to Lemma \ref{sub0} it remains to show that if $[k:j]$ is an {\bf U}-root,
then $j\in S.$ Suppose in contrary that $j\notin S.$ Then Lemma \ref{sub0}
implies $a=\Phi ^{S}(1+j,m)\in \,${\bf U}. 

If $S$ is $(1+j,m)$-regular, or $1+j<\psi (m),$ then $a\neq 0$ and
$[1+j:m]$ is an  {\bf U}-root, see Corollary \ref{dec1}.
This is a contradiction, for $[k:m]=[k:j]+[1+j:m].$

Suppose, finally, that $S$ is not $(1+j,m)$-regular and $1+j\geq \psi (m).$ 
Since $S$ indeed is $(k,m)$-regular,
this is possible only in two cases: $j=\psi (m)-1$, or $n\notin S,$ $\psi (j)-1\in S,$
see the shifted scheme representations (\ref{grab1}), (\ref{grab3}). 

In the former case by Lemma \ref{sub0} either 
$\Phi ^{S}(1+n,m)\in \,${\bf U} (if $n\notin S),$ or 
$\Phi ^{S^{\prime \prime }}(k,m)\in \,${\bf U} and 
$\Phi ^{S}(1+j,m)\in \, ${\bf U}, for 
$j\notin S^{\prime \prime } \subseteq S$
 (if $n\in S).$ Therefore $[n+1:m]=[\psi (m),n]=[j+1:n]$
is an {\bf U}-root due to Corollary \ref{dec1}. We get a contradiction 
$[k:m]=[k:\psi (m)-1]+[\psi (m),n]+[n+1:m].$

In the latter case similarly $\Phi ^{S}(1+n,m)\in \,${\bf U} and 
$\Phi ^{S^{\prime \prime }}(1+j,n)\in \,${\bf U}. Hence Corollary \ref{dec1} 
implies that $[n+1:m],$ $[1+j:n]$ are {\bf U}-roots. Again contradiction:
$[k:m]=[k:j]+[1+j,n]+[n+1:m].$
\end{proof}
\begin{lemma} 
If $[k:m]=\sum _{i=1}^{r+1}[l_i:m_i],$ $k\leq m\leq 2n,$ $l_i\leq m_i<\psi (l_i),$
then it is possible to replace some of the pairs $(l_i,m_i)$ with $(\psi (m_i), \psi (l_i))$
so that the given decomposition takes up the form
\begin{equation}
[k:m]=[1+k_0:k_1]+[1+k_1:k_2]+\ldots +[1+k_r:m]
\label{dso}
\end{equation}
with $k-1=k_0<k_1<k_2<\ldots <k_r<m=k_{r+1}.$
\label{sos}
\end{lemma}
\begin{proof} 
We use induction on $m-k.$ Either $x_k$ or $x_m$ is the maximal letter among 
$\{ x_j | k\leq j\leq m\} .$ Hence there exists at least one $i$ such that, respectively,
$l_i=k$ or $l_i=\psi (m).$ In the former case we may put $k_1=m_i$
and apply the inductive supposition to $[m_i+1:m].$ In the latter case we put
$k_r=\psi (m_i)-1.$ Then $[k_r+1:m]=[\psi (m_i): \psi (l_i)].$
One may apply the inductive supposition to $[k:k_r].$
\end{proof} 

\begin{lemma} 
If $[k:m],$ $k\leq m\neq \psi (k)$ is a sum of {\bf U}-roots, then $[k:m]$ itself is an {\bf U}-root.
\label{sos1}
\end{lemma}
\begin{proof} 
Without loss of generality we may suppose that $m<\psi (k),$ for $[k:m]=[\psi (m):\psi (k)].$
By Lemma \ref{sos} we have a decomposition (\ref{dso}), where $[1+k_i:k_{i+1}],$ $0\leq i<r$
are {\bf U}-roots. Increasing, if necessary, the number $r$ we may suppose that all roots 
$[1+k_i:k_{i+1}],$ $0\leq i<r$ are simple. 

If $k_{i+1}<\psi (1+k_i),$ then by Proposition \ref{phib}
we find a set $S_i\subseteq [1+k_i,k_{i+1}-1]$ such that 
$\Phi ^{S_i}(1+k_i,k_{i+1})\in \, ${\bf U}. Moreover by Lemma \ref{sub1}
the set $S_i$ may be taken to be $(1+k_i,k_{i+1})$-regular. 

If $k_{i+1}>\psi (1+k_i),$  then of course $\psi (1+k_i)< \psi (\psi (k_{i+1})),$ and again by 
Proposition  \ref{phib} and Lemma \ref{sub1} we find a $(\psi (k_{i+1}), \psi (1+k_{i}))$-regular 
set $T_i\subseteq [\psi (k_{i+1}), \psi (1+k_{i})-1]$ such that   
$\Phi ^{T_i}(\psi (k_{i+1}), \psi (1+k_{i}))\in \, ${\bf U}. By Corollary \ref{dec0} with 
$S \leftarrow T_i$ we have $\Phi ^{T_i}(\psi (k_{i+1}), \psi (1+k_{i}))$ 
$\sim \Phi ^{S_i}(1+k_{i}, k_{i+1}),$ where $S_i$ is $(1+k_{i}, k_{i+1})$-regular.
Thus in all cases 
\begin{equation}
f_i\stackrel{df}{=}\Phi ^{S_i}(1+k_i,k_{i+1})\in {\bf U},\ \ {S}_i\subseteq [1+k_i,k_{i+1}-1]
\label{ttt}
\end{equation}
with regular $S_i$ (we stress that this is a restriction on $S_i$ only if $1+k_i\leq n<k_{i+1}$).

By Definition \ref{root} we have to construct an element $c\in \, ${\bf U} with the leading 
super-word $u[k,m].$ Firstly we shall prove that for $r=1$ the element $c=[f_0,f_1]$
is such an element even if $[1+k_{i}:k_{i+1}]$ are not necessarily simple roots,
but $S_i,$ $i=0,1$ are still regular sets.

There is the following natural diminishing process for the decomposition
of a linear combination of super-words in the PBW-basis given in Theorem \ref{BW}, 
and Propositions  \ref{strB}, \ref{strBu}. Let $W$ be a super-word.
First, due to \cite[Lemma 7]{Kh3} we decompose the super-word $W$ 
in a linear combination of smaller monotonous super-words. Then, we replace  
each non-hard super-letter with the decomposition of its value that exists 
by Definition \ref{tv1}, and again decompose the appeared super-words 
in linear combinations of smaller monotonous super-words, and so on, 
until we get a linear combination of monotonous super-words in hard
super-letters. If these super-words are not restricted, 
we may apply Definition \ref{h1} and repeat the process
until we get only monotonous restricted words in hard super-letters.

This process shows that if a super-word $W$ starts with a super-letter smaller than $u[k,m],$
then all super-words in the PBW-decomposition of $W$ do as well. Using this remark we shall
prove the following auxiliary statement.

{\it If $k\leq i<j<m<\psi (k),$ $m\neq \psi (i)-1,$ then all super-words in the PBW-decomposition 
of $[u[k,i], \Phi ^{S}(1+j,m)]$ start with super-letters smaller than} $u[k,m].$

Indeed, by definition (\ref{dhs}) we have 
$$
\Phi ^{S}(1+j,m)=u[1+j,m]+
\sum_{m>s\geq 1+j} \gamma _s \Phi ^{S}(1+s,m)\cdot u[1+j,s], \ \gamma _s \in {\bf k}.
$$
We use induction on $m-j.$
By Proposition \ref{NU} we have $[u[k,i], u[1+j,m]]=0,$ for inequalities $\psi (k)>m>j$ imply $j\neq \psi (k).$
Denote $u=u[k,i],$ $v=\Phi ^{S}(1+s,m),$ $w=u[1+j,s].$
Relation (\ref{br1}) reads $[u,v\cdot w]=[u,v]\cdot w+p_{uv}v\cdot [u,w].$
By the inductive supposition all super-words in the PBW-decomposition
of $[u,v]$ start with smaller than $u[k,m]$ super-letters. Hence so do ones for $[u,v]\cdot w.$
The element $v$ depends only on $x_i,$ $i>k.$ Therefore so do all super-letters
in the PBW-decomposition of $v,$ while the starting super-letters of   $v\cdot [u,w]$
are still less than $u[k,m].$ Thus, all super-words in the PBW-decomposition 
of $[u[k,i], \Phi ^{S}(1+j,m)]$ start with super-letters smaller than $u[k,m].$
This proves the auxiliary statement.

Now  we have 
$$
[f_1,f_2]=\left[ \Phi ^{S_0}(k, k_1),\Phi ^{S_1}(1+k_1, m)\right] 
$$
$$
=\left[ u[k,k_1]+
\sum_{k_1>s\geq k} \gamma _s \Phi ^{S_0}(1+s,k_1)\cdot u[k,s],\right.
$$
$$
\left. u[1+k_1,m]+
\sum_{m>l\geq 1+k_1} \beta _l \Phi ^{S_1}(1+l,m)\cdot u[1+j,l]\right] 
$$
$$
=u[k,m]+\sum_{m>l\geq 1+k_1} \beta _l \left[u[k,k_1] ,
\Phi ^{S_1}(1+l,m)\cdot u[1+k_1,l]\right] 
$$
$$
+\sum_{k_1>s\geq k} \gamma _s 
 \left[   \Phi ^{S_0}(1+s,k_1)\cdot u[k,s], f_2\right] . 
$$
We see that each element in the latter sum has a non trivial left factor that 
depends only on $x_i,$ $i>k,$ this is either $\Phi ^{S_0}(1+s,k_1)$ or 
$f_2.$ Hence all super-words in the PBW-decomposition 
of that elements start with super-letters smaller than $u[k,m].$ To check the former sum
denote $u=u[k,k_1],$ $v=\Phi ^{S_1}(1+l,m),$ $w=u[1+k_1,l].$ By (\ref{br1})
the general element in the sum is proportional to 
$[u,v\cdot w]=[u,v]\cdot w+p_{uv}v\cdot [u,w].$ By the above auxiliary statement
with $i\leftarrow k_1,$ $j\leftarrow l$ all super-words in the PBW-decomposition 
of $[u,v]$ start with super-letters smaller than $u[k,m].$ Hence so do the ones
of $[u,v]\cdot w.$ The element $v$ depends only on $x_i,$ $i>k.$ Therefore the starting 
super-letters in the PBW-decomposition of   $v\cdot [u,w]$
are less than $u[k,m]$ as well. Thus the leading term of $[f_0,f_1]$ indeed is $u[k,m].$
This completes the case $r=1.$

\smallskip
Consider the general case. Denote by $t$ the index such that $1+k_t\leq n\leq k_{t+1},$
if any. Recall that $S_t$ is either white or black $(1+k_t,k_{t+1})$-regular,
while each of $S_i,$ $i\neq t$ is both white and black $(1+k_i,k_{i+1})$-regular,
for its degree in $x_n$ is less than or equal to $1.$ We shall consider four options for the regular
set $S_t$ given in (\ref{grab}--\ref{grab3}) separately.

1. $k_{t+1}<\psi (1+k_t),$ {\it and $S_t$ is white regular}. Let 
$S =\cup _{i=0}^t{S}_i\cup \{ k_i\, | \, 0<i<t\} .$ The set $S$ is white 
$(k,k_{t+1})$-regular since all complete columns on the shifted scheme (\ref{grab}) for
$\Phi ^{S}(k,k_{t+1})$ coincide with ones for $\Phi ^{S_t}(k_t,k_{t+1}).$
By Lemma \ref{xny} we have 
$$
\Phi ^{S}(k,k_{t+1})\sim [f_tf_{t-1}\ldots f_2f_1f_0]
$$
with an arbitrary alignment of brackets on the right hand side. In the same way consider 
the set $S^{\prime }=\cup _{i=t+1}^r{S}_i\cup \{ k_i\, | \, t+1<i<r\} .$ This set is white
$(1+k_{t+1}, m)$-regular, for the shifted scheme (\ref{grab}) for
$\Phi ^{S^{\prime }}(1+k_{t+1},m)$ has no complete columns at all. 
Lemma \ref{xny} yields 
$$
\Phi ^{S^{\prime }}(1+k_{t+1},m)\sim [f_rf_{r-1}\ldots f_{t+2}f_{t+1}].
$$
Now we may apply the considered above case $r=1$ with $S_0\leftarrow S,$
 $S_1\leftarrow S^{\prime },$ $t_1\leftarrow k_{t+1}.$  
Thus the leading super-word of the element
\begin{equation}
c=\left[ [f_tf_{t-1}\ldots f_2f_1f_0],[f_rf_{r-1}\ldots f_{t+2}f_{t+1}]\right]
\label{tc1}
\end{equation}
equals $u[k,m]$ and obviously $c\in \, ${\bf U}, for $f_i\in \, ${\bf U}, $0\leq i\leq r.$

2. $k_{t+1}>\psi (1+k_t),$ {\it and $S_t$ is white regular}. In perfect analogy we consider 
the sets $S =\cup _{i=0}^{t-1}{S}_i\cup \{ k_i\, | \, 0<i<t-1\} ,$ and 
$S^{\prime }=\cup _{i=t}^r{S}_i\cup \{ k_i\, | \, t<i<r\} .$ 
By the case $r=1$ under the substitutions $S_0\leftarrow S,$
 $S_1\leftarrow S^{\prime },$ $t_1\leftarrow k_t,$ we see that  the required element is
\begin{equation}
c=\left[ [f_{t-1}f_{t-2}\ldots f_1f_0],[f_rf_{r-1}\ldots f_{t+1}f_t]\right] .
\label{tc2}
\end{equation}

3. $k_{t+1}<\psi (1+k_t),$ {\it and $S_t$ is black regular}. Let 
$S =\cup _{i=0}^t{S}_i.$ The set $S$ is black 
$(k,k_{t+1})$-regular since all complete columns on the shifted scheme (\ref{grab2}) for
$\Phi ^{S}(k,k_{t+1})$ coincide with ones for $\Phi ^{S_t}(k_t,k_{t+1}).$
No one of the points $k_1, k_2, \ldots , r_r$ belongs to the set $S$, see (\ref{ttt}).
Therefore by multiple use of Corollary \ref{xnz} we have 
$$
\Phi ^{S}(k,k_{t+1})\sim [f_0f_1\ldots f_{t-1}f_t]
$$
with an arbitrary alignment of brackets on the right hand side. In the same way consider 
the set $S^{\prime }=\cup _{i=t+1}^r{S}_i.$ This set is black
$(1+k_{t+1}, m)$-regular, for the shifted scheme (\ref{grab2}) for
$\Phi ^{S^{\prime }}(1+k_{t+1},m)$ has no complete columns at all. 
The multiple use of Corollary \ref{xnz} yields 
$$
\Phi ^{S^{\prime }}(1+k_{t+1},m)\sim [f_{t+1}f_{t+2}\ldots f_{r-1}f_r].
$$
Now we may find $c$ using the case $r=1$ with $S_0\leftarrow {S},$
 $S_1\leftarrow {S}^{\prime },$ $t_1\leftarrow k_{t+1}:$  
\begin{equation}
c=\left[ [f_0f_1\ldots f_{t-1}f_t],[f_{t+1}f_{t+2}\ldots f_{r-1}f_r]\right] .
\label{tc3}
\end{equation}

4. $k_{t+1}>\psi (1+k_t),$ {\it and $S_t$ is black regular}. In perfect analogy we consider 
the sets $S =\cup _{i=0}^{t-1}{S}_i,$ and 
$S^{\prime }=\cup _{i=t}^r{S}_i.$ 
By the case $r=1$ under the substitutions $S_0\leftarrow {S},$
 $S_1\leftarrow {S}^{\prime },$ $t_1\leftarrow k_t,$ we see that  
the required element is
\begin{equation}
c=\left[ [f_0f_1\ldots f_{t-2}f_{t-1}],[f_tf_{t+1}\ldots f_{r-2}f_{r-1}f_r]\right] .
\label{tc4}
\end{equation}
The proof is complete.
\end{proof}

\begin{lemma} 
If $[k:m],$ $k\leq m<\psi (k)$ is a simple {\bf U}-root, $k\leq j<m,$ then $[k:j]$ is an {\bf U}-root
if and only if $[1+j:m]$ is not a sum of {\bf U}-roots.
\label{su}
\end{lemma}
\begin{proof} 
If $[k,j],$ is an {\bf U}-root, then $[1+j:m]$ is not a sum of {\bf U}-roots, for $[k:m]=[k:j]+[1+j:m]$ is a simple {\bf U}-root.

We note, first, that the converse statement is valid if the minimal $S$, such that 
$\Phi ^{S}(k,m) \in \, ${\bf U},  is $(1+j, m)$-regular. 
Indeed, in this case $\Phi ^{S}(1+j,m) \neq 0$ due to Corollary \ref{dec1}. By Lemma \ref{sub1} the element $[k:j]$ is an {\bf U}-root if and only if  $j\in S.$  
If $j\notin S,$ then by Lemma \ref{sub0} we have $a=\Phi ^{S}(1+j,m)\in \,${\bf U}.
Hence the nonzero homogeneous element $a$ is a polynomial
in PBW-generators of {\bf U}. Thus $[1+j:m],$ being the degree of $a,$ is a sum of 
{\bf U}-roots (by Lemma \ref{sos1} this is even an {\bf U}-root, for   
the regularity hypothesis implies $\psi (1+j)\neq m$). 

Suppose, next, that $S$ is not $(1+j, m)$-regular, and $j\notin S.$
 In this case $1+j\leq n<m.$ Moreover 
$m\geq \psi (1+j)$ since otherwise all complete columns in the shifted scheme
(\ref{grab}--\ref{grab3}) of 
$\Phi ^{S}(1+j,m)$ coincide with that of  $\Phi ^{S}(k,m).$
Obviously in general only the first from the left complete column for $\Phi ^{S}(1+j,m)$
may be different from a complete column for $\Phi ^{S}(k,m).$
Hence we have just the following three options:
1) $\psi (1+j)=m;$ 2) $\psi (1+j)\in S,$ while $n\notin {S};$ 
3) $\psi (1+j)\notin S,$ while $n\in {S}.$

1) In the shifted scheme of
$\Phi ^{S}(k,m),$ the point   $j=\psi (m)-1$ has the same 
color as $n$ does , see (\ref{grab}), (\ref{grab2}); that is, $n$ is a white point. 
At the same time, since $S$ is always 
$(n+1,m)$-regular, we already know that the point $n$ is white if and only if 
$[n+1:m]$ is an {\bf U}-root. Thus $[n+1:m]$ is an {\bf U}-root, while
$[1+j:m]=[1+j:n]+[n+1:m]=2[n+1:m]$ is a sum of two {\bf U}-roots.

2) In the second case certainly $S$ is $(n+1,m)$-regular. Hence $n\notin S$
implies that $[n+1:m]$ is an {\bf U}-root. By Lemma \ref{sub0} we have
$\Phi ^{S^{\prime \prime }}(k,\psi (1+j))\in \, ${\bf U} with 
$S^{\prime \prime }\subseteq S,$ for $\psi (1+j)\in S.$
In particular, still $n\notin S.$
Hence again the same lemma implies $a=\Phi ^{S}(n+1,\psi (1+j))\in \, ${\bf U}.
By Corollary (\ref{dec1}) the leading super-word of $a$ equals $u[1+j,n];$
that is, $[1+j:n]$ is an {\bf U}-root. Now $[1+j:m]=[1+j:n]+[n+1:m]$ is a sum of two roots,
which is required. 

3) By Lemma \ref{sub0} we have
$\Phi ^{S^{\prime \prime }}(k,n)\in \, ${\bf U} with 
$S^{\prime \prime }\subseteq S,$ for $n\in S.$
In particular, still $j\notin S^{\prime \prime }.$ Hence the same lemma
implies that $[1+j:n]$ is an {\bf U}-root. Since $\psi (1+j)\notin S,$ and
obviously $S$ is $(\psi (1+j), m)$-regular, we already know that $[1+\psi (1+j):m]=[\psi (j):m]$
is an {\bf U}-root. Now we have $[1+j:m]=[1+j:n]+[n+1:\psi (1+j)]+[\psi (j):m]$ is a sum of 
{\bf U}-roots, for $[n+1:\psi (1+j)]=[1+j:n].$
\end{proof}
\begin{lemma} 
A $($homogeneous$)$ right coideal subalgebra {\bf U} that contains {\bf k}$[G]$ 
is uniquely defined by the set of all its
simple roots.
\label{sub2}
\end{lemma}
\begin{proof} Since obviously two subalgebras with the same PBW-basis coincide,
it suffices to find a PBW-basis of {\bf U} that depends only on a set of simple {\bf U}-roots.
We note firstly that the set of all {\bf U}-roots is uniquely defined by the set of simple {\bf U}-roots.
Indeed, if $[k:m]$ is an {\bf U}-root, then it is a sum of simple {\bf U}-roots. 
By Lemma \ref{sos} there exists a sequence $k-1=k_0<k_1<\ldots <k_r<m=k_{r+1}$
such that $[1+k_i:k_{i+1}],$ $0\leq i\leq r$ are simple {\bf U}-roots.
Conversely, 
if there exists a sequence $k-1=k_0<k_1<\ldots <k_r=m+1$
such that $[1+k_i:k_{i+1}],$ $0\leq i<r$ are simple {\bf U}-roots, then 
by Lemma \ref{sos1} the element $[k:m]$ is an {\bf U}-root. Of course the decomposition of $[k:m]$
in a sum of simple {\bf U}-roots is not unique in general, however for the construction 
of the PBW-basis we may fix that decomposition for each non-simple {\bf U}-root 
from the very beginning. 

Now if $[k:m]$ is a simple {\bf U}-root, Lemmas \ref{uni} and \ref{sub1}
show that the element $\Phi ^{S}(k,m)\in \, ${\bf U} 
is uniquely defined by the set of simple {\bf U}-roots. We include this element
in the PBW-basis of {\bf U}. If $[k:m]$ is a non-simple {\bf U}-root with a fixed decomposition
in a sum of simple {\bf U}-roots, then we include
in the PBW-basis the element $c$ defined in  one of the formulae (\ref{tc1}-\ref{tc4})
depending up the type of the decomposition.  
\end{proof}
\begin{lemma} 
If for $($homogeneous$)$ right coideal subalgebras
{\bf U}, {\bf U}$^{\prime }$ containing {\bf k}$[G]$ we have $r({\bf U})=r({\bf U}^{\prime }),$
 then {\bf U}$\, =\, ${\bf U}$^{\prime }.$
\label{su3}
\end{lemma}
\begin{proof} 
By Lemma \ref{sub2} it suffices to show that the $r$-sequence uniquely defines the set of all simple roots. We use the downward induction on $k,$ the onset of a simple {\bf U}-root.
If $k=n,$ then the only possible $\gamma =[n:n]$ is a simple {\bf U}-root if and only if  $\theta _n=1.$
Let $k<n.$ 
By definition there  do not exist simple {\bf U}-roots of the form
$[k:m],$ $m>\tilde{\theta }_k$,  while $[k:\tilde{\theta }_k]$
 is a simple {\bf U}-root. If $m<\tilde{\theta }_k,$
then by Lemma \ref{su} the element $[k:m]$ is an {\bf U}-root if and only if
$[m+1:\tilde{\theta }_k]$  is not a sum of {\bf U}-roots
starting with a number greater than $k.$ 
By inductive supposition the $r$-sequence
defines all roots starting with a number greater than $k.$
 Hence by Lemma \ref{su} the $r$-sequence defines  the set of
all {\bf U}-roots of the form $[k:m],$ $m<\tilde{\theta }_k$
as well. Thus  the $r$-sequence
defines the set if all {\bf U}-roots and the set of simple {\bf U}-roots.
\end{proof}

 \section{Examples}
In this section we find the simple roots for some fundamental 
examples of right coideal subalgebras. We keep all notations of the 
above  section.

\begin{example} \rm 
Let ${\bf U}(k,m)$ be a right coideal subalgebra generated over  {\bf k}$[G]$ by 
a single element $u[k,m],$ $k\leq m\leq \psi (k).$ 
By (\ref{co}) the right coideal generated by $u[k,m]$ is spanned by the elements 
$g_{ki}u[i+1,m].$ Hence ${\bf U}(k,m)$ as an algebra
is generated over {\bf k}$[G]$ by the elements $u[i,m],$ $k\leq i\leq m.$ 
Respectively, the additive monoid of degrees of  homogeneous elements from ${\bf U}(k,m)$  is
generated by $[i:m],$ $k\leq i\leq m.$ In this monoid the indecomposable elements
(by definition they are simple ${\bf U}(k,m)$-roots)  are precisely $[i:m],$ $k\leq i\leq m,$ $i\neq \psi (m).$
The length of $[i:m]$ equals $m-i+1.$ However, if  $i>\psi (m),$ then the maximal letter 
among $x_j,$  $i\leq j\leq m$ is $x_{\psi (m)},$ for  $[i:m]=[\psi (m):\psi (i)]$
with $\psi (m)\leq \psi (i)<\psi (\psi (m)).$ Hence  the maximal length 
of a simple root starting with $\psi (m)$ equals $m-(\psi (m)+1)+1=2(m-n)-1,$
while there are no simple roots of the form 
$[k^{\prime }:m^{\prime }],$ $k^{\prime }\leq m^{\prime }<\psi (k^{\prime })$ 
with $k^{\prime }>\psi (m).$
Thus due to Definition \ref{tet} we have
\begin{equation}
\theta _i=\left\{
\begin{matrix} 
m-i+1, \hfill & \hbox{ if } k\leq i<\psi (m); \cr
2(m-n)-1, \hfill & \hbox{ if } k\leq i=\psi (m)\leq n; \cr
0, \hfill & \hbox{ otherwise. } 
\end{matrix} 
\right. 
\label{r00}
\end{equation} 
The set $\{ u[i,m] \, |\, k\leq i\leq m, i\neq \psi (m)\} $ is a set of PBW-generators for ${\bf U}(k,m)$
over {\bf k}$[G].$
\label{ex1}
\end{example}
\begin{example} \rm Let us analyze in details the simplest (but not trivial, \cite{BDR})
case $n=2.$ Consider the following  six elements
$w_1=u[1,3]=[[x_1,x_2],x_2],$ $w_2=u[2,4]=[x_2,[x_2,x_1]],$
$w_3=u[1,2]=[x_1,x_2],$ $w_4=u[3,4]=[x_2,x_1],$ $w_5=x_1,$ $w_6=x_2.$ 
Denote by $U_j,$ $1\leq j\leq 6$ a right coideal subalgebra generated by $w_j$ and {\bf k}$[G].$  

By means of (\ref{r00}) we have $r(U_1)=(3,1).$ Indeed, in this case $k=1,$ $m=3,$
$\psi (m)=2;$ hence $\theta _1=m-1+1=3$ according to the first option of (\ref{r00}),
while  $\theta _2=2(m-n)-1=1$ due to the second option of (\ref{r00}).

In the same way $r(U_2)=(3,0),$ for in this case $k=2,$ $m=4,$
$\psi (m)=1;$ hence $\theta _1=2(m-2)-1=3$ according to the second option,
while  $\theta _2=0$ due to the third one.

In perfect analogy we have $r(U_3)=(2,1),$ $r(U_4)=(2,0),$ $r(U_5)=(1,0),$ $r(U_6)=(0,1).$
We see that all these six right coideal subalgebras are different. There are two more
(improper) right coideal subalgebras $U_7=U_q^+({\mathfrak so}_{5}),$
$U_8={\bf k}[G]$ with the $r$-sequences $(1,1)$ and $(0,0)$
respectively. Thus we have found all $(2n)!!=8$ possible right coideal subalgebras in
$U_q^+({\mathfrak so}_{5})$ containing $G.$ They form the following lattice:
 
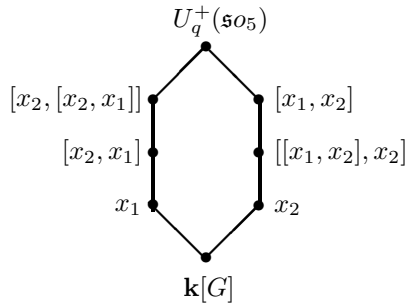
\begin{picture}(40,100)(-150,10)
\thicklines
\put(20,80){$\bullet $} 
   \put(24,81){\line(1,-1){19}}
\put(10, 90){$U_q^+({\mathfrak so}_{5})$}
\put(40,60){$\bullet \ [x_1,x_2]$}
\put(40,40){$\bullet \ [[x_1,x_2],x_2]$}
\put(40,20){$\bullet \  x_2$}
 \put(43,21){\line(0,1){42}}
 \put(0,60){$\bullet $} \put(-52,60){$[x_2,[x_2,x_1]]$} \put(2,62){\line(1,1){21}}
\put(0,40){$ \bullet $} \put(-32,40){$[x_2,x_1]$}
\put(0,20){$\bullet $} \put(-12,20){$x_1$}
\put(2.5,20){\line(0,1){42}} 
\put(1,23.5){\line(1,-1){20}}
 \put(20,0){$\bullet $} \put(14, -12){${\bf k}[G]$} 
\put(22,2){\line(1,1){21}}
\end{picture}

\

\

\label{ex2}
\end{example}

\smallskip
Our next goal is to generalize formula (\ref{r00}) to an arbitrary right coideal subalgebra 
{\bf U}$^{S}(k,m)$ generated over {\bf k}$[G]$ (as a right coideal subalgebra) 
by a single element $\Phi ^{S}(k,m)$ with a $(k,m)$-regular set $S$.
\begin{proposition}
If $S$ is $(k,m)$-regular, then the coproduct of $\Phi ^{S}(k,m)$
has a decomposition
\begin{equation}
\Delta (\Phi ^{S}(k,m))=\sum a^{(1)}\otimes a^{(2)},
\label{coph}
\end{equation}
where degrees of the left components of tensors belong to the additive monoid $\Sigma $
 generated by all $[1+t:s]$ with $t$ being a white point $(t=k-1,$ or $t\notin S,$ $k\leq t<m)$, 
and $s$ being a black point $(s\in S\cap [k,m-1],$ or $s=m).$
  \label{ex2i}
\end{proposition}
\begin{proof}
Let $S$ be white $(k,m)$-regular. 
Lemma \ref{xny} shows that $\Phi ^{S}(k,m)$
is a linear combination of products (in different orders) of the elements 
$u_i=u[1+s_i,s_{i+1}],$ $0\leq i\leq r.$ Hence by (\ref{co}) the coproduct  
is a linear combination of products of the tensors
\begin{equation}
u_i\otimes 1,\ \ f_i\otimes u_i, \ \ h_{i}u[1+t_i,s_{i+1}]\otimes u[1+s_i,t_i],
\label{tens}
\end{equation}
where $s_i<t_i<s_{t+1},$ $f_i={\rm gr}\, (u_i),$  $h_i={\rm gr}\, (u[1+s_i,t_i]).$
Degrees of the left components  of these tensors, except  $u_i\otimes 1,$ $i>0,$
 belong to $\Sigma .$ We stress that in each product there is exactly one
tensor of (\ref{tens}) related to a given $i.$ 

Denote by $\Sigma ^{\prime }$
the additive monoid generated by all $[1+t:s],$ where $t\notin S,$ $k\leq t<m,$
while $s$ is a black point.
By  induction on the number $r$
of elements in $S \cap [k,m-1]$ we shall prove that there exists a decomposition (\ref{coph})
such that for each $i$ either $D(a^{(1)})\in \Sigma ^{\prime }$ or 
$D(a^{(1)})=[k:s]+\alpha ,$ where $s$ is a black point, and  $\alpha \in \Sigma ^{\prime }.$

If $r=0,$ then 
$\Phi ^{S}(k,m)=u[k,m],$ and the statement follows from (\ref{co}).

If $r>0,$ then Corollary \ref{xny1} implies that 
$\Phi ^{S}(k,m)\sim \left[ \Phi ^{S}(1+s_1,m), u[k,s_1]\right] .$
By the inductive supposition we have 
$\Delta (\Phi ^{S}(1+s_1,m))=\sum b^{(1)}\otimes b^{(2)},$ where either
$D(b^{(1)})=\alpha \in \Sigma _1^{\prime },$ or $D(b^{(1)})=[1+s_1:s]+\alpha ,$ 
$\alpha \in \Sigma _1^{\prime }$ with 
$s$ being a black point on the scheme of $\Phi ^{S}(1+s_1,m),$
see (\ref{grb}). Here $\Sigma _1^{\prime }$ is the $\Sigma ^{\prime }$ related to 
$\Phi ^{S}(1+s_1,m):$
the additive monoid generated by all $[1+t:s],$ where $t\notin S,$ $s_1<t<m,$
and $s$ is a black point.
 Certainly $\Sigma _1^{\prime }\subseteq \Sigma ^{\prime },$
for on the scheme of $\Phi ^{S}(1+s_1,m)$ there is 
just one point, $s_1,$ that has color  other than it has on the scheme of $\Phi ^{S}(k,m).$

By (\ref{co}) the coproduct of $u_0=u[k,s_1]$ is a linear combination 
of the tensors (\ref{tens}) with $i=0.$ Degree of the left components 
of the tensors of 
$$\left[ b^{(1)}\otimes b^{(2)},h_0u[1+t_0,s_1]\otimes u[k,t_0]\right] $$
equals either $[1+t_0:s_1]+\alpha $ or $[1+t_0:s_1]+[1+s_1:s]+\alpha =[1+t_0:s]+\alpha .$
In both cases it belongs to $\Sigma ^{\prime } $ since $t_0$ is a white point in both schemes,
and $t_0\neq k-1.$

In the same way, degree of the left components 
of the tensors of $\left[ b^{(1)}\otimes b^{(2)},u_0\otimes 1\right] $
equals either $[k:s_1]+\alpha $ or $[k:s_1]+[1+s_1:s]+\alpha =[k:s]+\alpha .$
In both cases the degree has a required form.

It remains to consider the skew commutator
$$
\left[ b^{(1)}\otimes b^{(2)},f_0\otimes u_0\right] =b^{(1)}f_0\otimes b^{(2)}u_0
-p(b^{(1)}b^{(2)},u_0) f_0b^{(1)}\otimes u_0b^{(2)}.
$$
Degree of the left components of these tensors equals $D(b^{(1)}).$
We shall prove that one of the following three options is valid: 
$\left[ b^{(1)}\otimes b^{(2)},f_0\otimes u_0\right] =0,$ or
$D(b^{(1)})\in \Sigma ^{\prime },$ or $D(b^{(1)})=[k:s]+\alpha ,$ $\alpha \in \Sigma ^{\prime }$
with black $s.$

The comments given around (\ref{tens}) show that there exists a sequence of elements 
$(t_i\, |\,   0\leq i\leq r),$ such that   $s_i\leq t_i\leq s_{i+1},$ and 
\begin{equation}
D(b^{(1)})=\sum _{i=1}^r [1+t_i:s_{i+1}], \ \ \ D(b^{(2)})=\sum _{i=1}^r [1+s_i:t_i],
\label{ste}
\end{equation}
where formally $[1+s_i:s_i]=[1+s_{i+1}:s_{i+1}]=0.$ We consider separately the following two cases.

{\bf Case 1.} $t_1>s_1.$ Due to the first of (\ref{ste}), degree of $b^{(1)}$ in 
$x_{1+s_1}$ is less than or equal to 1. At the same time the equality $D(b^{(1)})=[1+s_1:s]+\alpha $
shows that this degree equals 1, and $x_{1+s_1}$-th component of $\alpha $ is zero.
Hence there exists $i\geq 2$ such that $t_i<\psi (1+s_1)\leq s_{i+1}.$ However,
$\psi (1+s_1)=\psi (s_1)-1$ is a white point, for $S$ is white $(k,m)$-regular.
In particular $\psi (1+s_1)\neq s_{i+1};$ that is, $\psi (1+s_1)<s_{i+1}.$
Now the nonempty interval $[1+\psi (1+s_1):s_{i+1}]=[\psi (s_1):s_{i+1}]$
must be covered by $\alpha \in \Sigma _1^{\prime }.$
This is possible only if $\alpha $
has a summand $\alpha _1=[\psi (s_1):s_j],$ $j\geq i+1,$
for degree of  $\Phi ^{S}(1+s_1,m)$ in each of $x_l,$ $\psi (s_1)\leq l\leq m$
equals 1, while the $x_{\psi (s_1)-1}$-th component of $\alpha $ is zero
(recall that $x_{\psi (s_1)-1}=x_{1+s_1}).$ Thus, we have $\alpha -\alpha _1\in \Sigma ^{\prime }_1.$

If $\psi (s_j)>k,$ or, equivalently, $s_j<\psi (k),$ then  $\psi (s_j)-1$ is a white point,
for $\psi (1+s_1)<s_{i+1}\leq s_j$ implies $s_1>\psi (s_j)-1.$ We have 
$$
\alpha _1+[1+s_1:s]=[\psi (s_1):s_j]+[1+s_1:s]=[\psi (s_j):s]\in \Sigma ^{\prime }.
$$ 
Hence $D(b^{(1)})=(\alpha _1+[1+s_1:s])+(\alpha -\alpha _1)\in \Sigma ^{\prime },$
which is required.

If $\psi (s_j)<k,$ or, equivalently, $s_j>\psi (k),$ then $\psi (k)$ is a white point, see (\ref{grab1}).
Hence $[\psi (s_j):k-1]=[1+\psi (k):s_j]\in \Sigma ^{\prime },$ while
$$
\alpha _1+[1+s_1:s]=[\psi (s_j):k-1]+[k:s]=[\psi (s_j):s]\in [k:s]+\Sigma ^{\prime },
$$
 and  $D(b^{(1)})=(\alpha _1+[1+s_1:s])+(\alpha -\alpha _1)\in [k:s]+\Sigma ^{\prime }.$

Of course $s_j\neq \psi (k)$ since $S$ is white $(k,m)$-regular, see (\ref{grab1}).

{\bf Case 2.} $t_1=s_1.$ Assume, first, that the sequence $(t_i\, |\, 1<i\leq r)$ does not contain 
the point $\psi (s_1)-1=\psi (1+s_1).$ We have seen, see comments around (\ref{tens}), that
$b^{(2)}$ is a product of the elements $u[1+s_i,t_i],$ $i>0$ in some order.
For $i=1$ the tensor $u_1\otimes 1$ does participate in the construction of $b^{(1)}\otimes b^{(2)}$ 
(recall that now $t_1=s_1).$
 By Proposition \ref{NU} with $i\leftarrow s_1,$ $j\leftarrow s_i$ $m\leftarrow t_i$
we have $\left[ u[1+s_i,t_i],u_0\right] =0,$ $i>1,$
for now $t_i\neq \psi (s_1)-1$ and $s_i\neq \psi (k),$ see (\ref{grab1}).
Hence ad-identity (\ref{br1f}) implies $\left[ b^{(2)},u_0\right] =0;$
that is, $b^{(2)}u_0=p(b^{(2)},u_0)u_0b^{(2)}.$ Since $f_0={\rm gr}(u_0),$ we have
$$
(b^{(1)}\otimes b^{(2)})(f_0\otimes u_0)=b^{(1)}f_0\otimes b^{(2)}u_0
$$
$$
=p(b^{(1)},u_0)f_0b^{(1)}\otimes p(b^{(2)},u_0)u_0b^{(2)}
=p(b^{(1)}b^{(2)},u_0)(f_0\otimes u_0)(b^{(1)}\otimes b^{(2)}).
$$
This equality in more compact form is
$
\left[ b^{(1)}\otimes b^{(2)},f_0\otimes u_0\right] =0,
$
which is one of the required options.

Suppose, next, that $\psi (s_1)-1=t_i$ for a suitable $i,$ $1<i\leq r.$
Due to the first of (\ref{ste}) degree of $b^{(1)}$ in $x_{1+s_i}=x_{t_i}$ equals 1,
while the equality $D(b^{(1)})=[1+s_1:s]+\alpha $ implies that $x_{1+s_1}$-th component
of $\alpha $ is zero. At the same time $t_i\neq s_{i+1},$ for $t_i$ and $s_1$ are located 
in the same column of the shifted scheme (\ref{grab}), (\ref{grab1}). Hence, again due to the first
of (\ref{ste}), the nonempty interval $[1+t_i:s_{i+1}]=[\psi (s_1):s_{i+1}]$ must be
covered by $\alpha \in \Sigma ^{\prime }_1.$ 
This is possible only if $\alpha $ has a summand $\alpha _1=[\psi (s_1):s_j],$ $j\geq i+1,$
for degree of  $\Phi ^{S}(1+s_1,m)$ in each of $x_l,$ $\psi (s_1)\leq l\leq m$
equals 1, while the $x_{\psi (s_1)-1}$-th component of $\alpha $ is zero
(recall that $x_{\psi (s_1)-1}=x_{1+s_1}).$ Thus, we have $\alpha -\alpha _1\in \Sigma ^{\prime }_1.$

If $\psi (s_j)>k,$ or, equivalently, $s_j<\psi (k),$ then  $\psi (s_j)-1$ is a white point,
for $\psi (1+s_1)<s_{i+1}\leq s_j$ implies $s_1>\psi (s_j)-1.$ We have 
$$
\alpha _1+[1+s_1:s]=[\psi (s_1):s_j]+[1+s_1:s]=[\psi (s_j):s]\in \Sigma ^{\prime }.
$$ 
Hence $D(b^{(1)})=(\alpha _1+[1+s_1:s])+(\alpha -\alpha _1)\in \Sigma ^{\prime },$
which is required.

If $\psi (s_j)<k,$ or, equivalently, $s_j>\psi (k),$ then $\psi (k)$ is a white point, see (\ref{grab1}).
Hence $[\psi (s_j):k-1]=[1+\psi (k):s_j]\in \Sigma ^{\prime },$ while
$$
\alpha _1+[1+s_1:s]=[\psi (s_j):k-1]+[k:s]=[\psi (s_j):s]\in [k:s]+\Sigma ^{\prime },
$$
 and  $D(b^{(1)})=(\alpha _1+[1+s_1:s])+(\alpha -\alpha _1)\in [k:s]+\Sigma ^{\prime }.$
Of course $s_j\neq \psi (k)$ since $S$ is white $(k,m)$-regular, see (\ref{grab1}).
 This completes the proof for a white regular set $S$.

If $S$ is black $(k,m)$-regular, then by Proposition \ref{dec} we have 
$\Phi ^{S}(k,m)\sim \Phi ^T(\psi (m),\psi (k)),$ where 
$T=\overline{\psi (S)-1}$ is a white $(\psi (m),\psi (k))$-regular set.
If $t,s$ are, respectively, white and black points for $\Phi ^{S}(k,m),$
then so do $\psi (s)-1,$ $\psi (t)-1$ with respect to $\Phi ^T(\psi (m),\psi (k)).$ We have
$$
[1+t:s]=[\psi (s):\psi (1+t)]=[1+(\psi (s)-1):\psi (t)-1].
$$
Hence  $\Phi ^{S}(k,m)$ and $\Phi ^T(\psi (m),\psi (k))$ define the same additive
monoid $\Sigma .$ It remains to apply already proved statement to $\Phi ^T(\psi (m),\psi (k)).$
\end{proof}
\begin{corollary}
If $S$ is $(k,m)$-regular, then all {\bf U}$^{S}(k,m)$-roots belong 
to the monoid $\Sigma$ defined in the above proposition. 
\label{call}
\end{corollary}
\begin{proof}
We are reminded that coassociativity of the coproduct implies that the left components
of the tensor (\ref{coph}) span a right coideal. Hence {\bf U}$^{S}(k,m)$
as an algebra is generated by the $a^{(1)}$'s and {\bf k}$[G].$
Hence the degrees of all homogeneous elements from {\bf U}$^{S}(k,m)$
belong to $\Sigma .$ In particular all {\bf U}$^{S}(k,m)$-roots, 
being the degrees of PBW-generators, belong to $\Sigma $ as well.
\end{proof}
\begin{lemma}
Let $S$ be a white $(k,m)$-regular set. An element $[1+t:s],$ $t<s$ with white $t$ and black $s$
is indecomposable in $\Sigma $ if and only if one of the following 
conditions is fulfilled:

$a)$ The point $\psi (1+t)$ is not  black $($it is white or does not appear on the scheme at all$\, ).$

$b)$ On the shifted scheme all columns between $t$ and $s$ are white-black or 
black-white ones $($in particular all of them are complete and $n\notin [t,s]).$
  \label{ex2j}
\end{lemma}
\begin{proof}
If no one of the conditions is fulfilled, then $\psi (1+t)$ is a black point and there exists
 $j,$ $t\leq j\leq s,$ such that both $j$ and $\psi (1+j)$ are white points on the scheme
(white regular scheme has no black-black columns).
Certainly $j\neq t,$ $j\neq s.$ We have 
$$[1+t:j]=[\psi (j):\psi (1+t)]=[1+\psi (1+j):\psi (1+t)]\in \Sigma .$$
Thus $[1+t:s]=[1+t:j]+[1+j:s]$ is a non trivial decomposition in $\Sigma .$

Conversely. Assume that $[1+t:s]$ is decomposable in $\Sigma :$ 
\begin{equation}
[1+t:s]=\sum _{i=1}^r[1+l_i:s_i].
\label{cjj}
\end{equation}
Without loss of generality we may suppose that $s_i\leq \psi (1+l_i)$
due to $[1+l_i:s_i]=[\psi (s_i):\psi (1+l_i)].$ Moreover, if $s_i=\psi (1+l_i),$
then $[1+l_i:n]=[1+n:s_i]\in \Sigma ,$ for $n$ is a white point ($S$ is white regular).
This allows one to replace $[1+l_i:s_i]$ with $2[1+n:s_i]$ in (\ref{cjj}). Thus we may suppose that
$s_i\leq \psi(1+l_i)$ for all $i$ in (\ref{cjj}).

By Lemma \ref{sos} we find a sequence 
$t=t_0<t_1<\ldots <t_r<s=t_{r+1}$ such that for each $i$ either $t_i$ is white and $t_{i+1}$
is black points, or $\psi (1+t_{i+1})$ is white and $\psi (1+t_i)$ is black ones.
In the former case to the index $i$ we associate the sign  ``$+$,"
while in the latter case to $i$ we associate the sign  ``$-$."
It is clear that in the sequence of indices $0,1,2,\ldots , r$ no one pair of neighbors 
has the same sign associated. 

If, now, $\psi (1+t)$ is not a black point, then ``$+$" is associated to the index $0.$
Hence ``$-$" is associated to the index $1.$ In particular $\psi (1+t_1)$ is a black point.
However $t_1$ is also a black point. This is impossible, for $S$ is white regular.

Assume that on the shifted scheme all columns between $t$ and $s$ are white-black or 
black-white ones. If $t_1$ is a white point, then both of $t_0=t,t_1$ are white,
while  both of $\psi (1+t_1), \psi (1+t_0)$ are black points; that is, no sign may be associated
to the index $0.$ Hence $t_1$ is a black point, while $\psi (1+t_1)$ must be a white one.
In this case ``$-$" may not be associated to the index $1.$ Thus ``$+$" is associated to $1.$
But then $t_1$ is a white point, a contradiction.
\end{proof}
\begin{lemma}
Let $S$ be a black $(k,m)$-regular set. An element $[1+t:s],$ $t<s$ with white $t$ and black $s$
is indecomposable in $\Sigma $ if and only if one of the following 
conditions is fulfilled:

$a)$ The point $\psi (1+s)$ is not white $($it is black or does not appear on the scheme at all$\, ).$

$b)$ On the shifted scheme all columns between $t$ and $s$ are white-black or 
black-white ones $($in particular all of them are complete and $n\notin [t,s]).$
  \label{ex2k}
\end{lemma}
\begin{proof} The proof follows from the above lemma by means of Lemma \ref{bwd1}
and Proposition \ref{dec}.
\end{proof}
\begin{lemma}
Let $S$ be a $(k,m)$-regular set. An element $\alpha =[a:b]$
is a simple  {\bf U}$^{S}(k,m)$-root if and only if $\alpha \in \Sigma $ and
it is indecomposable in $\Sigma \, ($in particular 
 $\alpha =[1+t:s],$ $t<s$
 with white $t$ and black $s$ determined in Lemmas $\ref{ex2j},$  $\ref{ex2k}).$
  \label{ex2u}
\end{lemma}
\begin{proof} 
Without loss of generality we may suppose that $k\leq m<\psi (k)$ due to Proposition \ref{dec}.
We have already mentioned that all {\bf U}$^{S}(k,m)$-roots belong to $\Sigma $
(see Corollary \ref{call}).

Certainly $[k:m]$ is an {\bf U}$^{S}(k,m)$-root, for 
$\Phi ^{S}(k,m)\in {\bf U}^{S}(k,m).$ 
Since $\psi (k-1)-1=\psi (k)>m,$ the point $\psi (k-1)-1$ does not appear on the scheme
of $\Phi ^{S}(k,m).$ If $S$ is black $(k,m)$-regular, then $\psi (m)-1$ is a black point,
see (\ref{grab2}). Hence 
Lemmas \ref{ex2j} and \ref{ex2k} show that in both cases  $[k:m]$ is indecomposable in $\Sigma .$
Thus $[k:m]$ is a simple {\bf U}$^{S}(k,m)$-root.

If $s$ is a black point, then $[1+s:m]\notin \Sigma $ (otherwise $[k:m]$ would be
decomposable in $\Sigma ).$ In particular $[1+s:m]$ is not a sum of  
{\bf U}$^{S}(k,m)$-roots. By Lemma \ref{su} the element $[k:s]$ is an 
{\bf U}$^{S}(k,m)$-root
(in particular Lemma \ref{sub1} implies that 
$S$ equals the minimal set $S^{\prime }$ such that 
$\Phi ^{S^{\prime }}(k,m)\in {\bf U}^{S}(k,m)).$
 If additionally $[k:s]$ is indecomposable in $\Sigma ,$
then it is a simple {\bf U}$^{S}(k,m)$-root.

If $t,s$ are, respectively,  white and black points, $k\leq  t<s,$
then by Lemma \ref{sub0}, we have 
$\Phi ^{S^{\prime \prime }}(k,s)\in {\bf U}^{S}(k,m)$
for a suitable (minimal) set ${S}^{\prime \prime }\subseteq S.$
Since $t$ is still a white point for $\Phi ^{S^{\prime \prime }}(k,s),$
the same lemma applied to $\Phi ^{S^{\prime \prime }}(k,s)$
implies $\Phi ^{S^{\prime \prime }}(1+t,s)\in {\bf U}^{S}(k,m).$

Let $\alpha $ be indecomposable in $\Sigma .$ Since by definition $\Sigma $ is an additive
monoid generated by elements of the form $[1+t:s]$ with white $t$ and black $s,$
all indecomposable elements have that form: $\alpha =[1+t:s].$
If, first, $[1+t:s]$ satisfies property b)
of Lemma \ref{ex2j} or Lemma \ref{ex2k}, then $n\notin [t,s].$ Hence $S^{\prime \prime }$
(as well as any other set) is white and black $(1+t,s)$-regular. By Corollary  \ref{dec1} we have
$\Phi ^{S^{\prime \prime }}(1+t,s)\neq 0,$ hence $[1+t:s]$ is an 
{\bf U}$^{S}(k,m)$-root. This is simple, for it is indecomposable in $\Sigma .$

If, next, $[1+t:s]$ satisfies property a) of Lemma \ref{ex2j} or Lemma \ref{ex2k},
then $[k:s]$  also satisfies this property; that is, $[k:s]$ is  indecomposable in $\Sigma .$
In particular $[k:t]\notin \Sigma ,$ and hence $[k:t]$ is not an {\bf U}$^{S}(k,m)$-root.
By Lemma \ref{su} applied to the simple {\bf U}$^{S}(k,m)$-root $[k:s]$
we see that $[1+j:s]$ is a sum of {\bf U}$^{S}(k,m)$-roots. Since $[1+j:s]$
is indecomposable in $\Sigma $ and all roots belong to $\Sigma ,$ the sum has just one 
summand; that is, $[1+j:s]$ is a simple {\bf U}$^{S}(k,m)$-root.

Conversely, if $\alpha $ is a simple {\bf U}$^{S}(k,m)$-root, then 
by Corollary \ref{call} we have $\alpha \in \Sigma .$ In particular $\alpha $
is a sum of indecomposable in $\Sigma $ elements. However we have already proved 
that each indecomposable in $\Sigma $ element is an {\bf U}$^{S}(k,m)$-root.
Thus the sum has just one summand; that is, $\alpha $ is indecomposable in $\Sigma .$
\end{proof}
\begin{theorem}
Let $S$ be a white $($black$)$ $(k,m)$-regular set. The right coideal subalgebra
{\bf U}$^{S}(k,m)$ coincides with the subalgebra ${\frak A}$ generated over {\bf k}$[G]$
by all elements $\Phi ^{S}(1+t,s),$ where $t<s$ are, respectively,
white and black points that
satisfy one of the conditions of Lemma $\ref{ex2j}\ ($Lemma $\ref{ex2k}).$ 
  \label{iex2}
\end{theorem}
\begin{proof} 
Of course we should show that $\Phi ^{S}(1+t,s)\in {\bf U}^{S}(k,m).$
To do this suppose firstly that $s<\psi (1+t).$ Denote by $S^{\prime }$ a minimal set
such that $\Phi ^{S^{\prime }}(1+t,s)\in {\bf U}^{S}(k,m),$ see Lemmas \ref{sub1}, \ref{ex2u}.

If $a\in {S}\cap [1+t, s-1],$
then by definition $[1+t:a]\in \Sigma .$ Hence $[1+t:a]$ is a sum of 
{\bf U}$^{S}(k,m)$-roots. Lemma \ref{sos1} applied to $[1+t:s]$ shows that
$[1+t:a]$ itself is an {\bf U}$^{S}(k,m)$-root (note that $a\neq \psi (1+t)$ since
$a<s<\psi (1+t)$). Thus Lemma \ref{sub1} applied to $[1+t:s]$
shows that $a\in {S}^{\prime };$ that is, $S \cap [1+t,s-1]\subseteq {S}^{\prime }.$

If $b\in {S}^{\prime },$ then by Lemma \ref{sub1} applied to $[1+t:s]$ the element $[1+t:b]$
is an {\bf U}$^{S}(k,m)$-root. In particular $[1+t:b]\in \Sigma .$ If $b\notin S,$ 
then by definition $[1+b:s]\in \Sigma ,$ and we get a contradiction  $[1+t:s]=[1+t:b]+[1+b:s].$
Thus $b\in S;$ that is, $S^{\prime }=S \cap [1+t,s-1],$
and  $\Phi ^{S}(1+t,s)=\Phi ^{S^{\prime }}(1+t,s)\in {\bf U}^{S}(k,m).$

If $s>\psi (1+t),$ then by Proposition \ref{dec} we have 
$\Phi ^{S}(1+t,s)\sim $ $\Phi ^{T}(\psi (s), \psi (1+t)).$ Since certainly
$\psi (1+t)<\psi (\psi (s)),$ we may apply already considered case:
$\Phi ^{T}(\psi (s), \psi (1+t))\in {\bf U}^{T}(\psi (m), \psi (k))={\bf U}^{S}(k,m).$

If $[a:b]$ is a non-simple {\bf U}$^{S}(k,m)$-root, then it has a decomposition
in sum of simple  roots of the form $[1+t:s].$ The element $c$ defined in each of formulae 
(\ref{tc1}-\ref{tc4})  belongs to the subalgebra ${\frak  A}$ generated by all $\Phi ^{S}(1+t,s).$
Hence {\bf U}$^{S}(k,m)$ has PBW-generators from ${\frak  A};$
that is,  {\bf U}$^{S}(k,m)={\frak A}.$
\end{proof}

The proved theorem allows one easily to find the root sequence for {\bf U}$^{S}(k, m)$
with regular $S$. Due to Corollary \ref{dec0} it suffices consider the case $k\leq m<\psi (k)$ only.

\begin{proposition}
Let $S$ be a white  $(k,m)$-regular set, $k\leq m<\psi (k).$ 
The root sequence $(\theta _i, 1\leq i\leq n)$
for {\bf U}$^{S}(k,m)$ has the following form in terms of the shifted scheme of 
${\Phi }^{S}(k,m):$
\begin{equation}
\theta _i=\left\{
\begin{matrix} 
0, \hfill & \hbox{ if } i-1 \hbox{ is not white};\hfill \cr
\psi (i)-a_i, \hfill & \hbox{ if } i-1 \hbox{ is white, } \psi (i) \hbox{ is black}; \hfill \cr
b_i-i+1, \hfill & \hbox{ if } i-1 \hbox{ is white, } \psi (i) \hbox{ is not black,}\hfill 
\end{matrix} 
\right. 
\label{erse}
\end{equation} 
where $a_i$ is the minimal number such that $(a_i,\psi (a_i)-1)$ is a white-white column,
while $b_i,$ $i\leq b_i<\psi (i),$ is the maximal black point, if any, otherwise $b_i=i-1\ ($hence 
$\theta _i=b_i-i+1=0).$  
\label{rse}
\end{proposition}
\begin{proof} 
An element $\alpha =[1+t:s]$ given in Lemma \ref{ex2u} defines a simple 
{\bf U}$^{S}(k,m)$-root starting with $i$ if either 
$i=1+t\ \& \ s<\psi (1+t)$ or $s=\psi (i) \ \& \ s>\psi (1+t).$

If $i-1$ is not a white point, then, of course, $i\neq 1+t,$ hence $s=\psi (i).$ 
The column $(s,i-1)=(\psi (i),i-1)$ is not a black-black one, for $S$ is white-regular.
Therefore it is incomplete; that is, $t=i-1$ does not appear in the scheme, a contradiction.
Thus there are no simple {\bf U}$^{S}(k,m)$-roots starting with $i$ at all, and
$\theta _i=0.$

Assume $i-1$ is white, while $\psi (i)$ is black. In this case $[1+n:\psi (i)]$ satisfies condition
a) of Lemma \ref{ex2j}. Hence $[i:n]=[\psi (n):\psi (i)]=[1+n:\psi(i)]$ 
 is a simple {\bf U}$^{S}(k,m)$-root starting with $i.$ In particular $\theta _i>n-i.$

If $i=1+t,$ $s<\psi (1+t),$ then $[1+t:s]$ does not satisfy condition a) of Lemma \ref{ex2j},
for $\psi (1+t)=\psi (i)$ is black. If it satisfies condition b), then the length of $[1+t:s]$
is less than $n-i.$

If $s=\psi (i),$ $s>\psi (1+t),$ then $[1+t:s]$ satisfies condition a) of Lemma \ref{ex2j} 
if and only if 
$(t, \psi (t+1))$ is a white-white column. In this case the length equals 
$s-(1+t)+1=\psi (i)-t.$ It has the maximal value if $t$ is minimal: $t=a_i.$

Assume $i-1$ is white, while $\psi (i)$ is not black. In this case $s\neq \psi (i).$
Hence $i=1+t,$ and $s$ is a black point such that $s<\psi (1+t)=\psi (i).$
The length of $[1+t:s]$ equals $s-t=s-i+1.$ It takes the maximal value if
$s$ is the maximal black point such that $i\leq s<\psi (i);$ that is $s=b_i.$
If all points in the interval $[i, \psi (i)-1]$ are white, then
there are no simple {\bf U}$^{S}(k,m)$-roots starting with $i$ at all. Hence still
$\theta _i=b_i-i+1=0.$
\end{proof}
\begin{proposition}
Let $S$ be a black $(k,m)$-regular set, $k\leq m<\psi (k).$ 
The root sequence $(\theta _i, 1\leq i\leq n)$
for {\bf U}$^{S}(k,m)$ has the following form in terms of the shifted scheme of
${\Phi }^{S}(k,m):$
\begin{equation}
\theta _i=\left\{
\begin{matrix} 
0, \hfill & \hbox{ if } i-1 \hbox{ is not white, } \psi (i) \hbox{ is not black}; \hfill \cr
\psi (i)-d_i, \hfill & \hbox{ if } i-1 \hbox{ is not white, } \psi (i) \hbox{ is black}; \hfill  \cr
\psi (i)-c_i, \hfill & \hbox{ if } i-1 \hbox{ is white, } \hfill 
\end{matrix} 
\right. 
\label{erseb}
\end{equation} 
where $c_i$ is the minimal number such that $(c_i,\psi (c_i)-1)$ is a black-black column,
while $d_i,$ $i\leq d_i<\psi (i),$ is the minimal white point, if any, otherwise $d_i=\psi (i)\ ($hence 
$\theta _i=\psi (i)-d_i=0).$  
\label{rseb}
\end{proposition}
\begin{proof} 
The proof follows from Lemma \ref{ex2k} in a quite similar way like the proof of the above 
proposition follows from Lemma \ref{ex2j}.
\end{proof}

\begin{example} \rm 
Consider a right coideal subalgebra $U(w)$ generated over ${\bf k}[G]$ by the element
$w=[[x_3, [x_3x_2x_1]],x_2]$ with $n=3$ (recall that the value of $[x_3x_2x_1]$ in 
$U_q^{+}({\mathfrak so}_{7})$ is independent of the precise alignment of brackets, see (\ref{jak3})).
By definition (\ref{ww}) we have $[x_3, [x_3, [x_2,x_1]]]=u[3,6],$ while Lemma \ref{xny}
 implies $[u[3,6],x_2]\sim \Phi ^{\{ 2\} }(2,6).$ Here $\{ 2\}$ is a white $(2,6)$-regular set, however
$6>\psi (2)=5.$ By Proposition \ref{dec} we have $\Phi ^{\{ 2\} }(2,6)\sim \Phi ^{\{ 1,2,3\} }(1,5).$
Since $5<\psi (1)=6$ and $\{ 1,2,3\}$ is a black $(1,5)$-regular set,
to find the root sequence for $U(w)={\bf U}^{\{ 1,2,3\} }(1,5),$ we may apply Proposition \ref{rseb}.
The shifted scheme 
\begin{equation}  
\begin{matrix}
 \ &\stackrel{5}{\bullet }  & \stackrel{4}{\circ }  &\stackrel{3}{\bullet } \Leftarrow \cr 
 \stackrel{0}{\circ } &\stackrel{1}{\bullet } &\stackrel{2}{\bullet }  & \stackrel{3}{\bullet } \hfill  
\end{matrix}
\label{ab2}
\end{equation}
shows that $c_1=1,$ $c_2=3,$ $c_3=3,$ while $d_1=4,$ $d_2=4,$ $d_3=\psi (3)=4.$ 
 If $i=1,$
then $i-1=0$ is a white point, and by the third option of (\ref{erseb}) 
we have $\theta _1=\psi (1)-c_1=5.$
If $i=2,$ then $i-1=1$ and $\psi (i)=5$ are black points. Hence the second option of (\ref{erseb}) 
applies: $\theta _2=\psi (2)-d_2=5-4=1.$ If $i=3,$ then $i-1=2$ is a black point,
while  $\psi (i)=4$ is a white point; that is, according to the first option of (\ref{erseb}) 
we have $\theta _3=0.$  Thus $\theta (U(w))=(5,1,0).$
\label{ex3}
\end{example}

\section{Construction}
Our next goal is a construction of a right coideal subalgebra with a given root sequence
\begin{equation}
{\bf \theta }=(\theta_1, \theta_2, \ldots ,\theta_n),\hbox{ such that } 0\leq \theta_k\leq 2n-2k+1,\ 
1\leq k\leq n.
\label{secu}
\end{equation}
We shall need the following auxiliary objects.
\begin{definition} \rm
By downward induction on $k$ we define sets $R_k\subseteq [k,\psi (k)-1 ],$ 
$T_k\subseteq [k,\psi (k)],$ $1\leq k\leq 2n$
associated to a given sequence (\ref{secu})
as follows. 

For $k>n$ we put $R_k=T_k=\emptyset .$

Suppose that 
$R_i,$ $T_i,$ $k<i\leq 2n$ are already defined. We denote by {\bf P} the following binary predicate
on the set of all ordered pairs $i\leq j$:
\begin{equation}
{\bf P}(i,j)\rightleftharpoons j\in T_{i}\vee \psi (i)\in T_{\psi (j)}.
\label{pred}
\end{equation}
Of course it is defined only on pairs $(i,j)$ such that $k<i\leq j<\psi (k)$ yet.
We note that {\bf P}$\, (i,j)={\bf P}(\psi (j),\psi (i)).$ Also it is useful to note that
for given $i,j$ one of the conditions $j\in T_{i}$ or $\psi (i)\in T_{\psi (j)}$ is false
due to $T_s\subseteq [s,\psi (s)],$ all $s,$ and $T_s=\emptyset ,$ $s>n$
with the only exception being $j=\psi (i)$ when these conditions coincide.
In particular, we see that if $j<\psi (i),$ then {\bf P}$\, (i,j)$ is equivalent to $j\in T_i.$

If $\theta _k=0,$ then we set 
$R_k=T_k=\emptyset .$ If $\theta _k\neq 0,$ then 
by definition $R_k$ contains 
$\tilde{\theta }_k=k+\theta _k-1$
and all $m$ satisfying the following three properties
\begin{equation}
\begin{matrix}\smallskip
a)\ k\leq m<\tilde{\theta }_k; \hfill \cr \smallskip
b)\ \neg \hbox{\bf P}\, (m+1, \tilde{\theta }_k); \hfill \cr
c) \  \forall r (k\leq r<m)\ \  \hbox{\bf P}\, (r+1, m)\Longleftrightarrow 
\hbox{\bf P}\, (r+1, \tilde{\theta }_k). \hfill
\end{matrix}
\label{pet1}
\end{equation}
Further, we define an auxiliary set 
\begin{equation}
T_k^{\prime }=R_k\cup \bigcup_{s\in R_k} \{ a\, |\, s<a<\psi (k), \ {\bf P}(s+1,a)\} ,
\label{pet2p}
\end{equation}
and finally,
\begin{equation}
T_k=\left\{ \begin{matrix} T_{k}^{\prime }, \hfill 
& \hbox{ if } \psi (R_k+1)\cap T_k^{\prime }=\emptyset ;\hfill \cr 
T_{k}^{\prime }\cup \{ \psi (k)\}, \hfill & \hbox{ otherwise.}\hfill 
\end{matrix}
\right. 
\label{pet2}
\end{equation}
\label{tski}
\end{definition}

For example, the first step of the construction is as follows.
If $\theta _n=0,$ certainly we have $R_n=T_n=\emptyset .$ Since by definition 
$\theta_n\leq 2n-2n+1=1,$ there exists just one additional option 
$\theta _n=1.$ In this case $\tilde{\theta }_n=n,$ and $R_n=\{ n\} ,$ 
while $T_n^{\prime }=R_n.$ We have $\psi (R_n+1)\cap T_n^{\prime }=\{ n\} \neq \emptyset .$
Hence $T_n=\{ n, \psi (n) \}=\{ n, n+1 \} .$ 

\begin{example} \rm
Assume $n=3,$ ${\bf \theta }=(5,1,0).$
Since $\theta_3=0,$ by definition $R_k=T_k=\emptyset ,$ $k\geq 3.$ 

Let $k=2.$ We have $\theta_2=1\neq 0,$ hence $\tilde{\theta }_2=2+\theta_2-1=2\in R_2.$
Certainly there are no points $m$ that satisfy $k=2\leq m<\tilde{\theta }_2=2;$
that is, $R_2=\{ 2\} .$ Eq. (\ref{pet2p}) yields 
$$
T_2^{\prime }=\{ 2\}\cup \bigcup_{s\in \{ 2\} } \{ a\, |\, 2=s<a<\psi (2)=5, \ {\bf P}(3,a)\}=\{ 2\} . 
$$
We have $\psi (R_2+1)\cap T_2^{\prime }=\{ 4\} \cap \{ 2\} =\emptyset ,$ hence $T_2=\{ 2\} .$

To find $R_1$ it is convenient to fix already known values of the predicate {\bf P}
in tabulated form.

\begin{center}
\begin{tabular}{*{6}{|c}|}
\hline
$i \, \backslash \, j$ & 2 & 3 & 4 & 5 \\
\hline 
$2 \hfill $ & $T \hfill $ & $  F \hfill $ & $ F \hfill $ & $ F $ \\
\hline 
$3 \hfill $ & $ \ \hfill $ & $ F \hfill $ & $ F\hfill $ & $ F \hfill $ \\
\hline 
$4 \hfill $ & $ \ \hfill $ & $ \ \hfill $ &  $ F \hfill $ & $ F \hfill $ \\
\hline
$5 \hfill $ & $ \ \hfill $ & $ \ \hfill $ & $ \ \hfill $ & $ T \hfill $ \\
\hline

\end{tabular}
\end{center}

\centerline{\bf Values of P}

\smallskip 
\smallskip 

 We have $\theta _1\neq 0;$ that is,  $\tilde{\theta }_1=1+5-1=5\in R_1.$

There exist four points $m$ that satisfy $k=1\leq m<\tilde{\theta }_1=5;$
they are $1,$ $2,$ $3,$ and $4.$ 
The point $m=4$ does not satisfy condition $b),$  for {\bf P}$\, (5,5)$ is true.
Hence $4\notin R_1.$ The points $m=1,2,$ and $3$ satisfy condition $b)$ since in
the last column of the tableaux there is just one value $``T";$ this corresponds
to $m+1=5.$

Let us check condition $c)$ for $m=1.$ The interval $1=k\leq r<m=1$ is empty.
Therefore the equivalence $c$) is true (elements from the empty set 
satisfy all conditions, even $r\neq r\, $). Thus $1\in R_1.$

Equivalence $c)$ in terms of the tableaux of the values of {\bf P}  
means that the column corresponding to
$j=m$ equals  a subcolumn corresponding to $j=\tilde{\theta }_1=5.$ This is indeed the case
for $m=3,$ but not for $m=2.$ Thus $R_1=\{ 1,3,5\} .$ 

To find $T_1^{\prime }$ we have to check just two remaining points: $a=2,4.$ 
From the tableaux we see that {\bf P}$\, (x,4)$  is always false, hence  $4\notin T^{\prime }_1 .$
At the same time {\bf P}$\, (s+1,2)$ is true for $s=1\in R_1.$ Hence $2\in T^{\prime }_1.$

Finally, $\psi (R_1+1)\cap T_1^{\prime }=\{ 5,3,1\} \cap \{ 1,2,3,5\} \neq \emptyset ,$
hence $T_1=\{ 1,2,3,5,6\} .$

Thus, for ${\bf \theta }=(5,1,0)$ we have $R_3=T_3=\emptyset ,$
$R_2=T_2=\{ 2\} ,$ $R_1=\{ 1,3, 5\} ,$ and $T_1=\{ 1,2,3,5,6\} .$
\label{rex1}
\end{example}
\begin{theorem} For each sequence 
${\bf \theta }=(\theta_1, \theta_2, \ldots ,\theta_n),$ such that 
$$0\leq \theta_k\leq 2n-2k+1,\ \ 1\leq k\leq n$$ 
there exists a homogeneous right coideal subalgebra
{\bf U}$\, \supseteq {\bf k}[G]$ with $r(\hbox{\bf U})={\bf \theta }.$ In what follows we shall
denote it by {\bf U}$_{\theta }.$
\label{su4}
\end{theorem}
\begin{proof} 
Denote by {\bf U} a subalgebra generated over {\bf k}$[G]$ by values in $U_q({\mathfrak so}_{2n+1})$
or in $u_q({\mathfrak so}_{2n+1})$ of the following elements
\begin{equation}
\Phi ^{S}(k,m),\ \ 1\leq k\leq m \hbox{ with } m\in R_k,\ S =T_k.
\label{pet3}
\end{equation}
(For example, if $\theta =(5,1,0),$ then the generators are
$x_1,$ $x_2,$ $[x_3x_2x_1],$ $[[x_3[x_3x_2x_1]],x_2] ).$
 We shall prove that {\bf U} is a right coideal subalgebra with $r({\bf U})=\theta .$ 
To attain these ends we need to check
some properties of $R_k, T_k,$ and {\bf P}.

\smallskip
\noindent
{\bf Claim 1.} {\bf P}$\, (k,m)$ {\it is true if and only if there exists a sequence 
\begin{equation}
k-1=k_0<k_1<\ldots <k_r<m=k_{r+1},
\label{tir}
\end{equation}
 such that for each $i,$ $0\leq i\leq r,$
either  $k_{i+1}\in R_{1+k_i}$ or} $\psi (1+k_i)\in R_{\psi (k_{i+1})}.$

\smallskip
We shall use induction on $m-k.$ If $m=k,$ then the condition $k\in T_k$ is equivalent to
$k\in T_k^{\prime },$ for $k\neq \psi (k).$ Eq. (\ref{pet2p}) implies, in turn, that  
$k\in T_k^{\prime }$ is equivalent to $k\in R_k.$ Thus {\bf P}$\, (k,k)$ is equivalent 
to $k\in R_k\vee \psi (k)\in R_{\psi (k)};$ that is, we have sequence (\ref{tir}) with $r=0.$

Assume, first, $m\in T_k.$
If $m\in R_k,$ we put $k_1=m+1,$ $r=1.$ 

If $m\notin R_k,$ $m\neq \psi (k),$ then by definition $m\in T_k^{\prime };$ that is, by
(\ref{pet2p}) there exists $s\in R_k,$ $s<m$ such that {\bf P}$\, (s+1,m)$
is true.
By the inductive supposition applied to $(s+1,m)$ there exists sequence (\ref{tir})
with $k_0=s.$ One may extend this sequence from the left by $k-1<k<s,$
for $s\in R_k.$

If $m=\psi (k),$ then by definition $\psi (s_1+1)\in T_k^{\prime }$ for a suitable $s_1\in R_k.$
Of course, we have that {\bf P}$\, (k,\psi (s_1+1))$ is true. Hence considered 
above case with $m\leftarrow \psi (s_1+1)$ yields sequence (\ref{tir}) with 
$k_{r+1}=\psi (s_1+1).$ We may extend this sequence from the right by $\psi (s_1+1)<\psi (k)=m$ 
since $s_1=\psi (1+\psi (s_1+1))\in R_{\psi (\psi (k))}=R_k.$ 

Assume, next, $\psi (k)\in T_{\psi (m)}.$ Since $\psi (k)-\psi (m)=m-k,$ 
we may apply considered above 
case with $k\leftarrow \psi (m),$ $m\leftarrow \psi (k).$ Hence there exists sequence (\ref{tir})
with $k_0=\psi (m)-1,$ $k_{r+1}=\psi (k).$ Let us denote $k_i^{\prime }=\psi (k_i)-1,$ $0\leq i\leq r+1.$
We have
\begin{equation}
k-1=k_{r+1}^{\prime }<k_r^{\prime }<\ldots <k_1^{\prime }<k_0^{\prime }=m.
\label{tir1}
\end{equation}
In this case 
$k_i^{\prime }\in R_{1+k_{i+1}^{\prime }}$ is equivalent to $\psi (1+k_i)\in R_{\psi (k_{i+1})},$
while  $\psi (1+k_{i+1}^{\prime })\in R_{\psi (k_{i}^{\prime })}$ is equivalent to 
$k_{i+1}\in R_{1+k_i}.$

\smallskip
Conversely, suppose that we have sequence (\ref{tir}). Without loss of generality
we may suppose that $m\leq \psi (k),$ otherwise we turn to (\ref{tir1}).
The inductive supposition implies that {\bf P}$\, (1+k_1,m)$ is true.  
Moreover $k_1\in R_k.$ Indeed, otherwise we have 
$\psi (k)\in R_{\psi (k_1)}\subseteq [\psi (k_1),k_1-1].$ In particular  $\psi (k)<k_1,$ and
hence $k>\psi (k_1).$ However $k_1\leq m\leq \psi (k)$ implies 
$\psi (k_1)\geq k.$
Now, if $m\neq \psi (k),$ then definition (\ref{pet2p}) with $s\leftarrow k_1,$ $a\leftarrow m$
implies $m\in T_k^{\prime }.$ 

Let $m=\psi (k).$ In this case considering as above sequence (\ref{tir1})
we have $k_r^{\prime }\in R_k.$ By definition $k_r^{\prime }=\psi (k_r)-1.$
Hence $k_r\in \psi (R_k+1).$ At the same time definition (\ref{pet2p}) shows that
$k_r\in T_k^{\prime }$ since the inductive supposition implies that {\bf P}$\, (1+k_1,k_r)$
is true provided that $r>1,$ while if $r=1,$ then $k_r=k_1\in R_k.$
Thus definition (\ref{pet2}) implies $m=\psi (k)\in T_k.$

\smallskip
\noindent
{\bf Claim 2.}  {\it If {\bf P}$\, (k,s)$ and  {\bf P}$\, (s+1,m),$ then {\bf P}$\, (k,m).$} 

\smallskip
\noindent
Indeed, one may extend from the right sequence (\ref{tir}) corresponding to the pair
$(k,s)$ by the sequence corresponding to $(s+1,m).$

\smallskip
\noindent
{\bf Claim 3.}  {\it If {\bf P}$\, (k,m),$ then for each $s,$ $k\leq s<m$ either {\bf P}$\, (k,s)$
or {\bf P}$\, (s+1,m).$} 

\smallskip
\noindent
We use induction on $m-k.$ Without loss of generality we may suppose that
$m\leq \psi (k),$ for {\bf P}$\, (k,m)$ is equivalent to {\bf P}$\, (\psi (m),\psi (k)).$
By Claim 1 there exists sequence (\ref{tir}) with $k_0=k-1,$ $k_{r+1}=m.$
The same claim implies {\bf P}$\, (1+k_1,m)$ provided that $r\geq 1.$

Since $k\leq s<m,$ there exists $i,$ $1\leq i\leq r,$ such that 
$k_i<s\leq k_{i+1}.$ If $i\geq 1,$ then the inductive supposition applied 
to $(1+k_1,m)$ implies that either {\bf P}$\, (1+k_1,s)$ or {\bf P}$\, (s+1,m).$
In the latter case  we have got the required condition.
If {\bf P}$\, (1+k_1,s)$ is true, then Claim 2 implies {\bf P}$\, (k,s),$ for  
{\bf P}$\, (k,k_1)$ is true according to Claim 1.

Thus, it remains to check the case $i=0;$ that is, $k\leq s\leq k_1.$
In this case $k_1\in R_k.$ Indeed, otherwise we have 
$\psi (k)\in R_{\psi (k_1)}\subseteq [\psi (k_1),k_1-1].$ In particular  $\psi (k)<k_1,$ and
hence $k>\psi (k_1).$ However $k_1\leq m\leq \psi (k)$ implies 
$\psi (k_1)\geq k.$

Claim 2 with $s\leftarrow 1+k_1,$ $k\leftarrow s+1$ says that conditions 
{\bf P}$\, (s+1,k_1)$ and {\bf P}$\, (1+k_1,m)$ imply {\bf P}$\, (s+1,m).$
Hence it is sufficient to show that either {\bf P}$\, (k,s)$ or {\bf P}$\, (s+1,k_1)$ is true.
If $s=k_1,$ then of course $s=k_1\in R_k$ yields {\bf P}$\, (k,s).$
This allows us to replace $m$ with $k_1$  and suppose further that $m\in R_k,$ $i=0.$
In this case condition (\ref{pet1}$c$) with $r\leftarrow s$ is
``{\bf P}$\, (s+1,m)\Longleftrightarrow {\bf P}\, (s+1,\tilde{\theta }_k).$"
Therefore we have to consider only one case $m=\tilde{\theta }_k.$

Let us suppose that for some $s,$ $k\leq s<\tilde{\theta }_k$ we have 
$\neg {\bf P}(k,s)$ and $\neg {\bf P}(s+1,\tilde{\theta }_k).$ By induction on $s,$
in addition to the induction on $m-k,$ we shall show
that these conditions are inconsistent 
(more precisely they imply $s\in R_k,$ which contradicts to 
$\neg {\bf P}(k,s);$ see definition (\ref{pet2p})).

Definition (\ref{pet1}) with $m=k$ shows that $k\in R_k$ if and only if 
$\neg {\bf P}(k+1,\tilde{\theta }_k).$
Since in our case $\neg {\bf P}(s+1,\tilde{\theta }_k),$
we have $s\in R_k,$ provided that $s=k.$

Let $s>k.$ Conditions (\ref{pet1}$a$) and (\ref{pet1}$b$) with $m\leftarrow s$ are valid. 
Suppose that (\ref{pet1}$c$) fails. In this case we may find a number $t,$ $k\leq t<s,$
such that $\neg ({\bf P}(t+1,s)\Longleftrightarrow  {\bf P}(t+1,\tilde{\theta }_k)).$

If ${\bf P}(t+1,s)$ but $\neg {\bf P}(t+1,\tilde{\theta }_k),$ then by the inductive supposition
(induction on $s$) either {\bf P}$\, (k,t)$ or {\bf P}$\, (t+1,\tilde{\theta }_k);$ that is,
{\bf P}$\, (k,t)$ is true. Claim 2 implies  {\bf P}$\, (k,s),$ a contradiction. 

If {\bf P}$\, (t+1,\tilde{\theta }_k),$
 but $\neg {\bf P}(t+1,s),$ then the inductive supposition
of the induction on $m-k$
with $k\leftarrow t+1,$ $m\leftarrow \tilde{\theta }_k$
shows that either {\bf P}$\, (t+1,s)$ or {\bf P}$\, (s+1,\tilde{\theta }_k);$
that is, {\bf P}$\, (s+1,\tilde{\theta }_k),$ again a contradiction.

Thus $s$ satisfies all conditions (\ref{pet1}a -- \ref{pet1}$c$), hence $s\in R_k.$

\smallskip
\noindent
{\bf Claim 4.} {\it If $k\leq m<\tilde{\theta }_k,$ then $m\in T_k$
if and only if $\neg {\bf P}(m+1,\tilde{\theta }_k).$}

\smallskip
\noindent
First of all recall that condition $m\in T_k$ is equivalent to {\bf P}$\, (k,m),$
for by definition $\tilde{\theta }_k<\psi (k).$

According to Claim 3 one of the conditions {\bf P}$\, (k,m)$ or 
{\bf P}$\, (m+1,\tilde{\theta }_k)$ always holds. If both of them are valid, then
due to Claim 1 we find sequence (\ref{tir}) with $k_0=k-1,$ $k_{r+1}=m,$
 such that $k_{i+1}\in R_{1+k_i}\vee \psi (1+k_i)\in R_{\psi (k_{i+1})},$ $0\leq i\leq r.$
Due to (\ref{pet1}$b$) we have $m\notin R_k,$ and of course  $\psi (k)\notin R_{\psi (m)},$
for $m\leq \tilde{\theta }_k<\psi (k).$ Hence $r>1.$ 

Again by the first claim  we get {\bf P}$\, (1+k_1,m).$
Since $k_1\leq m<\psi (k),$ we have $\psi (k)\notin R_{\psi (k_1)}.$ Hence $k_1\in R_k.$
Therefore $k_1$ satisfies condition (\ref{pet1}$b$), which is  $\neg {\bf P}(1+k_1,\tilde{\theta }_k).$
However Claim 2 shows that the conditions 
{\bf P}$\, (1+k_1,m)$ and {\bf P}$\, (m+1,\tilde{\theta }_k )$ imply {\bf P}$\, (1+k_1,\tilde{\theta }_k );$
a contradiction, that proves the claim.

\smallskip
\noindent
{\bf Claim 5.} {\it The set $T_k$ is $(k,m)$-regular for all $m\in R_k.$}

\smallskip
\noindent
We may suppose that $k\leq n<m$ since otherwise we have nothing to prove.
Assume, first, that $n$ is a white point; that is $n\notin T_k,$ while scheme (\ref{grab})
has a black column, say $n-i\in T_k,$ $n+i\in T_k,$ $i>0.$ Condition $n+i\in T_k$
implies {\bf P}$\, (k,n+i).$ Hence by Claim 3 with $m\leftarrow n+i,$ $s\leftarrow n$
we have {\bf P}$\, (k,n)\vee {\bf P}(n+1,n+i).$ However $n\notin T_k$ implies 
$\neg {\bf P}(k,n),$ for $\psi (k)\notin T_{\psi (n)}=T_{n+1}=\emptyset .$
Hence {\bf P}$\, (n+1,n+i)$ is true. We have 
{\bf P}$\, (n+1,n+i)$ $={\bf P}(\psi (n+i), \psi (n+1))$ $={\bf P}(n-i+1,n)$ is also true.
Since $n-i\in T_k$ implies {\bf P}$\, (k,n-i),$ Claim 2 with $s\leftarrow n-i,$
$m\leftarrow n$ shows that {\bf P}$\, (k,n)$ is true. This is a contradiction, for
$n\notin T_k$ implies $\neg {\bf P}(k,n).$

Let, then, $n$ be a black point; that is, $n\in T_k,$ while scheme (\ref{grab2})
have a white column, say $n-i\notin T_k,$ $n+i\notin T_k,$ $i>0.$
Condition $n-i\notin T_k$ implies $\neg {\bf P}(k,n-i),$ for $T_{\psi (n-i)}=T_{n+i+1}=\emptyset .$
By Claim 3 with $m\leftarrow n,$ $s\leftarrow n-i$ we have
${\bf P}(n-i+1,n),$ for $n\in T_k$ implies {\bf P}$\, (k,n).$ Hence
{\bf P}$\, (n-i+1,n)$ $={\bf P}(\psi (n), \psi (n-i+1))$ $={\bf P}(n+1, n+i)$ is true as well.
At the same time Claim 4 with $m\leftarrow n+i$
implies {\bf P}$\, (n+i+1,\tilde{\theta }_k),$ while Claim 2 with $k\leftarrow n+1,$
$s\leftarrow n+i$ implies {\bf P}$\, (n+1,\tilde{\theta }_k).$
Again Claim 4 with $m\leftarrow n$ shows that $n\notin T_k,$ a contradiction.

It remains to show, next, that if $n\in T_k,$ then the first from the left complete
column of (\ref{grab2}) is a black one; that is $\psi (m)-1\in T_k.$
Let to the contrary $\psi (m)-1\notin T_k.$ Then we have
$\neg {\bf P}(k,\psi (m)-1),$ for $T_{\psi (\psi (m)-1)}=T_{m+1}=\emptyset .$
Claim 3 with $s\leftarrow \psi (m)-1,$ $m\leftarrow n$ implies
{\bf P}$\, (\psi (m),n),$ while
Claim 4 with $m\leftarrow n$ implies $\neg {\bf P}(n+1, \tilde{\theta }_k).$
We see that the point $r=n<m$ does not satisfy condition (\ref{pet1}c),
 for {\bf P}$\, (n+1,m)$ $={\bf P}(\psi (n),m)$ $={\bf P}(\psi (m), n)$ is true,
while {\bf P}$\, (n+1, \tilde{\theta }_k)$ is false. Thus $m\notin R_k,$
a contradiction.

\smallskip
\noindent
{\bf Claim 6.} {\it Let  ${\bf \tilde{U}}$ be a subalgebra generated
by all right coideals {\bf U}$\, ^{T_k}(k,m),$
 $m\in R_k.$  If $1\leq a\leq b\leq 2n,$ $b\neq \psi (a),$
then {\bf P}$\, (a,b)$ is true if and only if $[a:b]$ is an ${\bf \tilde{U}}$-root. In particular the set of all
 ${\bf \tilde{U}}$-roots is} $\{ [k:m]\, |\, m\in T_k^{\prime }\} .$

\smallskip
\noindent
Certainly ${\bf \tilde{U}}$ is a right coideal subalgebra that contains {\bf k}$[G].$ By Theorem
\ref{iex2} it is generated over {\bf k}$[G]$ by elements $\Phi ^{T_k}(1+t,s),$
where $t<s$ are, respectively, white and black points for $\Phi ^{T_k}(k,m);$
that is, $t=k-1$ or $t\notin T_k,$ and $s=m$ or $s\in T_k.$ In particular {\bf P}$\, (k,s)$
is true, while {\bf P}$\, (k,t)$ is false ($\psi (k)\notin [t,\psi (t)]\supseteq T_{\psi (t)},$
for $k\leq t<s<\psi (k)$). Hence, by Claim 3 with $s\leftarrow t$ we have 
{\bf P}$\, (1+t,s).$ 

If $\gamma =[a:b],$ $a\leq b<\psi (a)$ is an ${\bf \tilde{U}}$-root, 
then, by definition, in ${\bf \tilde{U}}$ there exists
a homogeneous element $c_u\in {\bf \tilde{U}}$ of the form (\ref{vad22}) of degree $\gamma .$
Since ${\bf \tilde{U}}$ is generated by $\Phi ^{T_k}(1+t,s),$ the degree $\gamma $ is a sum 
of degrees $[1+t:s]$ of the generators. In particular
$\gamma =\sum _i[a:b],$ where {\bf P}$\, (a_i,b_i)$ are true, and $b_i\neq \psi (a_i).$
By Lemma \ref{sos} we may modify the decomposition of $\gamma $ so that
$$
\gamma =[k_0-1:k_1]+[1+k_1:k_2]+\ldots [1+k_r:k_{r+1}],
$$
where $a-1=k_0<k_1<\ldots <k_r<b=k_{r+1},$ and for each $i,$ $0\leq i\leq r,$
still {\bf P}$\, (1+k_i,k_{i+1})$ is true. Now Claim 2 implies  {\bf P}$\, (a,b).$
Hence $b\in T_a^{\prime },$ for $a\leq b<\psi (a).$

Conversely, if $m\in T_k^{\prime },$ then by Claim 1 we have a sequence
$
k-1=k_0<k_1<\ldots <k_r<m=k_{r+1},
$
 such that for each $i,$ $0\leq i\leq r,$
either  $k_{i+1}\in R_{1+k_i}$ or $\psi (1+k_i)\in R_{\psi (k_{i+1})}.$
By definition ${\bf \tilde{U}}$ contains elements
$\Phi ^{T_{a_i}}(a_i,b_i),$ where $a_i=1+k_i,$ $b_i=k_{i+1}$ provided that
$k_{i+1}<\psi (1+k_i),$ and $a_i=\psi (k_{i+1}),$ $b_i=\psi (1+k_i)$ provided that
 $k_{i+1}>\psi (1+k_i).$ Hence $[a_i:b_i]=[1+k_i:k_{i+1}]$ are ${\bf \tilde{U}}$-roots.
In particular $[k:m]$ is a sum of ${\bf \tilde{U}}$-roots.
By Lemma \ref{sos1} the element $[k:m]$ itself is an ${\bf \tilde{U}}$-root.

\smallskip
\noindent
{\bf Claim 7.} {\it  The set of all simple ${\bf \tilde{U}}$-roots is $\{ [k:m]\, |\, m\in R_k\}.$ In particular}
$r({\bf \tilde{U}})=\theta .$ 

\smallskip
\noindent
If $\gamma =[k:m],$ $k\leq m<\psi (k)$ is a simple ${\bf \tilde{U}}$-root, 
then due to the above claim {\bf P}$\, (k,m)$ is true.
Hence, according to Claim 1, we may find sequence (\ref{tir}).
In this case $\gamma =[k:k_1]+[1+k_1:k_2]+\ldots +[1+k_{r}:m]$ is a sum of ${\bf \tilde{U}}$-roots,
for {\bf P}$\, (1+k_i,k_{i+1})$ is true by definition (\ref{pred}).
Since $\gamma $ is simple this is possible only if $r=0.$ Thus,  $m=k_1 \in R_k,$
for $\psi (k)\notin [m,\psi (m)]\supseteq R_{\psi (m)}.$

Conversely. Let $m\in R_k.$ Then by  definition (\ref{pet2}) we have $m\in T_k.$
Claim 6 implies that $[k:m]$ is an ${\bf \tilde{U}}$-root.
If it is not simple, then it is a sum of two or more  ${\bf \tilde{U}}$-roots, 
$[k:m]=[k:k_1]+[1+k_1:k_2]+\ldots [1+k_r:m],$ where, 
due to Claim 6, we have {\bf P}$\, (1+k_i,k_{i+1}),$ $0\leq i\leq r$ are true.
Claim 2 implies that {\bf P}$\, (1+k_1,m)$ is true as well. Definition (\ref{pet1}c) with 
$r\leftarrow k_1$ implies {\bf P}$\, (1+k_1,\tilde{\theta }_k).$ Now Claim 4
provides a contradiction, $k_1\notin T_k$ (recall that {\bf P}$\, (k,k_1)$ implies $k_1\in T_k$
since $k\leq k_1\leq m<\psi (k)$).

\smallskip
\noindent
{\bf Claim 8.} ${\bf \tilde{U}}$ {\it as an algebra is generated by {\bf k}$[G]$ and 
$\Phi ^{T_k}(k,m),$ $m\in R_k;$ that is,} ${\bf \tilde{U}}={\bf U}.$

\smallskip
\noindent
It suffices to note that {\bf U} contains a set of PBW-generators for ${\bf \tilde{U}}$
over {\bf k}$[G].$
If $[k:m]$ is an ${\bf \tilde{U}}$-root, then it is a sum of simple ${\bf \tilde{U}}$-roots,
$[k:m]=\sum [k_i:m_i],$ $m_i\in R_{k_i}.$ Elements $f_i=\Phi ^{T_{k_i}}(k_i,m_i)$
by definition belong to {\bf U}.
The PBW-generator corresponding to the root $[k:m]$ can be taken to be
a polynomials in $f_i$ determined in one
of the formulae $(\ref{tc1}-\ref{tc4})$ depending up of the type of the decomposition of 
$[k:m]$ in a sum of simple roots.

Theorem \ref{su4} is completely proved. \end{proof}

\begin{corollary} Every $($homogeneous if $q^t=1,$ $t>4$ $)$ right coideal subalgebra {\bf U} of 
$U^{+}_q(\mathfrak{so}_{2n+1}),$ $q^t\neq 1$  $($respectively of $u^{+}_q(\mathfrak{so}_{2n+1}))$
that contains $G$ is generated as an algebra by $G$ and a set of elements $\Phi ^{S}(k,m)$
with $(k,m)$-regular sets $S$.
\label{rug}
\end{corollary}
\begin{proof}
Theorems \ref{teor} and \ref{su4} imply that {\bf U} has the form {\bf U}$_{\theta },$
where $\theta $ is the root sequence. At the same time definition (\ref{pet3})
shows that {\bf U}$_{\theta }$ as an algebra is generated by $G$ and elements  
$\Phi ^{T_k}(k,m),$ $m\in R_k.$ It remains to apply Claim 5.
\end{proof}

\section{Right coideal subalgebras that do not contain the coradical}

In this small section we restate the main result in a slightly more general form. More precisely
we consider homogeneous right coideal subalgebras
in  $U^{+}_q(\mathfrak{so}_{2n+1}),$ (respectively in $u^{+}_q(\mathfrak{so}_{2n+1})$)
whose intersection with the coradical is a subgroup. 
We are reminded  that for every submonoid 
$\Omega \subseteq G$ the set of all linear combinations {\bf k}$\, [\Omega]$
is a right coideal subalgebra.
Conversely if $U_0\subseteq \,${\bf k}$\, [G]$ is a right coideal subalgebra, then
$U_0=\, ${\bf k}$\, [\Omega]$ for $\Omega =U_0\cap G$
since  $a=\sum_i \alpha_i g_i\in U_0$ implies $\Delta (a)=
\sum_i \alpha_i g_i\otimes g_i\in U_0\otimes \, ${\bf k}$\, [G];$ that is, $\alpha_i g_i\in U_0.$
\begin{definition} \rm
For a sequence $\theta=(\theta_1, \theta_2, \ldots ,\theta_n),$ such that $0\leq \theta_k\leq 2n-2k+1,$
$1\leq k\leq n$ we denote by {\bf U}$^1_{\theta }$ a subalgebra with 1 generated by 
$g^{-1}_{km}\Phi ^{S}(k,m),$ where $g_{km}=g(u(k,m))$ and 
$m\in R_k,$ $S =T_k;$ see Theorem \ref{su4}.
\label{1te}
\end{definition}
\begin{lemma} The subalgebra
{\bf U}$^1_{\theta }$ is a homogeneous right coideal, and 
{\bf U}$^1_{\theta }\cap G=\{ 1\} .$ 
\label{1su}
\end{lemma}
\begin{proof}
The subalgebra
{\bf U}$^1_{\theta }$ is homogeneous since it is generated by homogeneous elements.
Its zero homogeneous  component equals {\bf k} since among the generators just one,
the unity, has zero degree. 

Denote by $B_{\theta }$ a {\bf k}-subalgebra generated by 
 $\Phi ^{S}(k,m)$ with $m\in R_k,$ $S=T_k.$
The algebra {\bf U}$^1_{\theta }$ is spanned by all elements of the form 
$g_a^{-1}a,$ $a\in B_{\theta }.$ Since {\bf U}$_{\theta }$ is a right coideal, 
for any homogeneous $a\in B_{\theta }$ we have
$\Delta (a)=\sum g(a^{(2)})a^{(1)}\otimes a^{(2)}$ where $a^{(1)}\in B_{\theta },$
$g_a=g(a^{(1)})g(a^{(2)}).$ Therefore 
$\Delta (g_a^{-1}a)=\sum g(a^{(1)})^{-1}a^{(1)}\otimes g_a^{-1}a^{(2)}$ with 
$g(a^{(1)})^{-1}a^{(1)}\in \, ${\bf U}$_{\theta }^1.$
\end{proof}

\begin{lemma} 
If $\Omega $ is a submonoid of $G,$ then ${\bf k }[\Omega ]{\bf U}_{\theta }^{1}$
is a homogeneous right coideal subalgebra,  
and ${\bf k }[\Omega ]{\bf U}_{\theta }^{1}\cap G=\Omega .$ Moreover 
${\bf k }[\Omega ]{\bf U}_{\theta }^{1}$ 
$={\bf k }[{\Omega }^{\prime } ]{\bf U}_{{\theta }^{\prime }}^{1}$
if and only if ${\Omega }={\Omega }^{\prime }$ and ${\theta }={\theta }^{\prime }.$
\label{1su1}
\end{lemma}
\begin{proof}
The subalgebra ${\bf k }[\Omega ]{\bf U}_{\theta }^{1}$ is homogeneous since it is generated by homogeneous elements.
Its zero homogeneous  component equals ${\bf k}[\Omega ].$ Hence 
${\bf k }[\Omega ]{\bf U}_{\theta }^{1}\cap G=\Omega .$ 
By the above lemma we have 
$$
\Delta ({\bf k }[\Omega ]{\bf U}_{\theta }^{1})\subseteq 
({\bf k }[\Omega ]\otimes {\bf k }[\Omega ]) \cdot ({\bf U}_{\theta }^{1}\otimes 
U^{+}_q(\mathfrak{so}_{2n+1})).
$$
Hence ${\bf k }[\Omega ]{\bf U}_{\theta }^{1}$ is a right coideal subalgebra.
Finally, the equality  ${\bf k }[\Omega ]{\bf U}_{\theta }^{1}$ 
$={\bf k }[{\Omega }^{\prime } ]{\bf U}_{{\theta }^{\prime }}^{1}$ implies both the
equality of zero homogeneous components, ${\bf k }[{\Omega }]={\bf k }[{\Omega }^{\prime }],$
and ${\bf U}_{\theta }={\bf k }[G]{\bf U}_{\theta }^{1}$ 
$={\bf k }[G]{\bf U}_{{\theta }^{\prime }}^{1}={\bf U}_{{\theta }^{\prime }}.$ Hence 
$\theta ={\theta }^{\prime}$ due to Theorem \ref{su4}.
\end{proof}
\begin{theorem} If $U$ is a homogeneous right coideal subalgebra   
of  $U^{+}_q(\mathfrak{so}_{2n+1})\,  ($resp. of $u^{+}_q(\mathfrak{so}_{2n+1}))$
such that $\Omega \stackrel{df}{=} U\cap G$ is a group,
then $U={\bf k }[\Omega ]{\bf U}_{\theta }^{1}$
for a suitable $ \theta .$
\label{orc}
\end{theorem}
\begin{proof}
Let $u=\sum h_ia_i\in U$ be a homogeneous element of degree $\gamma \in \Gamma ^{+}$
with different $h_i\in G,$ and $a_i\in A,$ where by $A$ we denote the {\bf k}-subalgebra
generated by $x_i,$ $1\leq i\leq n.$ Denote by ${\pi }_{\gamma }$ the natural
projection on the homogeneous component of degree $\gamma .$
Respectively $\pi _g,$ $g\in G$ is a projection on the subspace {\bf k}$g.$
We have $\Delta (u)\cdot (\pi _{\gamma }\otimes \pi _{h_i})=h_ia_i\otimes h_i.$
Thus $h_ia_i\in U.$

By Theorem \ref{su4} and Theorem \ref{teor}
 we have ${\bf k}[G]U={\bf U}_{\theta }$ for a suitable $\theta .$
If $u=ha\in U,$ $h\in G,$ $a\in A,$ then $\Delta (u)\cdot (\pi _{hg_a}\otimes \pi _{\gamma })=
hg_a\otimes ha.$ Therefore $hg_a\in U\cap G=\Omega ;$ that is, $u=\omega g_a^{-1}a,$
$\omega \in \Omega .$ Since $\Omega $ is a subgroup we get $g_a^{-1}a\in U.$
It remains to note that all elements $g_a^{-1}a,$ such that $ha\in U$ span the algebra 
{\bf U}$_{\theta }^1.$
\end{proof}

If $U\cap G$ is not a group, then $U$ may have a more complicated structure, see 
\cite[Example 6.4]{KL}.

\end{document}